\documentclass[a4paper,11pt]{article}
\usepackage{amssymb,amsmath,amsthm}
\usepackage{caption}

\numberwithin{equation}{section}

\newcommand{\m}[1]{\mathbb{ #1}}
\newcommand{\mc}[1]{\mathcal{ #1}}

\newfont{\goth}{eufm10 at 12pt}
\newfont{\gots}{eufm8 at 9pt}
\newcommand{\rmd}{{\,\rm d}}

     \def\ol{\overline}    
   
\def\al{\alpha}       \def\be{\beta}        \def\ga{\gamma}
\def\de{\delta}       \def\eps{\varepsilon}  
\def\th{\theta}       
       \def\la{\lambda}      
\def\si{\sigma}                
\def\ph{\varphi}               \def\ps{\psi}
       \def\Ga{\Gamma}       
\def\La{\Lambda}             \def\Ph{\Phi}

\def\espace{\vspace{0.8em}}

\newtheorem{Thm}{Theorem}[section]
\newtheorem{Prop}[Thm]{Proposition}
\newtheorem{Lem}[Thm]{Lemma}
\newtheorem{Cor}[Thm]{Corollary}
\newtheorem{Fact}[Thm]{Fact}
\newtheorem{Def}[Thm]{Definition}
\newtheorem{Ex}[Thm]{Example}
\theoremstyle{definition}
\newtheorem{Rem}[Thm]{Remark}
\def\bt{\begin{Thm}}
\def\et{\end{Thm}}
\def\br{\begin{Rem}}
\def\er{\end{Rem}}
\def\bc{\begin{Cor}}
\def\ec{\end{Cor}}
\def\bp{\begin{Prop}}
\def\ep{\end{Prop}}
\def\bl{\begin{Lem}}
\def\el{\end{Lem}}
\def\bd{\begin{Def}}
\def\ed{\end{Def}}
\def\bq{\begin{quotation}}
\def\eq{\end{quotation}}
\def\bfa{\begin{Fact}}
\def\efa{\end{Fact}}
\def\bex{\begin{Ex}}
\def\eex{\end{Ex}}

\def\ra{\rightarrow}
\def\vs{\vspace{1em}}

\title{Harmonic quasi-isometric  maps II~:\\  
negatively curved manifolds}
\author{Yves Benoist \&
Dominique Hulin}
\date{}

\begin{document}
\maketitle
\begin{abstract}
We prove that a quasi-isometric map, and more generally a coarse embedding, 
between pinched Hadamard manifolds 
is within bounded distance from a unique harmonic map. 
\end{abstract}

\setcounter{tocdepth}{2}

\renewcommand{\thefootnote}{\fnsymbol{footnote}} 
\renewcommand{\thefootnote}{\arabic{footnote}} 

\section{Introduction}  
\label{secintro}

The aim of this article, which is a sequel to \cite{BH15}, is the following theorem.
\bt
\label{thdfxhx}
Let $f:X\ra Y$ be a quasi-isometric map between two pinched Hadamard manifolds.
Then there exists a unique harmonic map $h:X\ra Y$ which stays within bounded distance from $f$
i.e. such that
$$\textstyle
\sup_{x\in X}d(h(x),f(x))<\infty\; .
$$
\et

We first recall  a few definitions. 
A pinched Hadamard manifold $X$ is a
complete simply-connected Riemannian manifold of dimension at least $ 2$ whose sectional curvature is pinched between 
two negative constants:
$-b^2\leq K_X\leq -a^2<0$.
A map $f:X\ra Y$ between two metric spaces $X$ and $Y$ 
is said to be {\it quasi-isometric} if there exist
constants $c\geq 1$ and $C\geq 0$ such that 
$f$ is {\it $(c,C)$-quasi-isometric}. This means that
one has 
\begin{equation}
\label{eqnquasiiso}
c^{-1}\, d(x,x')-C
\;\leq\; 
d(f(x),f(x'))
\;\leq\; 
c\, d(x,x')+C\; 
\end{equation}
for all $x$, $x'$ in $X$. 
A $\mc C^2$ map $h:X\ra Y$ between two Riemannian manifolds $X$ and $Y$ is said to be 
{\it harmonic} if  it satisfies the elliptic nonlinear partial differential equation
${\rm tr}(D^2h)=0$ where $D^2h$ is the second covariant derivative of $h$.

Partial results towards the existence statement 
were obtained in \cite{Pansu89}, \cite{TamWan95}, \cite{HardtWolf97},
\cite{Markovic02}, \cite{BonSch10}.
A major breakthrough was achieved by Markovic who solved the Schoen conjecture, 
i.e. the case where $X=Y$ is the hyperbolic plane $\m H^2_\m R$, and by  Lemm--Markovic who proved 
the existence for the case $X=Y=\m H^k_\m R$ in \cite{Marko3}, 
\cite{Marko2} and \cite{Markon}.
The existence when both $X$ and $Y$ are rank one symmetric spaces,
which was conjectured by Li and Wang in \cite[Introduction]{LiWang98},
was proved in our paper  \cite{BH15}.
We refer to \cite[Section 1.2]{BH15} for more motivations and   a precise historical perspective on this result.

Partial results towards the uniqueness statement were obtained by Li and Tam in \cite{LiTam93},
and by Li and Wang in \cite{LiWang98}.  
All these papers  were dealing with rank one symmetric spaces.

Note that Theorem \ref{thdfxhx} was conjectured
by Markovic 
at the end of the conference talk www.youtube.com/watch?v=A5Yt83I1FrY,
during a 2016 Summer School in Grenoble. 
According to our knowledge, Theorem \ref{thdfxhx} is new even in the case 
where both $X$ and $Y$ are assumed to be surfaces. 

The strategy of the proof of the existence follows the lines of the proof in \cite{BH15}.
As in \cite{BH15}, we replace the quasi-isometric map $f$ by a  
$\mc C^\infty$-map whose first two covariant derivatives are bounded.
But we need  to modify the barycenter argument we used in \cite{BH15}
for this smoothing step. 
See Subsection \ref{secroulip}  for more details on 
this step.
As in \cite{BH15}, we then introduce the harmonic maps $h_R$ that coincide with $f$
on a sphere of $X$ with large radius $R$ and 
we need a uniform bound for the distances between 
the maps $h_{_R}$ and $f$. The heart of our argument is in Chapter \ref{secharmonicmap} 
which contains the boundary estimates, and in Chapter \ref{secinterior} 
which contains the interior estimates, for $d(h_R,f)$.
The proof of these interior estimates 
is based on a new simplification of  an idea by Markovic in \cite{Marko2}.
Indeed we will introduce a point $x$ where $d(h_{_R}(x),f(x))$ is maximal and focus on a subset $U_{\ell_0}$ of a 
sphere $S(x,\ell_0)$ whose definition 
\eqref{eqnur} is much simpler than in \cite{Marko2} or \cite{BH15}.
This simplification is the key point which allows us to extend the arguments of \cite{BH15} to pinched Hadamard manifolds.
In this proof  we use a uniform control 
on the harmonic measures on all the spheres of $X$,
which is given in Proposition \ref{prosixrcxthe}.
We refer to Section \ref{secstrexi}  for more details on our strategy of proof of the existence.

In order to prove the uniqueness, 
we need to introduce Gromov-Hausdorff limits  of the pointed metric spaces $X$ and $Y$
with respect to base points going to infinity and  therefore to deal with 
$\mc C^{2}$-Riemannian manifolds with  $\mc C^1$-metrics.
This will be done in Chapter \ref{secunihar}. 
We refer to Section \ref{secstruni} for more details on our strategy of proof of the uniqueness.

In Chapter \ref{secweacoahar}, we extend Theorem \ref{thdfxhx}
to coarse embeddings (see Definition \ref{defcoaemb} and Theorem \ref{thmharcoa}). 
The proof is similar but relies on the existence
of a boundary map for coarse embeddings. We also show that 
Theorem \ref{thdfxhx} can not be extended to Lipschitz maps (Example \ref{exaharlip}) 

Chapter \ref{secbouweacoa} is dedicated 
to  the existence of this boundary map which, for a coarse embedding,
 is well-defined outside a set of zero Hausdorff dimension (Theorem \ref{thmboumap}).
The existence of such a boundary map seems to be new.

We thank the MSRI for its hospitality during the Fall 2016 
where this project was developed.  
We are also very grateful to  A. Ancona, U. Hamenstadt, 
M. Kapovich and F. Ledrappier for sharing their insight with us.

\section{Smoothing}
\label{sechadamani}

In this chapter, we recall a few basic facts on Hadamard manifolds,
and we explain how to replace our quasi-isometric map $f$ by a 
$\mc C^\infty$ map whose first two covariant derivatives are bounded.

\subsection{The geometry of Hadamard manifolds}  
\label{secgeotriangle}
\bq
We first recall basic estimates on Hadamard manifolds for triangles, for images of triangles under  quasi-isometric maps, and for the Hessian of the distance function.
\eq

All the Riemannian manifolds will be assumed to be connected. We will denote by 
$d$ their distance function. 

A Hadamard manifold  is a complete simply connected Riemannian manifold $X$ of dimension $k\geq 2$
whose curvature is non positive $K_X\leq 0$.  
For instance, the Euclidean space $\m R^k$ is a Hadamard manifold with zero curvature $K_X=0$,
and the real hyperbolic space $\m H^k_{\m R}$  is 
a Hadamard manifold with constant  curvature $K_X=-1$.
We will  say that  $X$ is pinched if there exist constants 
$a,b>0$ such that 
$$-b^2\leq K_X\leq -a^2<0.$$
For instance, the non-compact rank one symmetric spaces are pinched Hada\-mard manifolds.
 
Let  $x_0$, $x_1$, $x_2$ be three points on a Hadamard manifold $X$.
The {\it Gromov product} of the points $x_1$ and $x_2$ seen from $x_0$ is defined as
\begin{equation}
(x_1|x_2)_{x_0}:= (d(x_0,x_1)+d(x_0,x_2)-d(x_1,x_2))/2.
\end{equation}

We recall the basic comparison lemma which is one of the motivations for introducing the Gromov product.

\bl
\label{lemcomparison}
Let $X$ be a Hadamard manifold with
$-b^2\leq K_X\leq -a^2< 0$.
Let $T$ be a geodesic triangle in $X$ with vertices $x_0$, $x_1$, $x_2$,
and let $\theta_0$ be the angle of $T$ at the vertex $x_0$.\\
$a)$ One has\;\; $(x_0|x_2)_{x_1}\geq d(x_0,x_1)\,\sin^2(\theta_0/2)$.\\
$b)$ One has \;\;
$\theta_0\leq 4\, e^{-a\,(x_1|x_2)_{x_0}} .$\\
$c)$ Moreover, if $\min( (x_0|x_1)_{x_2},(x_0|x_2)_{x_1})\geq b^{-1}$, 
one has 
$\theta_0\geq e^{-b(x_1|x_2)_{x_0}}$.
\el

\begin{proof} This is classical. See for instance \cite[Lemma 2.1]{BH15}.
\end{proof}

We now recall the effect of a quasi-isometric map on the Gromov product.

\bl
\label{lemproduct}
Let $X$, $Y$ be Hadamard manifolds with
$-b^2\leq K_X\leq -a^2< 0$ and $-b^2\leq K_Y\leq -a^2< 0$,
and let $f:X\ra Y$ be a $(c, C)$-quasi-isometric map.
There exists $A=A(a,b,c,C)>0$ such that, for all 
$x_0$, $x_1$, $x_2$ in $X$, one has
\begin{equation}
\label{eqnproduct}
c^{-1}(x_1|x_2)_{x_0}-A
\;\leq \;(
f(x_1)|f(x_2))_{f(x_0)}
\;\leq \;
c(x_1|x_2)_{x_0}+A\; .
\end{equation}
\el

\begin{proof}
This is a general property of quasi-isometric maps between Gromov $\de$-hyperbolic spaces which is due to M. Burger. See \cite[Prop. 5.15]{GhysHarp90}.
\end{proof}

When $x_0$ is a point in a Riemannian manifold $X$, we denote by $d_{x_0}$ the distance function 
defined by $d_{x_0}(x)=d(x_0,x)$ for $x$ in $X$.
We denote  by $d^2_{x_0}$ the square of this function.
When $F:X\rightarrow \m R$ is a $\mc C^2$ function, 
we denote by $DF$ its differential 
and by $D^2F$ its second covariant derivative.

\bl
\label{lemhessiandist}
Let $X$ be a Hadamard manifold and $x_0\in X$.\\
Assume that  $-b^2\leq K_X\leq -a^2\leq 0$. 
The Hessian of the distance function $d_{x_0}$ satisfies on $X\setminus\{x_0\}$
\begin{equation}
\label{eqnd2rhox0}
a\coth (a\, d_{x_0})\, g_0\leq D^2 d_{x_0} \leq 
b \coth (b\, d_{x_0})\, g_0\, ,
\end{equation}
where $g_0:=g_X -Dd_{x_0}\otimes Dd_{x_0}$ and $g_{_X}$ is the Riemannian metric on $X$.
\el

When $a=0$ the left-hand side of \eqref{eqnd2rhox0} must be interpreted 
as $d_{x_0}^{-1}g_0$.

\begin{proof} This is classical. See for instance \cite[Lemma 2.3]{BH15}
\end{proof}

\subsection{Smoothing rough Lipschitz maps}  
\label{secsmoothquasi}

\bq
The following proposition will allow us to assume 
in Theorem \ref{thdfxhx} that the quasi-isometric map $f$ we start with is $\mc C^\infty$
with bounded derivative and bounded second covariant derivative.
\eq

\subsubsection{Rough Lipschitz maps}
\label{secroulip}
A map $f:X\ra Y$ between two metric spaces $X$ and $Y$ 
is said to be {\it rough Lipschitz} if there exist
constants $c\geq 1$ and $C\geq 0$ such that, 
for all $x$, $x'$ in $X$,
one has
\begin{equation}
\label{eqnroulip}
d(f(x),f(x'))
\;\leq\; 
c\, d(x,x')+C\; .
\end{equation}

\bp
\label{prosmoothquasi}
Let $X$, $Y$ be two  Hadamard manifolds with bounded curvatures
$-b^2\leq K_X\leq 0$ and $-b^2\leq K_Y\leq 0$. Let 
$f:X\ra Y$ be a rough Lipschitz map.
Then there exists a $\mc C^\infty$  map $\widetilde{f}:X\ra Y$ 
within bounded distance from $f$ 
and  whose  first two covariant derivatives $D\widetilde{f}$ and $D^2\widetilde{f}$
are bounded on $X$.
\ep

We denote $k=\dim X$ and $k'=\dim Y$.
We will first construct in \ref{seclipcon} 
a regularized  map $\widetilde{f}:X\ra Y$ 
which is Lipschitz continuous.
This first construction is the same as for rank one symmetric spaces 
in \cite[Proposition 2.4]{BH15}. This construction will not allow us to control the
second covariant derivative, hence we will have 
to combine this first construction with an iterative smoothing 
process in local charts 
that we will explain in \ref{secsecder}.

\subsubsection{Lipschitz continuity}
\label{seclipcon}
The first part of the proof of Proposition \ref{prosmoothquasi}
relies on the following lemma.

\bl
\label{lemsmoothquasi}
Let $Y$ be a  Hadamard manifold. \\
$a)$ Let $\mu$ be a positive finite Borel  measure 
on $Y$ supported by a ball $B(y_0,R)$. 
The function $Q_{\mu}$ on $Y$
defined by
$$\textstyle
Q_\mu(y)\; =\; \int_Yd(y,w)^2\,\rmd \mu(w)\, 
$$
has a unique minimum $y_\mu$ in $Y$ 
called the center of mass of $\mu$.
This center of mass $y_\mu$ belongs to the ball $B(y_0,R)$.\\
$b)$ Let $\mu_1$, $\mu_2$ be two positive finite Borel measures on $Y$. 
Assume that\\ 
$(i)$ $\mu_1(Y)\geq m$ and $\mu_2(Y)\geq m$ for some $m>0$,\\
$(ii)$ both $\mu_1$ and $\mu_2$ are supported on the same ball $B(y_0,R)$,\\
$(iii)$ the norm of  $\|\mu_1-\mu_2\|\leq \eps$.\\
Then, the distance between their centers of mass $y_{\mu_1}$ and $y_{\mu_2}$ is bounded by 
\begin{equation}
\label{eqndymu12}
d(y_{\mu_1},y_{\mu_2})\leq 4\eps R/m\, .
\end{equation}
\el

\begin{proof}[Proof of Lemma \ref{lemsmoothquasi}]
$a)$
Since the space $Y$ is a proper space, i.e.  its balls are compact, 
the function  $Q_{\mu}$ is proper and admits a minimum $y_\mu$.
Since $Y$ has non-positive curvature  the median inequality holds, namely one has
for all $y$, $y_1$, $y_2$, $y_3$ in $Y$ where $y_3$ is the midpoint of $y_1$ and $y_2$~:
\begin{equation}
\label{eqnmed}
\tfrac12 d(y_1,y_2)^2\leq d(y,y_1)^2+d(y,y_2)^2-2d(y,y_3)^2.
\end{equation}
Integrating \eqref{eqnmed} with respect to $\mu$, one checks that the function  $Q_{\mu}$ satisfies the following uniform convexity property~: if $y_3$ is the midpoint of $y_1$ and $y_2$
\begin{equation*}
\label{eqnqsc}
\tfrac{m}{2} d(y_1,y_2)^2\leq Q_{\mu}(y_1)+Q_{\mu}(y_2)-2Q_{\mu}(y_3)\, .
\end{equation*}
Applying this inequality with $y_1=y_\mu$ and $y_2=y$,
one gets for each $y$ in $Y$
\begin{equation}
\label{eqnqmuymu}
\tfrac{m}{2} d(y_\mu,y)^2\leq Q_{\mu}(y)-Q_{\mu}(y_\mu)\, ,
\end{equation}
so that $y_\mu$ is the unique minimum of $Q_\mu$.

We now check that $y_\mu\in B(y_0,R)$.
By  the median inequality \eqref{eqnmed}, the ball $B(y_0,R)$ is convex,
every point $y$ in $Y$ admits a unique nearest point $y'$ on $B(y_0,R)$,
and this point $y'$ also satisfies the inequality 
$$
d(y',w)\leq d(y,w)
\;\;\mbox{\rm for all $w$ in $B(y_0,R)$.}
$$
Therefore, one has $Q_\mu(y')\leq Q_\mu(y)$. This proves that 
the center of mass $y_\mu$ belongs to 
the ball $B(y_0,R)$.

$b)$ Applying twice Inequality \eqref{eqnqmuymu}, one gets the two inequalities
\begin{eqnarray*}
\tfrac{m}{2} d(y_{\mu_1},y_{\mu_2})^2
&\leq & 
Q_{\mu_1}(y_{\mu_2})-Q_{\mu_1}(y_{\mu_1})\, ,
\;\;\;\;{\rm }\\
\tfrac{m}{2} d(y_{\mu_1},y_{\mu_2})^2
&\leq & 
Q_{\mu_2}(y_{\mu_1})-Q_{\mu_2}(y_{\mu_2})\, .
\end{eqnarray*}
Summing these two inequalities yields
\begin{eqnarray*}
m\,d(y_{\mu_1},y_{\mu_2})^2
&\leq &
(Q_{\mu_1}-Q_{\mu_2})(y_{\mu_2})-(Q_{\mu_1}-Q_{\mu_2})(y_{\mu_1})\\
&\leq &
\eps \sup_{w\in B(y_0,R)}|d(y_{\mu_1},w)^2-d(y_{\mu_2},w)^2|\\
&\leq &
4\eps R  \, d(y_{\mu_1},y_{\mu_2})\, ,
\end{eqnarray*}
which proves \eqref{eqndymu12}. 
\end{proof}

We now choose a non-negative $\mc C^\infty$ function $\chi:\m R\ra \m R$ with support 
included in $]-1,1[$, which is equal to $1$ on a neighborhood of $[-\frac12,\frac12]$ and whose first derivative is bounded by
$|\chi'|\leq 4$.

\begin{proof}[Proof of Proposition \ref{prosmoothquasi}] {\it First step:
Lipschitz continuity.}
We now explain this first construction.
We can assume $b=1$. Since a rough Lipschitz map $f:X\rightarrow Y$ is always  within bounded distance from a Borel measurable map, we can assume that $f$ itself is Borel measurable.
For $x$ in $X$, we introduce the positive finite measure $\mu_x$ on $Y$ such that  the equality
$$
\mu_x(\ph)=\int_X\ph(f(z))\,\chi(d(x,z))\rmd {\rm vol}_X(z)\, 
$$
holds for any positive function $\ph$ on $Y$. 
The measure $\mu_x$ is the image by $f$ of a measure supported in the ball $B(x,1)$.
We define $\widetilde{f}(x)\in Y$ to be the center of mass of this measure $\mu_x$.
Lemma \ref{lemsmoothquasi}.a tells us that the map $x\ra \widetilde{f}(x)$ is 
well-defined.  The Lipschitz continuity of $\widetilde{f}$
will follow from Lemma \ref{lemsmoothquasi}.b applied to two measures 
$\mu_1:=\mu_{x_1}$ and $\mu_2:=\mu_{x_2}$ with $x_1$, $x_2$ in $X$. Let us check that
the three assumptions in Lemma \ref{lemsmoothquasi}.b are satisfied.

$(i)$ Because of the pinching of the curvature on $X$, 
the Bishop volume estimates tell us that
there exist positive constants $0<m_0<M_0$ such that one has, for all $x$~:

$$
m_0\leq {\rm vol}(B(x,\tfrac12))\leq \mu_x(Y)\leq {\rm vol}(B(x,1))\leq M_0\, .
$$

$(ii)$ When $x_1$, $x_2$ are two points of $X$ with 
$d(x_1,x_2)\leq 1$, the bound \eqref{eqnroulip} ensures that 
both $\mu_{x_1}$ and $\mu_{x_2}$ are supported on
the ball $B(f(x_1),2c+C)$.   

$(iii)$ The norm of the difference of these measures satisfies
\begin{eqnarray*}
\|\mu_{x_1} -\mu_{x_2}\|
&\leq& 
M_0 \sup_{z\in X}|\chi(d(x_1,z))-\chi(d(x_2,z))|\\
&\leq &
4M_0\, d(x_1,x_2)\, .
\end{eqnarray*}

Thus Lemma \ref{lemsmoothquasi} applies and yields 
a bound on the
Lipschitz constant  of $\widetilde{f}$, namely

\mbox{ }\hfill
$\displaystyle
{\rm Lip}(\widetilde{f})
:=\sup_{x_1\neq x_2}
\frac{d(\widetilde{f}(x_1),\widetilde{f}(x_2))}{d(x_1,x_2)}\leq \frac{16(2c+C)M_0}{m_0}.
$
\end{proof}

\subsubsection{Bound on the second derivative}
\label{secsecder}
The second step of the proof  of Proposition \ref{prosmoothquasi} relies on three lemmas. 
The first lemma provides
a nice system of charts on $Y$.

\bl
\label{lemquacoo} 
Let $Y$ be a  Hadamard manifold  with 
$-b^2\leq K_Y\leq 0$ and $k'=\dim Y$. There exist constants $r_0=r_0(k',b)>0$ and 
$c_0=c_0(k',b)> 1$ 
such that, for each $y$ in $Y$, there exists a 
$\mc C^\infty$ chart $\Ph_y$ for the open ball
\begin{equation}
\label{eqnphxbxr}
\Ph_y: \mathring{B}(y,r_0)\stackrel{\sim}{\longrightarrow} U_{y}\subset \m R^{k'}
\;\; {\rm with}\;\; 
\Ph_y(y)=0\; 
\end{equation}
and such that 
\begin{equation}
\label{eqndphdph}
\|D\Ph_y\|\leq c_0
\;\; ,\;\;
\|D\Ph_y^{-1}\|\leq c_0
\;\; ,\;\;
\|D^2\Ph_y\|\leq c_0
\;\; ,\;\;
\|D^2\Ph_y^{-1}\|\leq c_0
\; .
\end{equation}
In particular, one has for all $r< r_0$
\begin{equation}
\label{eqnphbbux}
\Ph_y(B(y,c_0^{-1}r))\subset B(0,r)
\;\;{\rm and}\;\; 
B(0,c_0^{-1}r)\subset \Ph_y(B(y,r)).
\end{equation}
\el

We have endowed $\m R^{k'}$ with the standard Euclidean structure.

\begin{proof}[Proof of Lemma \ref{lemquacoo}]
This is classical.
One can for instance choose the so-called almost linear coordinates,
as in \cite[Section 2]{Jost84} or \cite[Section 3]{Petersen94}.
They are defined in the following way.
We fix an orthonormal basis $(e_i)_{1\leq i\leq k'}$ for the tangent space $T_yY$
and introduce the points  $y_i:=exp_y(-e_i)$ in $Y$.
The map $\Ph_y$ is defined by the formula
$$\Ph_y(z)=(d(z,y_1)\!-\! 1,\ldots ,d(z,y_{k'})\!-\! 1),
$$
where $z$ belongs to a sufficiently small ball $\mathring{B}(y,r_0)$.
See \cite[p. 43 and 58]{Jost84} for a detailed proof.
\end{proof}

There exist better systems of coordinates, the so-called harmonic coordinates.
We will not need them in this chapter, but we will need them in 
Chapter \ref{secunihar} to prove uniqueness (see Lemma \ref{lemharcoo}).
\vs

The second lemma explains how to modify a Lipschitz map $g$ 
inside a tiny ball $B(x,r)$ of $X$
so that the new map $g_{x,r}$ is constant on the ball
$B(x,\frac{r}{2})$ and  the first two derivatives of $g_{x,r}$ are controlled by those of $g$.
We recall that $\chi:\m R\ra \m R$ is a non-negative $\mc C^\infty$ function with support 
included in $]-1,1[$, which is equal to $1$ on a neighborhood of $[-\frac12,\frac12]$ and which is $4$-Lipschitz, i.e.
$|\chi'|\leq 4$.

\bl
\label{lemggt}
Let $X$ and $Y$ be two  Hadamard manifolds  with bounded curvatures
$-b^2\leq K_X\leq 0$, $-b^2\leq K_Y\leq 0$. Let $r_0>0$ and $c_0\geq 1$  
be as in Lemma \ref{lemquacoo}.
Let  $g:X\ra Y$ be a Lipschitz map, $x$ be a point in $X$, $y=g(x)$
and let $0<r<r_0$. Assume that 
\begin{equation}
\label{eqnlipgr0}
{\rm Lip}(g)< \frac{r_0}{c_0^2r}.
\end{equation}
Then the following formulas define a Lipschitz map  $g_{r,x}:X\ra Y$
\begin{eqnarray*}
\label{eqnghat}
\nonumber
g_{r,x}(z) & =& g(x)
\hspace{28ex} {\rm when}\;\; d(z,x)\leq \tfrac{r}{2}\, ,\\
&=&\Ph_y^{-1}\left((1\! -\!\chi(\tfrac{d(z,x)}{r}))\,\Ph_y(g(z))\right)
\;\;\; {\rm when}\;\; \tfrac{r}{2}\leq d(z,x)\leq r\, ,\\
\nonumber
&=& g(z)
\hspace{28ex} {\rm when}\;\; d(z,x)\geq r\, .
\end{eqnarray*}
One has the inequality 
between the Lipschitz constants on the ball $B(x,r)$~:
\begin{equation}
\label{eqnlipbxr}
{\rm Lip}_{B(x,r)}(g_{r,x})\leq 5{c_0}^2\,{\rm Lip}_{B(x,r)}(g).
\end{equation}
In particular, one has 
\begin{equation}
\label{eqnlipgrx}
{\rm Lip}(g_{r,x})\leq 5{c_0}^2\,{\rm Lip}(g).
\end{equation}

Moreover, if $g$ is $\mc C^2$ in a neighborhood of a point $z$ in $X$,
then $g_{r,x}$ is also $\mc C^2$ in this neighborhood
and one has 
\begin{equation}
\label{eqnd2gxg}
\| D^2g_{r,x}(z)\|\leq 
\left(\| D^2g(z)\|+{\rm Lip}_{B(x,r)}(g)^2+1\right)M_r\, ,
\end{equation}
where the constant $M_r\geq 1$ depends only on r, b, $k$, $k'$ and $\chi$.
\el

\begin{proof}[Proof of Lemma \ref{lemggt}]
Condition \eqref{eqnlipgr0} ensures that, 
for any point $z$ in the ball $B(x,r)$, the image $g(z)$ belongs to the ball 
$\mathring{B}(y, c_0^{-2}r_0)$. Therefore, by \eqref{eqnphbbux},
the vector $\Ph_y(g(z))$ belongs to the ball 
$\mathring{B}(0, c_0^{-1}r_0)\subset \m R^{k'}$.
When we  multiply this vector by the scalar $1\! -\!\chi(.)$, the new vector
is still in the same ball. This is why, using again \eqref{eqnphbbux},
the element $g_{r,x}(z)$ is well-defined and belongs to $B(y,r_0)$.

The upper bound \eqref{eqnlipbxr} follows from the chain rule.
Indeed, when $z$ is a point in $B(x,r)$ where $g$ is differentiable,
 the bound \eqref{eqndphdph} yields
\begin{eqnarray*}
\|Dg_{r,x}(z)\| &\leq &
c_0 \left( \frac{4}{r}\|\Phi_y(g(z))\|+ \| D(\Phi_y\circ g )(z)\| \right) \\
&\leq & 5c_0 \,{\rm Lip}_{B(x,r)}(\Phi_y\circ g)
\;\leq\; 
5{c_0}^2\,{\rm Lip}_{B(x,r)}(g).
\end{eqnarray*}

The upper bound \eqref{eqnd2gxg} follows from similar and longer computations
left to the reader, which also use the bounds \eqref{eqnd2rhox0} 
for $D^2d_x$.
\end{proof}

We will also need a third lemma. We recall that a subset $X_0$ of 
a metric space $X$ is said to be $r$-separated
if the distance between two distinct points of $X_0$ is at least $r$.

\bl
\label{lemsepsep} 
Let $X$ be a  Hadamard manifold  with 
$-b^2\leq K_X\leq 0$. Let $k=\dim X$ and $N_0:=100^{k}$.
There exists a radius $r_0=r_0(k,b)>0$ such that, 
for any $r<r_0$,
every $\frac{r}{2}$-separated subset $X_0$ of $X$
can be decomposed as a union of at most $N_0$ subsets which are $2r$-separated.  
\el

\begin{proof}[Proof of Lemma \ref{lemsepsep}]
The bound on the curvature of $X$ 
and the Bishop volume estimates ensure that 
we can choose $r_0>0$ such that
\begin{equation}
\label{eqnvolvol}
{\rm vol} B(x,4r)
\leq
N_0\,{\rm vol} B(x,\tfrac{r}{4})
\;\;\;
\mbox{\rm for 
all $r<r_0$ and $x$ in $X$.}
\end{equation}
This $r_0$ works. Indeed, let $X_1, X_2, \ldots$, $X_{N_0}$ be a sequence
of disjoint $2r$-separated subsets of $X_0$
with $X_1$ maximal in $X_0$, $X_2$ maximal in $X_0\!\smallsetminus\! X_1$, and so on. Every point $x$ of $X_0$ must be in one of the $X_i$'s with $i\leq N_0$ because, if it is not the case, 
each $X_i$  contains a point
in $B(x,2r)$, contradicting \eqref{eqnvolvol}.   
\end{proof}

\begin{proof}[Proof of Proposition \ref{prosmoothquasi}] {\it Second step:
bound on $D^2\widetilde{f}$.}
According to the first step of this proof, 
we can now assume that the map $f:X\ra Y$ is $c$-Lipschitz with $c\geq 1$.
\\ We can choose a new radius $r_0=r_0(k,k',b)$ that satisfies 
both  conclusions of Lemma \ref{lemsepsep} for $X$ and of 
Lemma \ref{lemquacoo} for  $Y$.
We will use freely the notations of these two lemmas.
Now let 
$$
r_1=\frac{r_0}{5^{N_0}\,c_0^{2N_0+2}\, c}
$$ 
and  pick
a maximal $\frac{r_1}{4}$-separated subset $X_0$
of $X$. 
Thanks to Lemma \ref{lemsepsep}, we  write this set $X_0$ as a union 
$$
X_0= X_1\cup\cdots \cup X_{N_0}
$$
of $N_0$ subsets $X_i$ which are $2r_1$-separated. 

In order to construct $\widetilde{f}$ from $f$, 
we will use a finite iterative process based on  Lemma \ref{lemggt}.
Starting with $f_0=f$, we construct by induction
a finite sequence of maps $f_i$ for $i\leq N_0$ and we set 
$\widetilde{f}:=f_{N_0}$.
Using the notations of Lemma \ref{lemggt},  the map
$f_{i}$ is defined from $f_{i-1}$ by letting
\begin{eqnarray*}
\label{eqnfifi}
f_{i}(z) & =& (f_{i-1})_{r_1,x}(z)
\hspace{5ex} {\rm when}\;\; d(z,x)\leq r_1
\;\;\mbox{\rm for some $x$ in $X_{i+1}$},\\
&=&f_{i-1}(z)
\hspace{8ex} {\rm otherwise}
\end{eqnarray*}
so that the Lipschitz constants of these maps satisfy
\begin{equation}
\label{eqnlipfi}
{\rm Lip}(f_{i}) \leq 5c_0^2\, {\rm Lip}(f_{i-1})\leq 5^i{c_0}^{2i}c.
\end{equation}
Indeed, once $f_{i}$ is known to be well defined and to satisfy \eqref{eqnlipfi},
it also satisfies  the bound \eqref{eqnlipgr0}~: 
$\;
{\rm Lip}(f_{i}) < \tfrac{r_0}{c_0^2\, r_1}.
$
Therefore Lemma \ref{lemggt} ensures that  $f_{i+1}$ is well defined 
and, using  \eqref{eqnlipbxr}, that $f_{i+1}$  also satisfies  \eqref{eqnlipfi}~:
\begin{equation*}
\label{eqnlipfi2}
{\rm Lip}(f_{i+1})\leq 5c_0^2\, {\rm Lip}(f_i) \leq 5^{i+1}{c_0}^{2(i+1)}c.
\end{equation*}
Let $\Lambda:=M_{r_1}+25c_0^4+1$. By \eqref{eqnd2gxg} 
and \eqref{eqnlipfi}, one has for any $i\leq N_0$ and $z$ in $X$~:
\begin{equation}
\label{eqnd2fiz}
\| D^2f_i(z)\|+{\rm Lip}(f_i)^2+1\leq 
\Lambda \left(\| D^2f_{i-1}(z)\|+{\rm Lip}(f_{i-1})^2+1\right)\, .
\end{equation}
Since $X_0$ is a maximal $\frac{r_1}{4}$-separated subset
of $X$, every $z$ in $X$ belongs to at least one ball 
$\mathring{B}(x,\tfrac{r}{2})$ where $x$ is in one of the sets $X_{i_0}$.
But then the function $f_{i_0}$ is constant in a neighborhood of $z$. 
Therefore, using \eqref{eqnlipfi} and applying 
$(N_0\!-\! i_0)$ times the bound \eqref{eqnd2fiz}
one deduces that $\widetilde{f}$ is a $\mc C^2$-map 
that satisfies the uniform upper bound
\vs

\mbox{}\hfill
$
\|D^2\widetilde{f}(z)\|\leq ((5^{i_0}c_0^{2i_0}c)^2+1)\Lambda^{N_0-i_0}
\leq \La^{N_0}c^2.
$
\end{proof}

\section{Harmonic maps}
\label{secharmonicmap}

In this chapter we begin the proof 
of the existence part in Theorem \ref{thdfxhx}. 
We first recall basic facts satisfied by harmonic maps.
We explain why a standard compactness argument reduces 
this existence part
to proving 
a uniform upper bound on the distance between $f$
and the harmonic map $h_{_R}$ which is equal to $f$ on the sphere $S(O,R)$.
Then we provide this upper bound 
near this sphere $S(O,R)$.

\subsection{Harmonic functions and the distance function}  
\label{secboundedlap}

\bq
We recall basic facts on the Laplace operator on Hadamard manifolds.
\eq

The Laplace-Beltrami operator $\Delta$ on a Riemannian manifold $X$
is defined as the trace of the Hessian.
In local coordinates, the Laplacian of a function $\ph$ is
\begin{equation}
\label{eqnlaplacian}
\Delta \ph={\rm tr}(D^2\ph)=\tfrac{1}{v}\textstyle\sum_{i,j}
\tfrac{\partial}{\partial x_i}(v\, g_{_X}^{ij}
\tfrac{\partial}{\partial x_j}\ph)
\end{equation}
where $v=\sqrt{\det (g_{_Xij})}$ is the volume density.
The function $\ph$ is said to be harmonic if $\Delta \ph=0$ and
subharmonic if $\Delta \ph\geq 0$. 

We will  need the following basic lemma.

\bl
\label{lemboundedlap}
Let $X$ be a Hadamard manifold with
$K_X\leq -a^2\leq 0$ and $x_0$ be a point in  $X$. 
Then, the function $d_{x_0}$ is subharmonic. More precisely,
the distribution $\Delta\, d_{x_0}\! -\! a$ is  non-negative.
\el

\begin{proof} 
This is  \cite[Lemma 2.5]{BH15}.
\end{proof}

\subsection{Harmonic maps and the distance function}  
\label{secharmonicdist}

\bq
In this section, we recall two useful facts satisfied by a harmonic map $h$~:
the subharmonicity of the functions $d_{y_0}\circ h$, and 
Cheng's estimate for the differential $ Dh$.
\eq

\begin{Def}
\label{defharmonic}
Let $h:X\ra Y$ be a $\mc C^2$ map between two Riemannian manifolds.
The tension field of $h$ is the trace of the second covariant derivative $\tau(h):= {\rm tr} D^2h$.
The map $h$ is said to be {\it harmonic} if $\tau(h)=0$.  
\end{Def}

Note that the tension field  $\tau(h)$ is a {\it $Y$-valued vector field on $X$}, i.e. 
it is a section of the pulled-back of the tangent bundle $TY\ra Y$ under the map 
$h:X\ra Y$.

For instance, an isometric immersion with minimal image is always harmonic.
The problem of the existence, regularity  and uniqueness of harmonic maps under various boundary conditions
is a very classical topic (see \cite{EeLem78}, \cite{Schoen77}, \cite{Jost84},
\cite{donn94}, \cite{Simon96}, \cite{SchoenYau97} or \cite{LinWang08}).  
In particular, when $Y$ is simply connected and has non positive curvature, a harmonic map is always $\mc C^\infty$ i.e. it is indefinitely differentiable,
and is a minimum of the energy functional  
among maps that agree with $h$ outside a compact subset
of $X$.

\bl
\label{lemdefh}
Let $h:X\ra Y$ be a harmonic $\mc C^\infty$ map between Riemannian manifolds.
Let $y_0\in Y$ 
and let $\rho_h:X\ra \m R$ be the function 
$\rho_h:=d_{y_0}\circ h$. 
If $Y$ is Hadamard, the continuous function $\rho_h$ is subharmonic on $X$. 
\el

\begin{proof} See \cite[Lemma 3.2]{BH15}.
\end{proof}

Another crucial property of harmonic maps is the following bound 
for their differential due to Cheng.

\bl
\label{lemcheng}
Let $X$, $Y$ be two Hadamard manifolds with $-b^2\leq K_X\leq 0$. Let $k=\dim X$, 
$z$ be a point in  $X$,  $r>0$
and let $h:B(z,r)\ra Y$ be a harmonic $\mc C^\infty$ map 
such that 
the  image 
$ h(B(z,r))$ lies in a ball of radius $R_0$. 
Then one has the bound
\begin{equation*}
\label{eqncheng}
\|Dh(z)\|\leq 2^5\,k\,\tfrac{1+br}{r}\, R_0\; .
\end{equation*}
\el
In the applications, we will use this inequality with $r=b^{-1}$.

\begin{proof}
This is a simplified version of \cite[Formula 2.9]{Cheng80}. 
\end{proof}

\subsection{Existence of harmonic maps}  
\label{secexistharm}

\bq
In this section we prove Theorem \ref{thdfxhx},
taking for granted Proposition \ref{proexistharm} below. 
\eq

Let $X$ and $Y$ be two Hadamard manifolds 
whose curvatures are pinched $-b^2\leq K_X\leq -a^2<0$ and 
$-b^2\leq K_Y\leq -a^2<0$. Let $k=\dim X$ and $k'=\dim Y$. 
Let 
$f:X\ra Y$ be a $(c,C)$-quasi-\-isometric $\mc C^\infty$ map whose first two covariant derivatives are bounded.

We fix a point $O$ in $X$. 
For $R>0$, we denote by $B_{_R}:=B(O,R)$ 
the closed ball in $X$ with center $O$ and radius $R$
and by $\partial B_{_R}$ the sphere that bounds $B_{_R}$. 
Since the manifold $Y$ is a Hadamard manifold, there exists
a unique harmonic map $h_{_R}:B_{_R}\ra Y$ satisfying the Dirichlet condition
$h_{_R}=f$ on the sphere
$\partial B_{_R}$. 
Thanks to Schoen and Uhlenbeck in \cite{SchoenUhl82} and \cite{SchoenUhl83},
the harmonic map $h_{_R}$ is known to be $\mc C^\infty$ on the closed ball $B_{_R}$.
We denote by 
$$
d(h_{_R},f)=\sup_{x\in B(O,R)}d(h_{_R}(x),f(x))
$$
the distance between these two  maps.
The main step for proving existence in Theorem \ref{thdfxhx} is the following uniform estimate.

\bp
\label{proexistharm}
There exists a constant $\rho\geq 1$ such that, for any $R\geq 1$, one has
$d(h_{_R},f)\leq \rho$.
\ep

The constant $\rho$ is a function of 
$a$, $b$, $c$, $C$, $k$ and $k'$. 
More precisely, when $f$ satisfies \eqref{eqnquasiiso3}, $\rho$ only needs to satisfy
Conditions \eqref{eqnrho1}, \eqref{eqnrho2} and 
\eqref{eqnrho3}.

We briefly recall the classical argument 
used to deduce Theorem \ref{thdfxhx} 
from this Proposition.

\begin{proof}[Proof of Theorem \ref{thdfxhx}]
As explained in Proposition \ref{prosmoothquasi}, 
we may assume 
that the $(c,C)$-quasi-isometric map $f$ is $\mc C^\infty$ with bounded first two covariant derivatives.  
Pick an unbounded increasing sequence of radii $R_n$ 
and let $h_{R_n}:B_{R_n}\ra Y$ be the harmonic $\mc C^\infty$ map 
that agrees with $f$ on the sphere $\partial B_{R_n}$. 
Proposition \ref{proexistharm}
ensures that
the sequence of  maps $(h_{R_n})$ is locally uniformly bounded.
Using the Cheng Lemma \ref{lemcheng} 
it follows that the first derivatives are 
also locally uniformly bounded.
The Ascoli-Arzela theorem implies that, after extracting a subsequence,
the sequence $(h_{R_n})$ converges  uniformly on every ball $B_{_S}$
towards a continuous map $h:X\ra Y$.
Using the Schauder's estimates, one also gets 
a uniform bound for the $\mc C^{2,\al}$-norms of $h_{R_n}$ on $B_{_S}$.
If needed, note that these classical estimates will be recalled in  
Formulas \eqref{eqnharmonic3}, \eqref{eqnschac1} and 
\eqref{eqnschac2} in  Section \ref{secconhar}.
Therefore, using the Ascoli-Arzela theorem again, the sequence $(h_{R_n})$
converges in the $\mc C^2$-norm and the limit map $h$ is $\mc C^2$ and harmonic.
By construction, this limit harmonic map $h$ 
stays within bounded distance from the  
quasi-isometric map $f$.
\end{proof}

\begin{Rem}
{\rm By the uniqueness part of our Theorem 
\ref{thdfxhx} that we will prove in Chapter \ref{secunihar}, 
the harmonic map $h$
which stays within bounded distance from $f$ is unique.
Hence the above argument also proves that 
the whole family of harmonic maps $h_R$ converges to $h$
uniformly on the compact subsets of $X$
when $R$ goes to infinity.}  
\end{Rem}

\subsection{Boundary estimate}
\label{secboundary}

\bq
In this section, we begin the proof of 
Proposition \ref{proexistharm}~: 
we  bound the distance between  $h_{_R}$ and $f$ near the
sphere  $\partial B_{_R}$.
\eq

\bp
\label{proboundary}
Let $X$, $Y$ be Hadamard manifolds and  $k=\dim X$. 
Assume moreover that $ K_X\leq -a^2<0$ and $-b^2\leq K_Y\leq 0$. 
Let $c\geq 1$
and $f:X\ra Y$ be a $\mc C^\infty$ map with $\|Df(x)\|\leq c$ and 
$\|D^2f(x)\|\leq bc^2$. 
Let  $O\in X$, $R>0$ and set $B_{_R}:=B(O,R)$. 

Let $h_{_R}:B_{_R}\ra Y$
be the harmonic $\mc C^\infty$ map  whose restriction to the sphere
$\partial B_{_R}$ is equal to $f$.
Then,  one has for every $x$ in $B_{_R}$~:
\begin{equation}
\label{eqnboundary}
d(h_{_R}(x),f(x))\leq \tfrac{3kbc^2}{a}\, d(x,\partial B_{_R})\; .
\end{equation}
\ep

An important feature of this upper bound is that 
it does not depend on the radius $R$,
provided the distance  $d(x,\partial B_{_R})$ remains bounded.
This is why we call \eqref{eqnboundary} the {\it boundary estimate}.
The proof relies on an idea of Jost in \cite[Section 4]{Jost84}.

\begin{proof} This proposition is already in 
\cite[Proposition 3.8]{BH15}. We give here a slightly shorter proof.
Let $x$ be a point in $B_{_R}$ and 
$y$ be a point in  $Y$ chosen so that $d(y,f(B_{_R}))\geq b^{-1}$ 
and 
\begin{equation}
\label{eqndfxhrx}
 d_y(h_{_R}(x))-d_y(f(x))=
 d(f(x),h_{_R}(x))\; .
\end{equation}
This point $y$ is far away on the geodesic ray starting at $h_{_R}(x)$
and containing $f(x)$.
Let $\ph$ be  the $\mc C^\infty$ function
on the ball $B_{_R}$ defined by 
\begin{equation}
\label{eqnfz1}
\ph(z):= d_y(h_{_R}(z))-d_y(f(z)) -\tfrac{3kb c^2}{a} (R-d_O(z))
\;\; \mbox{\rm for all $z$ in $B_{_R}$.}
\end{equation}
This function is the sum of three functions $\ph=\ph_1+\ph_2+\ph_3$.

The first function $\ph_1: z\mapsto d_y(h_{_R}(z))$  is subharmonic on $B_{_R}$ i.e. one has $\Delta \ph_1\geq 0$. This follows from 
Lemma \ref{lemdefh} and the harmonicity of the map $h_{_R}$.

The second function $\ph_2:z\mapsto -d_y(f(z))$ has a bouded Laplacian, namely
$
|\Delta \ph_2|\leq 3k bc^2.
$
Indeed, since $y$ is far away, Formula 
\eqref{eqnd2rhox0} yields the bound $\| D^2d_y\|\leq 2b$ on $f(B_R)$ 
so that 
\begin{eqnarray*}
\label{eqndeg}
|\Delta \ph_2|=|\Delta (d_{y}\circ f)| \leq
k\|D^2d_{y}\|\|D f\|^2+ k\| Dd_{y}\| \|D^2f\|
\leq 3kbc^2\, .
\end{eqnarray*}

The third function $\ph_3:z\mapsto -\tfrac{3k bc^2}{a} (R-d_O(z))$ has a Laplacian bounded below 
$
\Delta \ph_3\geq 3kb c^2.
$
This follows from Lemma \ref{lemboundedlap} which says that $\Delta d_O \geq a$.

Hence the function $\ph$ is subharmonic~: $\Delta \ph\geq 0$. Since $\ph$ is zero on $\partial B_{_R}$, 
one gets $\ph(x)\leq 0$ as required.
\end{proof}

\section{Interior estimate}
\label{secinterior}

In this chapter we complete the proof of Proposition \ref{proexistharm}.

\subsection{Strategy}  
\label{secstrexi}

\bq
We first explain more precisely the notations and the assumptions 
that we will use in the whole chapter.
\eq

Let $X$ and $Y$ be Hadamard manifolds whose curvatures are pinched 
$-b^2\leq K_X\leq -a^2<0$ and 
$-b^2\leq K_Y\leq -a^2<0$. Let $k=\dim X$ and $k' =\dim Y$. 
We start with a $\mc C^\infty$ 
quasi-isometric map $f:X\ra Y$ whose first and second  covariant derivatives are bounded.
We fix constants $c \geq 1$ and $C>0$ such that
one has, for all $x$, $x'$ in $X$~:
\begin{equation}
\label{eqnquasiiso3}
\| Df(x)\|\leq c
\;\;{\rm , }\;\;\;
\| D^2f(x)\|\leq bc^2
\;\;\;\;{\rm and }\;\;\;
\end{equation}
\begin{equation}
\label{eqnquasiiso2}
c^{-1}\, d(x,x')-C
\;\leq\; 
d(f(x),f(x'))
\;\leq\; 
c\, d(x,x')\; .
\end{equation}
Note that the additive constant $C$ on the right-hand side term of \eqref{eqnquasiiso}
has been removed since the derivative of $f$ is now bounded by $c$. 

\subsubsection{Choosing the radius $\ell_0$}
We fix a point $O$ in $X$. 
We introduce a fixed radius $\ell_0$ depending only on 
$a$, $b$, $k$, $k'$, $c$ and $C$.
This radius $\ell_0$ is only required to satisfy the following three inequalities \eqref{eqnell0}, \eqref{eqnell1} and \eqref{eqnell2}
that will be needed later on.

The first condition we impose on the radius $\ell_0$ is 
\begin{equation}
\label{eqnell0}
b\ell_0 > 1.
\end{equation}

The second condition we impose on the radius $\ell_0$ is 
\begin{equation}
\label{eqnell1}
\ell_0 > \, \frac{(A+b^{-1})c}{ \sin^2(\eps_0/2)}
\;\;{\rm where}\;\; 
\eps_0:=(3c^2M)^{-N},
\end{equation}
where $A$ is the  constant given 
by Lemma \ref{lemproduct}, 
and $M$, $N$ are the constants given by Proposition \ref{prosixrcxthe}.

The third condition we impose on the radius $\ell_0$ is 
\begin{equation}
\label{eqnell2}
16\,e^{\frac{aC}{2}}\,e^{-\frac{a\ell_0}{4\, c}} <
\theta_0
\;\;{\rm where}\;\; 
\th_0:=e^{-bA}\, (\eps_0/4)^{\frac{bc}{a}}.
\end{equation}

\subsubsection{Assuming $\rho$ to be large} 
We want to prove Proposition \ref{proexistharm}. 
For $R>0$, recall that $h_{_R}:B(O,R)\ra Y$ is
the harmonic $\mc C^\infty$ map whose restriction to the sphere
$\partial B(O,R)$ is equal to $f$.
We let 
\begin{equation*}
\label{eqnror}
\rho:=\sup_{x\in B(O,R)}d(h_{_R}(x),f(x))\; .
\end{equation*}
We argue by contradiction. If this supremum $\rho$ is not uniformly bounded with respect to $R$,
we can  fix a radius $R$ such that 
$\rho$ satisfies the following three inequalities \eqref{eqnrho1}, \eqref{eqnrho2} and \eqref{eqnrho3}
that we will use later on.

The first condition we impose on the radius $\rho$ is 
\begin{equation}
\label{eqnrho1}
a\rho > 8kbc^2\ell_0 \, .
\end{equation}

The second condition we impose on the radius $\rho$ is 
\begin{equation}
\label{eqnrho2}
\frac{2^7(a\rho)^2}{\sinh(a\rho/2)} <  \theta_0\, .
\end{equation}

The third condition we impose on the radius $\rho$ is 
\begin{equation}
\label{eqnrho3}
\rho >
4c\ell_0 M\,(2^{10}e^{b\ell_0}k)^{N}
\end{equation}
where  $M$, $N$ are the constants given by Proposition \ref{prosixrcxthe}.

We denote by $x$ a point of $B(O,R)$ where the supremum \eqref{eqnror} is achieved:
$$
d(h_{_R}(x),f(x))=\rho\, .
$$
According to the boundary estimate in Proposition \ref{proboundary}, 
 Condition \eqref{eqnrho1} yields
\begin{equation*}
\label{eqndx0dbr}
d(x,\partial B(O,R))\geq 
\frac{a\rho}{3kbc^2}\geq 2\ell_0\, .
\end{equation*} 
Combined with Condition \eqref{eqnell0}, this    ensures  that 
the ball $B(x,\ell_0)$ with center $x$ and radius $\ell_0$ satisfies the inclusion 
$B(x,\ell_0)\subset B(O,R\!-\!b^{-1})$.
This inclusion will allow us to apply the Cheng 's lemma \ref{eqncheng} at each point $z$ of the ball 
$B(x,\ell_0)$.

\subsubsection{Getting a contradiction}
We will focus on the restrictions of both  maps $f$ and $h_{_R}$ 
to this ball $B(x,\ell_0)$.
We introduce the point $y:=f(x)$. 
For $y_1$, $y_2$ in $Y\smallsetminus\{ y\}$,
we denote by 
$\theta_y(y_1,y_2)$ the angle at $y$ of the geodesic triangle
with vertices $y$, $y_1$, $y_2$. 
For $z$ on the sphere $S(x,\ell_0)$, we will 
analyze the triangle inequality~:
\begin{equation}
\label{eqnththth}
\th_y(f(z),h_{_R}(x))\leq \th_y(f(z),h_{_R}(z))+\th_y(h_{_R}(z),h_{_R}(x))
\end{equation}
and prove that on a subset $U_{\ell_0}$ of the sphere, 
each term on the right-hand side is small
(Lemmas \ref{lemI1} and \ref{lemI2}) 
while the measure of $U_{\ell_0}$ is large enough (Lemma \ref{lemsiwr})
to ensure that the left-hand side is not that small
(Lemma \ref{lemI0}),
giving rise to the contradiction. 
These arguments rely on uniform lower and upper bounds
for the harmonic measures on the spheres of $X$ that will be given in Proposition \ref{prosixrcxthe}. 

We denote by $\rho_h$ the  function on $B(x,\ell_0)$
given by $\rho_h(z)=d(y,h_{_R}(z))$ where again $y=f(x)$.
By Lemma \ref{lemdefh}, this function is subharmonic.

\begin{Def}
\label{defurvrwr}
The subset $U_{\ell_0}$ of the sphere  
$S(x,\ell_0 )$ is given by
\begin{eqnarray}
\label{eqnur}
U_{\ell_0}&=&
\{z\in S(x,\ell_0)\mid \rho_h(z)\geq \rho-\frac{\ell_0}{2c}\, \}\, .
\end{eqnarray}
\end{Def}

\subsection{Measure estimate}  
\label{secmeasure}

\bq
We first observe that  
one can control the size of $\rho_h(z)$ and of $Dh_{_R}(z)$ 
on the ball $B(x,\ell_0)$.
We then derive a lower bound for the measure of $U_{\ell_0}$.
\eq

\bl
\label{lemrhx}
For $z$ in $B(x,\ell_0)$, one has 
\begin{equation*}
\label{eqnrohx}
\rho_h(z)\leq \rho+c\,\ell_0.
\end{equation*}
\el

\begin{proof}
The triangle inequality and \eqref{eqnquasiiso2} give,
for any $z$ in $B(x,\ell_0)$~:
\vspace{1ex}

\mbox{}\hfill$
\rho_h(z)\leq d(h_{_R}(z),f(z))+d(f(z),y)\leq \rho+c\,\ell_0\, .
$\hfill
\end{proof}

\bl
\label{lemdhx}
For $z$ in $B(x,\ell_0)$, one has 
\begin{equation*}
\label{eqndhx}
\|Dh_{_R}(z)\|\leq 2^8 kb\rho.
\end{equation*}
\el

\begin{proof}
For all $z$, $z'$ in $B(O,R)$ with $d(z,z')\leq b^{-1}$, the triangle inequality and \eqref{eqnquasiiso2}
yield
\begin{eqnarray*}
d(h_{_R}(z),h_{_R}(z'))&\leq &
d(h_{_R}(z),f(z))+d(f(z),f(z'))+d(f(z'),h_{_R}(z'))\\
&\leq &\rho+ b^{-1}c +\rho\leq 2\,\rho+c\ell_0 \leq 3\,\rho\, . 
\end{eqnarray*}
For these last two inequalities, we used Conditions \eqref{eqnell0} and 
\eqref{eqnrho1}.
Applying the Cheng's lemma \ref{lemcheng} with $R_0=3\rho$ and 
$r=b^{-1}$,
one then gets for all $z$ in $B(O,R\! -\! b^{-1})$ the bound
$\|Dh_{_R}(z)\|\leq 2^8 kb\rho$.
\end{proof}

We now give a lower bound for the measure of $U_{\ell_0}$.

\begin{Lem}
\label{lemsiwr}
Let $\si=\si_{x,\ell_0}$ be the harmonic measure on the sphere $ S(x,\ell_0)$
at the center point $x$.
Then one has
\begin{eqnarray}
\label{eqnsiur}
\si(U_{\ell_0})&\geq &\frac{1}{3\,c^2}\; .
\end{eqnarray}
\end{Lem}

\begin{proof}
By Lemma \ref{lemdefh}, the function $\rho_h$ is subharmonic
on the ball $B(x,\ell_0)$.
Hence this function $\rho_h$ is not larger than the harmonic function on the ball
with same boundary values on the sphere $S(x,\ell_0)$. 
Comparing these functions at the center $x$, one gets
\begin{equation}
\label{eqnintrhru}
\int_{S(x,\ell_0)}(\rho_h(z)-\rho)\rmd\si(z) \geq 0\; .
\end{equation}

By Lemma \ref{lemrhx}, the function $\rho_h$ is bounded by $\rho+c\,\ell_0$.
Hence Equation \eqref{eqnintrhru} and the definition of $U_{\ell_0}$ implies
\begin{equation*}
c\, \ell_0\,\si(U_{\ell_0})-\frac{\ell_0}{2c}\,(1-\si(U_{\ell_0}))\geq 0\; 
\end{equation*} 
so that $\si(U_{\ell_0})\geq \frac{1}{3c^2}$.
\end{proof}

\subsection{Upper bound for $\theta_y(f(z),h_{_R}(z))$}
\label{secboundi1}
\bq 
For all $z$ in $U_{\ell_0}$, we give an upper bound for the 
the angle between $f(z)$ and $h_{_R}(z)$ seen from the point $y=f(x)$.
\eq

\bl
\label{lemI1}
For  $z$ in $U_{\ell_0}$, one has
\begin{eqnarray}
\label{eqnI1b}
\theta_y(f(z),h_{_R}(z))&\leq & 4\,e^{\frac{aC}{2}}\,e^{-\frac{a\ell_0}{4\, c}}.
\end{eqnarray} 
\el

\begin{proof}
For $z$ in $U_{\ell_0}$, we consider the triangle with vertices 
$y$, $f(z)$ and $h_{_R}(z)$. Its side lengths satisfy
$$
d(h_{_R}(z),f(z))\leq \rho
\;\; ,\;\;
d(y,f(z))\geq \,\frac{\ell_0}{c} - C
\;\; ,\;\;
d(y,h_{_R}(z))\geq\rho-\frac{\ell_0}{2c}\, ,
$$
where we used successively the definition of $\rho$, 
the quasi-isometry lower bound \eqref{eqnquasiiso2} 
and the definition of $U_{\ell_0}$.
Hence, one gets the following lower bound for the Gromov product
$$
(f(z)|h_{_R}(z))_y\geq \frac{\ell_0}{4c} -\frac{C}{2}.
$$
Since $K_Y\leq -a^2$,  Lemma \ref{lemcomparison} now
yields
\vspace{1ex}

\mbox{}\hfill
$\displaystyle
\theta_y(f(z),h_{_R}(z))\leq 
4\, e^{\frac{aC}{2}}\,e^{-\frac{a\ell_0}{4\, c}}\, .
$ 
\hfill
\end{proof}

\subsection{Upper bound for $\theta_y(h_{_R}(z),h_{_R}(x))$}
\label{secboundi2}
\bq 
For all $z$ in $S(x,\ell_0)$, we give an upper bound for the angle between 
$h_{_R}(z)$ and $h_{_R}(x)$ seen from the point $y=f(x)$.
\eq  

\bl
\label{lemI2}
For all $z$ in the sphere $S(x,\ell_0)$, one has
\begin{eqnarray}
\label{eqnI2b}
\theta_y(h_{_R}(z),h_{_R}(x))&\leq &
\frac{2^5\,(a\rho)^2}{{\sinh}(a\rho/2)}.
\end{eqnarray} 
\el

The proof will rely on the following lemma 
which also ensures that this angle $\theta_y(h_{_R}(z),h_{_R}(x))$
is well defined.

\bl
\label{lemI3} 
For  all $z$ in the ball $B(x,\ell_0)$,
one has
$\rho_h(z)\geq \rho/2.$
\el

\begin{proof}[Proof of Lemma \ref{lemI3}]
Assume by contradiction that there exists a point $z_1$ in the ball $B(x,\ell_0)$ 
such that $\rho_h(z_1)=\rho/2$. Set $r_1:=d(x,z_1)$. One has $0<r_1\leq \ell_0$.
According to  Lemma \ref{lemdhx}, one can bound the differential of $h_{_R}$
on the ball $B(x,\ell_0)$, namely
\begin{eqnarray*}
\sup_{B(x,\ell_0)}\| Dh_{_R}\|
\leq 2^{8}kb\rho\, . 
\end{eqnarray*}  
Hence one has 
\begin{eqnarray*}
\rho_h(z)\leq \frac{3\rho}{4}
\;\; \mbox{\rm for all $z$ in $S(x,r_1)\cap B(z_1, \frac{1}{2^{10}kb}).$} 
\end{eqnarray*}  
By comparison with the hyperbolic plane with curvature $-b^2$,
this intersection contains the trace on the sphere $S(x,r_1)$ of a cone 
$C_\al$ with vertex $x$ and angle $\alpha$ as soon as 
$\sin\frac{\al}{2}\leq \frac{\sinh(2^{-11}/k)}{\sinh(br_1)}$.
For instance we will choose for $\al$ the angle $\al:= e^{-b\ell_0}2^{-10}/k$.

Let $\si'=\si_{x,r_1}$ be the harmonic measure on the sphere $S(x,r_1)$
for the center point $x$.
Using the subharmonicity of 
the function $\rho_h$ as in the proof of Lemma \ref{lemdefh}, one gets the inequality
\begin{equation}
\label{eqnintrhtu}
\int_{S(x,r_1)}(\rho_h(z)-\rho)\rmd\si'(z) \geq 0\; .
\end{equation} 
By Lemma \ref{lemrhx}, the function $\rho_h$ is bounded by $\rho+c\,\ell_0$.
Using the bound $\rho_h(z)\leq \tfrac34\,\rho$ when $z$ is in the cone $C_\alpha$, 
Equation \eqref{eqnintrhtu} now implies that
$$
c\,\ell_0-\frac{\rho}{4}\,\si'(C_\alpha)\geq 0\; .
$$
Using the uniform lower bounds for the harmonic measures on the spheres of $X$ in Proposition \ref{prosixrcxthe},
one gets
$$
\rho\leq 4c\ell_0M\, \al^{-N}
=4c\ell_0 M\,(2^{10}e^{b\ell_0}k)^{N},
$$
which contradicts Condition \eqref{eqnrho3}. 
\end{proof}

\begin{proof}[Proof of Lemma \ref{lemI2}]
Let us first sketch the proof. Let $z$ be a point on the sphere 
$S(x,\ell_0)$.
We denote by $t\mapsto z_t$, for $0\leq t\leq \ell_0$,  
the geodesic segment
between $x$ and  $z$.
By Lemma \ref{lemI3}, the curve 
$t\mapsto h_{_R}(z_t)$ lies outside of the ball 
$B(y,\rho/2)$
and by Cheng's bound on $\| Dh_{_R}(z_t)\|$ 
one controls the length of this curve.

We now  detail the argument. 
We denote by 
$(\rho(y'), v(y'))\in\; ]0,\infty[\times T^1_{y}Y$
the polar exponential coordinates  centered at $y$. For a point $y'$  in $Y\smallsetminus\{ y\}$, they are
defined by the equality  $y'=\exp_y(\rho(y')v_\rho(y'))$.
Since $K_Y\leq -a^2$ the Alexandrov comparison theorem 
for infinitesimal triangles and the Gauss lemma (\cite[2.93]{GHL04}) yield
\begin{equation*}
\label{eqndvy}
{\sinh}(a\rho(y'))\,\| Dv(y')\|\leq a\, .
\end{equation*}
Writing $v_h:=v\circ h_{_R}$, we thus have  for  any $z'$ in $B(x,\ell_0)$~: 
\begin{equation*}
\label{eqndvhdh}
{\sinh}(a\rho_h(z'))\,\| Dv_h(z')\|\leq a\| Dh_R (z')\|\, .
\end{equation*}
Hence,   Lemma \ref{lemI3} yields the inequality
\begin{eqnarray*}
\label{eqnIR2a}
\theta_y(h_{_R}(z), h_{_R}(x))
&\leq &
\ell_0\,\sup_{0\leq t\,\leq \,\ell_0}\| Dv_h(z_t)\|\\
&\leq &
\frac{a\ell_0}{{\sinh}(a\rho/2)}\,\sup_{0\leq t\,\leq \,\ell_0}\| Dh_R(z_t)\|\, .
\end{eqnarray*}
Therefore, using  Lemma \ref{lemdhx} and Condition \eqref{eqnrho1}, one gets
\vs

\mbox{ }\hfill
$\displaystyle
\theta_y(h_{_R}(z), h_{_R}(x))
\leq 
\frac{2^8 kb\rho\,a\ell_0}{{\sinh}(a\rho/2)}\leq 
\frac{2^5 (a\rho)^2}{{\sinh}(a\rho/2)}\,.
$
\end{proof}

\subsection{Lower bound for $\theta_y(f(z),h_{_R}(x))$}
\label{secboundi0}
\bq 
We find a point $z$ in $U_{\ell_0}$ for which the angle between $f(z)$ and $h(x)$ 
seen from $y=f(x)$ has an explicit lower bound.
\eq  

\bl
\label{lemI0} 
There exist two points $z_1$, $z_2$ in $U_{\ell_0}$ such that 
\begin{equation*}
\label{eqnthz1z2}
\th_y(f(z_1),f(z_2))\geq \th_0\, ,
\end{equation*}
where $\th_0$ is the angle given by \eqref{eqnell2}.
\el

\begin{proof}[Proof of Lemma \ref{lemI0}]
Let $\si_0:=\frac{1}{3c^2}$. According to Lemma \ref{lemsiwr}, one has
$\si( U_{\ell_0})\geq\si_0 >0$
Thus, using the uniform upper bounds for the harmonic measures on the spheres of $X$ in Proposition \ref{prosixrcxthe},
one can find 
$z_1$, $z_2$ in $U_{\ell_0}$ such that the angle $\theta_{x}(z_1,z_2)$ 
between $z_1$ and $z_2$ seen from $x$ satisfies
$$
\si_0 \leq M\,\theta_{x}(z_1,z_2)^{\frac{1}{N}}\, .
$$
This can be rewritten as 
\begin{equation}
\label{eqnu1u2}
\theta_{x}(z_1,z_2)\geq \eps_0\, ,
\end{equation}
where $\eps_0$ is the angle introduced in 
\eqref{eqnell1} by the equality
$\si_0=M\,\eps_{_0}^{\frac{1}{N}}$.
Therefore, using  Lemma \ref{lemcomparison}.$a$ and Condition \eqref{eqnell1}, we 
get the following lower bound on the Gromov products
\begin{equation*}
\label{eqnxrzz}
\min( (x|z_1)_{z_2},(x|z_2)_{z_1}))\geq 
\ell_0\sin^2(\eps_0/2)\geq (A+b^{-1})c.
\end{equation*}
Using then Lemma \ref{lemproduct}, one gets 
\begin{equation}
\label{eqnyrfxfx}
\min( (y|f(z_1))_{f(z_2)},(y|f(z_2))_{f(z_1)})\geq b^{-1}.
\end{equation}
This inequality \eqref{eqnyrfxfx} allows us to apply 
Lemma \ref{lemcomparison}.$c$, which gives
\begin{eqnarray*}
\theta_y(f(z_1),f(z_2))
&\geq & 
e^{-b(f(z_1)|f(z_2))_{y}}.
\end{eqnarray*}
Therefore, by Lemma \ref{lemproduct}, one has
\begin{eqnarray*}
\theta_y(f(z_1),f(z_2))&\geq &
e^{-bA}\, e^{-bc\,(z_1|z_2)_{x}}.
\end{eqnarray*}
Using Lemma \ref{lemcomparison}.$b$
and Condition \eqref{eqnu1u2}, one gets
\begin{eqnarray*}
\theta_y(f(z_1),f(z_2))
&\geq&
e^{-bA}\, (\theta_x(z_1,z_2)/4)^{\frac{bc}{a}}\\
&\geq&
e^{-bA}\, (\eps_0/4)^{\frac{bc}{a}}
= \th_0,
\end{eqnarray*}
according to the definition \eqref{eqnell2} of $\th_0$.
\end{proof}

\begin{proof}[End of the proof of Proposition \ref{proexistharm}] 
Using  Lemmas \ref{lemI1} and \ref{lemI2} and the triangle inequality \eqref{eqnththth}, one gets for any two points $z_i=z_{1}$ or $z_{2}$
in $U_{\ell_0}$~: 
\begin{eqnarray*}
\label{eqnsupdvfvr}
\theta_y(f(z_i),h_{_R}(x))
&\leq& 4\,e^{\frac{aC}{2}}\,e^{-\frac{a\ell_0}{4\, c}}+ \frac{2^5(a\rho)^2}{\sinh(a\rho/2)}\\
&< &
\tfrac12\th_0
\hspace{12ex} \mbox{\rm by Conditions \eqref{eqnell2} and \eqref{eqnrho2}.}
\end{eqnarray*}
Therefore, using again a triangle inequality, one has 
$$
\th_y(f(z_1),f(z_2))< \th_0,
$$
which contradicts Lemma \ref{lemI0}. 
\end{proof}

\subsection{Harmonic measures}
\label{secharmea}

\bq
The following proposition gives the uniform lower and upper bounds for the 
harmonic measure  on a sphere for the center which were used in the proof
of Lemmas \ref{lemI3} and \ref{lemI0}.
\eq

\bp
\label{prosixrcxthe}
Let $0<a<b$ and  $k\geq 2$ be an integer.
There exist positive constants 
$M$, $N$ depending only on 
$a$, $b$, $k$ 
such that for every $k$-dimensional Hadamard manifold $X$ with 
pinched curvature $-b^2 \leq K_X\leq -a^2$, for every point $x$ in $X$,
every radius $r>0$ and every angle $\theta\in [0,\pi]$ one has
\begin{equation}
\label{eqnsixrcxthe}
\tfrac{1}{M} \,\theta ^{N} \leq \sigma_{x,r} (C_{x,\theta}) \leq 
M \,\theta ^{\frac{1}{N}} 
\end{equation}
where
$\sigma_{x,r}$ denotes the harmonic measure on the sphere $S(x,r)$
at the point $x$
and where  $C_{x,\theta}$ stands for any 
cone with vertex $x$ and angle $\theta $.
\ep

We recall that, by definition, $\sigma_{x,r}$ 
is the unique probability measure on the sphere 
$S(x,r)$ such that, for every continuous function 
$h$ on the ball $B(x,r)$ which is harmonic in the 
interior $\mathring{B}(x,r)$, one has the equality 
$$
h(x)=\int_{S(x,r)}h(z)\rmd\si_{x,r}(z).
$$

A proof of Proposition \ref{prosixrcxthe} 
is given in 
\cite{BH17}.
It relies on 
various technical tools of the potential theory on
pinched Hadamard manifolds~: the Harnack inequality,
the barrier functions constructed by  Anderson and Schoen in 
\cite{AndersonSchoen85} and upper and lower bounds 
for the Green functions due to 
Ancona in \cite{Ancona87}. 
Related estimates are available like the one by Kifer--Ledrappier in  
\cite[Theorem 3.1 and 4.1]{KiferLedrappier} 
where \eqref{eqnsixrcxthe} is proven for the sphere at infinity
or by Ledrappier--Lim in 
\cite[Proposition 3.9]{LedrappierLim} where the H\"{o}lder regularity 
of the Martin kernel is proven.

\section{Uniqueness of harmonic maps}
\label{secunihar}

In this chapter we prove the uniqueness part in Theorem \ref{thdfxhx}.

\subsection{Strategy}
\label{secstruni}

In other words we will prove the following proposition.

\bp
\label{prouni}
Let $X$, $Y$ be two pinched Hadamard manifolds 
and let $h_0, h_1:X\ra Y$ be two quasi-isometric  harmonic maps 
that stay within bounded distance
of one another:
$$
\sup_{x\in X}d(h_0(x),h_1(x))<\infty\, .
$$
Then one has  $h_0=h_1$.
\ep

When $X=Y=\m H^2$, this proposition was first proven by Li and Tam 
in \cite{LiTam93}.
When both $X$ and $Y$ admit a cocompact group of isometries,
this proposition was then proven by Li and Wang in 
\cite[Theorem 2.3]{LiWang98}.  
The aim of this chapter is to explain how to 
get rid of these extra assumptions.
\vs

Note that the assumption that the $h_i$ are quasi-isometric is useful.
Indeed there does exist non constant bounded harmonic functions on $X$.
Note that there also exist bounded harmonic maps with open images.
Here is a very simple example. Let $0< \la < 1$. The map $h_\la$ from the Poincar\'e unit disk 
$\m D$ of $\m C$ into itself given by $z\mapsto \la z$ is harmonic. 
More generally, for any harmonic map $h:\m D\ra\m D$, the map 
$h_\la:\m D\ra\m D: z\mapsto h(\la z)$ is a harmonic map with bounded image.
\vs

Before going into the technical details, we first explain the strategy
of the proof of this uniqueness. 

\begin{proof}[Strategy of  proof of Proposition \ref{prouni}]
We recall that the distance function 
$x\mapsto d(h_0(x),h_1(x))$ is a subharmonic function on $X$
and that, by the maximum principle, a subharmonic function 
that achieves its maximum value is constant.
Unfortunately since $X$ is non-compact 
we can not a priori ensure that this bounded function 
achieves its maximum. This is why we will use a recentering argument.

We assume, by contradiction, that $h_0\neq h_1$ and 
we choose a sequence of points $p_n$ in $X$ for which the distances
\begin{eqnarray}
\label{eqndhpnhpn}
d(h_0(p_n),h_1(p_n))
\;\;\mbox{\rm converge to}\;\; 
\delta:=\sup_{x\in X}d(h_0(x),h_1(x))>0\, 
\end{eqnarray}
and we set $q_n:=h_0(p_n)$. 

The pinching condition on $X$ and $Y$ ensure that, 
after extracting
a subsequence, 
the pointed metric spaces
$(X,p_n)$ and $(Y,q_n)$ converge in the Gro\-mov--Hausdorff topology to 
pointed metric spaces $(X_\infty,p_\infty)$ and $(Y_\infty,q_\infty)$
which are $\mc C^{2}$ Hadamard manifolds with  $\mc C^1$ Riemannian metrics
satisfying the same pinching conditions
(Proposition \ref{prohadhad}).
Moreover, extracting again a subsequence, the  harmonic map $h_0$ (resp. $h_1$) 
seen as a sequence of maps between 
the pointed Hadamard manifolds $(X,p_n)$ and $(Y,q_n)$ 
converges locally uniformly to a map $h_{0,\infty}$ (resp $h_{1,\infty}$)
between 
the pointed $C^2$ Hadamard manifolds 
$(X_\infty,p_\infty)$ and $(Y_\infty,q_\infty)$.
These harmonic maps $h_{0,\infty}$ and $h_{1,\infty}$ are still harmonic quasi-isometric maps
(Lemma \ref{lemconhar}).

The limit distance function
$x\mapsto d(h_{0,\infty}(x),h_{1,\infty}(x))$ is a subharmonic function on $X_\infty$
that now achieves its maximum $\delta>0$ at the point $p_\infty$. Hence, 
by the maximum principle, this distance function is constant and equal to $\delta$
(Lemma \ref{lemlimequ}).
Generalizing \cite[Lemma 2.2]{LiWang98}, we will see 
in Corollary \ref{cortwohar} that this equidistance property implies that 
both $h_{0,\infty}$ and $h_{1,\infty}$ take their values in a geodesic of $Y_\infty$.
This contradicts the fact that  $h_{0,\infty}$ 
and $h_{1,\infty}$ are quasi-isometric maps,
and concludes this strategy of proof.
\end{proof}

In the following sections of Chapter \ref{secunihar}, we fill in the details of the proof.

\subsection{Harmonic coordinates}
\label{secharcoo}

\bq
We first introduce the so-called harmonic coordinates,
which improve the quasilinear coordinates introduced in Lemma \ref{lemquacoo}.
We refer to \cite{GreeneWu88} or \cite{Jost84} for more details.
\eq

The harmonic coordinates have been introduced by DeTurk and Kazdan
and extensively used by Cheeger, Jost, Karcher, Petersen...
to prove various  compactness results for compact Riemannian manifolds.
Beyond being harmonic, the main advantage of these coordinates 
is that, 
for every $\alpha\in ]0,1[$, they are uniformly bounded in $\mc C^{2,\al}$-norm,
i.e. they are uniformly bounded in $\mc C^2$-norm
and one also has a uniform control of 
the $\alpha$-H\"{o}lder norm of their second covariant derivatives.
Moreover, one has a uniform control on the size of the balls on which these harmonic charts are defined.
This is what the following lemma tells us.

We endow $\m R^k$ with the standard Euclidean structure.

\bl
\label{lemharcoo} 
Let $X$ be a  $k$-dimensional Hadamard manifold  with 
bounded curva\-ture
$-1\leq K_X\leq 0$. Let $0<\al <1$. 
There exist two constants $r_0=r_0(k)>0$ and 
$c_0=c_0(k,\al)>0$ 
such that, for every $x$ in $X$, there exists a 
$\mc C^\infty$-diffeomorphism 
\begin{equation}
\label{eqnpsxbxr}
\Psi_x: \mathring{B}(x,r_0)\stackrel{\sim}{\longrightarrow} U_{x}\subset \m R^{k}
\;\; {\rm with}\;\; 
\Psi_x(x)=0\; ,
\end{equation}
\begin{equation}
\label{eqndpsdps}
\|D\Psi_x\|\leq c_0
\;\; ,\;\;
\|D\Psi_x^{-1}\|\leq c_0
\;\; ,\;\;
\|D^2\Psi_x\|\leq c_0
\;\; ,\;\;
\|D^2\Psi_x^{-1}\|\leq c_0
\; 
\end{equation}
and such that
each component $z_1,\ldots, z_k$ of $\Psi_x$ is a harmonic function.

In particular, one has for all $r< r_0$~:
\begin{equation}
\label{eqnpsbbux}
\Psi_x(B(x,c_0^{-1}r))\subset B(0,r)
\;\;{\rm and}\;\; 
B(0,c_0^{-1}r)\subset\Psi_x(B(x,r)).
\end{equation}
$(iii)$ The  second covariant derivatives of all $\Psi_x$
are also uniformly $\al$-H\"{o}lder~: 
\begin{equation}
\label{eqndpscal}
\|D^2\Psi_x\|_{\mc C^\al}\leq c_0.
\end{equation}
\el

This $\al$-H\"{o}lder semi-norm $\|D^2\Psi_x\|_{\mc C^\al}$ 
is defined as follows.
Using the vector fields 
$\tfrac{\partial}{\partial z_1},\ldots, \tfrac{\partial}{\partial z_k}$ on $\mathring{B}(x, r_0)$ associated to our coordinate system
$\Psi_x=(z_1,\ldots, z_k)$, we reinterpret the tensor $D^2\Psi_x$ 
as a family of vector valued functions on $\mathring{B}(x, r_0)$.
Indeed, we set
$$
T_x^{ij}(z)= 
D^2\Psi_x(z)(\tfrac{\partial}{\partial z_i},\tfrac{\partial}{\partial z_j})\in \m R^k
\; ,\;\; \mbox{\rm for}\; i,j\; \mbox{in}\; \{1,\ldots ,k\},
$$
and the bound \eqref{eqndpscal} means that
\begin{equation}
\label{eqntijcal}
\|D^2\Psi_x\|_{\mc C^\al}:=\max_{i,j}\sup_{z,z'} 
\frac{\| T_x^{ij}(z)-T_x^{ij}(z')\|}{d(z,z')^\al}\leq c_0.
\end{equation}

These uniform bounds \eqref{eqndpsdps} and \eqref{eqndpscal} have three consequences.

First, 
in the harmonic coordinate systems $\Psi_x=(z_1,\ldots, z_k)$, the 
Chris\-toffel coefficients $\Gamma^\ell_{ij}$
are uniformly bounded in $\mc C^{\al}$-norm.
Indeed, these coefficients $(\Gamma^\ell_{ij})_{1\leq \ell\leq k}$
are the components of the vector
$
-D^2\Psi_x(\tfrac{\partial}{\partial z_i},\tfrac{\partial}{\partial z_j})
\in\m R^k$.

Second, on their domain of definition, the transition functions 
\begin{equation}
\label{eqnpsipsical}
\mbox{
$\Psi_{x'}\circ\Psi_x^{-1}$ 
are uniformly bounded in the $\mc C^{2,\al}$-norm.}
\end{equation}

Third, in the coordinate systems $\Psi_x=(z_1,\ldots, z_k)$, the 
coefficients of the metric tensor
\begin{equation}
\label{eqngijcal}
\mbox{
$g_{ij}:=g(\tfrac{\partial}{\partial z_i},\tfrac{\partial}{\partial z_j})$
are uniformly bounded in the $\mc C^{1,\al}$-norm.}
\end{equation}

\begin{proof}[Proof of Lemma \ref{lemharcoo}]
See \cite[p. 62 and 65]{Jost84} or \cite[Section 4]{Petersen94}.
\end{proof}

\subsection{Gromov-Hausdorff convergence}
\label{secgrohau}

\bq
In this section, we recall the definition of Gromov--Hausdorff convergence
for pointed metric spaces
and some of its key properties. We refer to \cite{BuragoBuragoIvanov} for more details.
\eq

\subsubsection{Definition} 
When $X$ is a metric space, we will denote by $d$ or $d_X$ the distance on $X$.
We denote by $B(x,R)$ the closed ball with center $x$ and radius $R$,
and by $\mathring{B}(x,R)$ the open ball.
We recall that a metric space $X$ is {\it proper} if all its balls are compact
or, equivalently, if $X$ is complete and for all $R>0$ and $\eps>0$ 
every ball of radius $R$ can be covered by finitely many balls with radius $\eps$.

We also recall the notion of Gromov--Hausdorff distance  between two (isometry class of proper) pointed metric spaces.
  
\begin{Def}
\label{defgrohau}
The Gromov--Hausdorff distance between two pointed metric spaces
$(X,p)$ and $(Y,q)$ is the infimum of the $\eps>0$ for which there exists
a subset $\mc R$ of $X\times Y$, called a {\it correspondence}, such that~:\\
$(i)$ the correspondence $\mc R$ contains the pair $(p,q)$,\\
$(ii)$ for all $x$ in  the ball $B(p,\eps^{-1})$, there exists $y$ in $Y$ with $(x,y)$ in $\mc R$,\\
$(iii)$ for all $y$ in  the ball $B(q,\eps^{-1})$, there exists $x$ in $X$ with $(x,y)$ in $\mc R$,\\
$(iv)$ for all $(x,y)$ and $(x',y')$ in $\mc R$, one has 
$|d(x,x')-d(y,y')|\leq \eps$.
\end{Def}

Heuristically, this correspondence $\mc R$ must be thought as an $\eps$-rough
isometry between these two balls with radius $\eps^{-1}$.  

Based on this definition, 
a sequence $(X_n,p_n)$ of pointed metric spaces converges 
to a pointed metric space
$(X_\infty ,p_\infty)$ if, for all $\eps>0$, there exists $n_0$ such that for $n\geq n_0$, 
there exists a map $f_n: B(p_n,\eps^{-1})\ra X_\infty$ such that\\
$(\al)$ $d(f_n(p_n),p_\infty)\leq \eps$ ,\\
$(\be)$ $|d(f_n(x),f_n(x'))-d(x,x')|\leq \eps\, $, for all $x$, $x'$ in $B(p_n,\eps^{-1})$,\\
$(\ga)$ the $\eps$-neighborhood of $f_n(B(p_n,\eps^{-1}))$ contains 
the ball $B(p_\infty,\eps^{-1}-\eps)$.

This definition \ref{defgrohau} is only useful for complete metric spaces. Indeed,
the Gromov--Hausdorff topology 
does not distinguish between a metric space and its completion.
It does not  distinguish either between two pointed metric spaces
that are isometric~: 
the Gromov--Hausdorff distance 
is a distance on the set of isometry classes of 
proper pointed metric spaces.
See \cite[Theorem 8.1.7]{BuragoBuragoIvanov}

The following equivalent definition of Gromov--Hausdorff convergence
is useful when one wants to get rid of the ambiguity coming from the group 
of isometries of $(X_\infty,p_\infty)$.

\begin{Fact}
\label{facxpxpxp}
Let $(X_n,p_n)$, for $n\geq 1$, and $(X_\infty,p_\infty)$ be pointed proper 
metric spaces.
The sequence $(X_n,p_n)$ converges to $(X_\infty,p_\infty)$
if and only if there exists a complete metric space $Z$ containing isometrically 
all the metric spaces $X_n$ and $X_\infty$ as disjoint closed subsets, and such
that\\
$(a)$ the sequence of points $p_n$ converges to $p_\infty$ in $Z$,\\
$(b)$ the sequence of closed subsets $X_n$ converges to $X_\infty$ for the Hausdorff topology.
\end{Fact}

Statement $(b)$ means that\\
- every point $z$ of $X_\infty$ 
is the limit of a sequence $(x_n)_{n\geq 1}$ with $x_n\in X_n$,\\
-  every cluster point $z\in Z$ of a sequence $(x_n)_{n\geq 1}$ with $x_n\in X_n$
belongs to $X_\infty$.

\begin{proof}[Sketch of proof of Fact \ref{facxpxpxp}] Assume that the sequence 
$(X_n,p_n)$ converges to $(X_\infty,p_\infty)$.
We want to construct the metric space $Z$.
We can choose a sequence $\eps_n\searrow 0$, and correspondences $\mc R_n$ 
on $X_n\times X_\infty$ as in Definition \ref{defgrohau} with 
$p=p_n$, $q=p_\infty$ and $\eps=\eps_n$.
This allows us to construct, for every $n\geq 1$, a metric space $Y_n$ 
which is the disjoint union of $X_n$ and 
$X_\infty$,
which contains isometrically both $X_n$ and  $X_\infty$
and such that the distance between two points $x$ in $X_n$ and $y$ in $X_\infty$
is given by 
\begin{equation}
\label{eqndxxdyy}
d_{Y_n}(x,y)=\inf\{ d_{X_n}(x,x')+\eps +d_{X_\infty}(y',y)\},
\end{equation}
where the infimum is over all the pairs $(x',y')$ which belong to $\mc R_n$.

The space $Z$ is defined as the disjoint union of 
all the $X_n$ and of $X_\infty$. 
The distance on $Z$ is given on 
each union $Y_n:= X_n\cup X_\infty$
by \eqref{eqndxxdyy}
and the distance between points $x$ in $X_m$ and $z$ in $X_n$
with $m\neq n$
is given by 
\begin{equation}
\label{eqndxydyz}
d_Z(x,z)=\inf\{ d_{Y_m}(x,y)+d_{Y_n}(y,z)\},
\end{equation}
where the infimum is over all the points $y$ in $X_\infty$.

Then  $(a)$ follows from $(i)$ and  $(b)$ 
follows from $(ii)$, $(iii)$ and $(iv)$.
\end{proof}

The choice of such isometric embeddings of all $X_n$ and $X_\infty$ in a fixed
metric space $Z$ will be called
a realization of the Gromov-Hausdorff convergence.
Such a realization is not unique. It is useful since it allows us to 
define the notion of a converging sequence of points $x_n$ in $X_n$ to  a limit $x_\infty$ in $X_\infty$ by the condition $d_Z(x_n,x_\infty)\xrightarrow[n\to\infty]{} 0$.

\subsubsection{Compactness criterion}
A fundamental tool in this topic is the following compactness result
for {\it uniformly proper} pointed metric spaces
due to Cheeger--Gromov~:

\begin{Fact}
\label{faccomcri}
Let $(X_n,p_n)_{n\geq 1}$ be a sequence of 
pointed proper metric spaces.
Suppose that, for all $R>0$ and $\eps>0$, there exists 
an integer $N=N(R,\eps)$ such that, for all $n\geq 1$, the ball
$B(p_n,R)$ of $X_n$ can be covered by $N$ balls with radius $\eps$.
Then there exists a subsequence of  $(X_n,p_n)$ 
which converges to a proper pointed metric space $(X_\infty,p_\infty)$.
\end{Fact}

For a proof see \cite[Theorem 8.1.10]{BuragoBuragoIvanov}.

The following lemma gives us a compactness property for sequences
of Lipschitz functions between pointed metric spaces.

\bl
\label{lemcomcri}
Let $(X_n,p_n)_{n\geq 1}$ and 
$(Y_n,q_n)_{n\geq 1}$ be sequences of 
pointed proper metric spaces
which converge respectively to proper pointed metric spaces 
$(X_\infty,p_\infty)$ and $(Y_\infty,q_\infty)$.
As in Fact \ref{facxpxpxp}, 
we choose metric spaces $Z_X$ and $Z_Y$ which realize these 
Gromov--Hausdorff convergences as Hausdorff convergences.

Let $c>1$ and let $(f_n:X_n\ra Y_n)_{n\geq 1}$ be a sequence of $c$-Lipschitz 
maps such that $f_n(p_n)=q_n$.  
Then there exists a $c$-Lipschitz map $f_\infty:X_\infty\ra Y_\infty$
such that, after extracting a subsequence, the sequence of  maps $f_n$ converges to $f_\infty$. 
This means that for each sequence  $x_n \in X_n$ which converges to 
$x_\infty\in X_\infty$, the sequence $f_n(x_n)\in Y_n$ converges to
$f_\infty(x_\infty)\in Y_\infty$.
\el

\begin{proof}
This follows from basic topology arguments.

{\bf First step}. 
We first choose a point $x_\infty$ in $X_\infty$ and a sequence $x_n$ in $X_n$ 
converging to $x_\infty$. Since the metric space $Z_Y$ is proper and the sequence $f_n(x_n)$ is bounded in $Z_Y$
we can assume after extracting a subsequence that 
the sequence $f_n(x_n)$ converges to a point $y_\infty\in Y_\infty$.
Since the $f_n$ are $c$-Lipschitz, this limit $y_\infty$ does not depend on 
the choice of the sequence $x_n$ converging to $x_\infty$.
We define $f_\infty(x_\infty):=y_\infty$. 

{\bf Second step}.
We choose a countable dense subset $S_\infty\subset X_\infty$ and
use Cantor's diagonal argument to ensure that the first step 
is valid simultaneously for all points $x_\infty$ in $S_\infty$.

{\bf Last step}. One checks that the limit map $f_\infty:S_\infty\ra Y_\infty$ is $c$-Lipschitz.
Hence it extends uniquely as a $c$-Lipschitz map $f_\infty:X_\infty\ra Y_\infty$
and the sequence $f_n$ converges locally uniformly to $f_\infty$.
\end{proof}

\subsubsection{Length spaces and Alexandrov spaces}

We recall a few well-known definitions (see \cite{BuragoBuragoIvanov}). 

A {\it length space} is a complete metric space for which 
the distance $\delta$ between two points is the infimum of the length of the curves
joining them. When $X$ is proper, any two points at distance $\delta$ can be joined by a curve of length $\delta$. Such a curve is called a {\it geodesic segment}.

Let $K\leq 0$. A {\it {\rm CAT}(K)-space} or {\it {\rm CAT}-space with curvature  at most $K$} 
is a length space in which any geodesic triangle $(P,Q,R)$ is 
thinner than a comparison triangle $(\ol{P},\ol{Q},\ol{R})$ 
in the plane $\ol{X}$ of constant curvature $K$. Let us explain what this means. 
A {\it comparison triangle} is a triangle in $\ol{X}$ with the same side lengths. 
For every point $P'$ on the geodesic segment $[P,Q]$ we denote by $\ol{P}'$
the corresponding point on the geodesic segment $[\ol{P},\ol{Q}]$ 
i.e. the point such that $d(P,P')=d(\ol{P},\ol{P}')$. 
{\it Thinner} means that one always has $d(P',R)\leq d(\ol{P}',\ol{R})$.
Note that a {\rm CAT}(0)-space is always simply connected
(See \cite[Corollary II.1.5]{BridsonHaefliger}).
We also recall that in a proper ${\rm CAT}(0)$-space, any two points 
can be joined by a geodesic and that this geodesic is unique. 

Similarly, a {\it metric space with curvature  at least $K$} 
is a length space in which any geodesic triangle $(P,Q,R)$ is 
thicker than a comparison triangle $(\ol{P},\ol{Q},\ol{R})$ 
in the plane $\ol{X}$ of constant curvature $K$. 
{\it Thicker} means that one always has $d(P',R)\geq d(\ol{P}',\ol{R})$.
\vs

The following proposition tells us that these properties 
are closed for the Gromov--Hausdorff topology.

\begin{Fact}
\label{facgrocat}
Let $(X_n,p_n)_{n\geq 1}$ and $(X_\infty,p_\infty)$ be pointed proper 
metric spaces. Let $K\leq 0$.
Assume that the sequence $(X_n,p_n)$ converges to $(X_\infty,p_\infty)$.\\
$(i)$ If the $X_n$'s are length spaces, then $X_\infty$ is also a length space.\\
$(ii)$ If the $X_n$'s are ${\rm CAT}(K)$ spaces,
then $X_\infty$ is also a ${\rm CAT}(K)$ space.\\
$(iii)$ If moreover the $X_n$'s have curvature at least $K$,
then $X_\infty$ too.
\end{Fact}

\begin{proof}
$(i)$ See  \cite[Theorem 8.1.9]{BuragoBuragoIvanov}.

$(ii)$ See \cite[Corollary II.3.10]{BridsonHaefliger}.

$(iii)$ See \cite[Theorem 10.7.1]{BuragoBuragoIvanov}.
\end{proof}

\subsection{Hadamard manifolds with $\mc C^1$ metrics}
\label{secliprie}

\bq
In this section we focus on $\mc C^2$ Hadamard manifolds 
when the Riemannian metric is only assumed to be $\mc C^1$. 
These Hadamard manifolds will occur in Section \ref{seclimhad} 
as Gromov-Hausdorff limits of pinched $\mc C^\infty$ Hadamard manifolds.
\eq

\subsubsection{Definition}
We need first to clarify the definitions.
We will deal with $\mc C^{2}$ manifolds $X$.
This means that $X$ has a system of charts 
$x\mapsto (x_1,\ldots ,x_k)$ into $\m R^k$ 
for which the transition functions are of class $\mc C^{2}$.
These manifolds will be endowed with a $\mc C^1$ Riemannian metric $g$.
This means that in any $\mc C^2$ chart, the functions 
$g(\tfrac{\partial}{\partial x_i},\tfrac{\partial}{\partial x_j})$
are continuously differentiable.

In general, on such a Riemannian manifold, there might exist 
two different geodesics 
which are tangent at the same point (see \cite{Hartman50}
for an example 
with a $\mc C^{1,\al}$-Riemannian metric).
The following lemma tells us that 
this kind of examples will not occur here since we are dealing only with
${\rm CAT}(0)$-spaces whose curvature is bounded below.
Note that, since the metric tensor is not assumed to be 
twice differentiable, 
the expression ``curvature bounded below'' refers to 
the definitions in Section \ref{secgrohau}.

\begin{Def}
\label{defhadlip}
By a $\mc C^{2}$ Hadamard manifold with a $\mc C^1$ metric, we mean
a $\mc C^{2}$ manifold endowed with a $\mc C^1$ Riemannian metric  which is
${\rm CAT}(0)$ and complete. 
\end{Def}

\subsubsection{Exponential map}

\bl
\label{lemhadlip}
Let $X$ be a $\mc C^{2}$ Hadamard manifold with a $\mc C^1$ metric
of bounded curvature.\\
$a)$ For all $x$ in $X$ and $v$ in $T_xX$ there is a unique geodesic
$t\mapsto {\rm exp}_x(tv)$ starting from $x$ at speed $v$.
This geodesic is of class $\mc C^2$.\\
$b)$ This exponential map induces an homeomorphism 
$\Psi:TX\xrightarrow{\sim}X\times X$ given by,
$\Psi(x,v)=(x,{\rm exp}_x(v))$ for  $x$ in $X$ and $v$ in $T_xX$.
\el

\begin{proof} 
This lemma looks very familiar. But, since
the Christoffel coefficients might not be Lipschitz continuous,
we cannot apply Cauchy--Lipschitz theorem on Existence and Uniqueness of
solutions of differential equations. 

$a)$ 
Since the Christoffel coefficients are continous, 
we can apply Peano--Arzela theorem. 
It tells us that there exists at least one geodesic of class $\mc C^2$
starting from $x$ at speed $v$.  Uniqueness follows 
from the lower bound on the curvature.
 
$b)$ Since $X$ is ${\rm CAT}(0)$, the map $\Psi$ is a bijection.
Since a uniform limit of geodesic on $X$ is also a geodesic,
the map $\Psi$ is continuous. This map $\Psi$ is also proper, therefore
it is an homeomorphism.
\end{proof}

\subsubsection{Geodesic interpolation of $h_0$ and $h_1$}
In the sequel of this section we prove a few technical properties 
of the interpolation $h_t$ of two equidistant Lipschitz maps $h_0$ and $h_1$
with values in a Hadamard manifold (lemmas \ref{lemgeoint}). 
In Section \ref{secequhar}, we will apply this lemma 
to two equidistant harmonic maps $h_0$ and $h_1$ 
obtained by a limit process.
This lemma \ref{lemgeoint} will be used to compare the energy
of $h_0$ and $h_1$ 
with the energy of some small perturbations of $h_0$ and $h_1$.
However, in this section \ref{secliprie}, we do not need to assume $h_0$ and $h_1$ to be harmonic.
Here are the precise assumptions and notations for Lemma \ref{lemgeoint}.

Let $X$ be a $\mc C^2$ Riemannian manifold with a $\mc C^1$ metric and 
$Y$ be a $\mc C^2$ Hadamard manifold with $\mc C^1$ metric.
Let $h_0,h_1:X\ra Y$ be two $\mc C^1$ maps such that
one has
\begin{equation}
d(h_0(x),h_1(x))=1
\;\; \mbox{\rm for all $x$ in $X$.}
\end{equation}
Since $Y$ is a Hadamard manifold, there exists a unique map
\begin{eqnarray}
\label{eqnhtxhtx}
h:[0,1]\times X&\ra&Y\\
\nonumber
(t,x)&\mapsto&h(t,x)=h_t(x)
\end{eqnarray}
such that, for all $x$ in $X$, the path $t\mapsto h_t(x)$ is the unit speed geodesic 
joining $h_0(x)$ and $h_1(x)$. 
This map $h$ is called the {\it geodesic interpolation} of $h_0$ and $h_1$.
By convexity of the distance function, $h$ is Lipschitz continuous.
Therefore, by Rademacher's theorem,
the map $h$ is differentiable on a subset  of full measure  (with respect to the 
Riemannian measure on $X$). 
In particular, there exists a subset $X'\subset X$ of full measure
such that, for all $x$ in $X'$, the map $h$ is differentiable at $(x,t)$ for almost all $t$ in $[0,1]$.
In particular, for all tangent vector $V\in T_xX$ at a point $x\in X'$, the following
derivative 
\begin{equation}
\label{eqnjvtdhv}
t\mapsto J_V(t):= D_xh_t(V)\in T_{h_t(x)}Y
\end{equation}
is well-defined for almost all $t$ in $[0,1]$.
Such a  measurable vector field $J_V$ on the geodesic $t\mapsto h_t(x)$ 
will be called a {\it Jacobi field}.
We denote by 
\begin{equation}
\label{eqntxtdht}
t\mapsto \tau_x(t):= \partial_th_t(x)\in T_{h_t(x)}Y
\end{equation}
the unit tangent vector to the geodesic $t\mapsto h_t(x)$.

\bl 
\label{lemgeoint}
We keep these assumptions and notations. Let $x$ be a point in $X'$ and $V\in T_xX$.\\
$a)$  There exists a constant $\al_V\in \m R$ such that 
\begin{equation}
\langle J_V(t),\tau_x(t)\rangle=\al_V
\; ,\;\;\mbox{for all $t$ in $[0,1]$ where $J_V(t)$ is defined.} 
\end{equation}  
$b)$ There exists a  convex function $t\mapsto \ph_V(t)$ on $[0,1]$ such that 
\begin{equation}
\ph_V(t)=\| J_V(t)\|
\; ,\;\;\mbox{for all $t$ in $[0,1]$ where $J_V(t)$ is defined.} 
\end{equation}  
$c)$ The function  $\psi_V:= (\ph_V^2-\al_V^2)^{1/2}$
is also a convex function on $[0,1]$.
\el

\begin{proof} 
When $Y$ is a $\mc C^\infty$ Hadamard manifold, 
the vector field $J_V$ is a classical Jacobi field and this lemma is well known. 
Indeed, the function $\psi_V$ is the norm of the orthogonal component $K_V$ 
of the Jacobi field $J_V$,
and Inequality \eqref{eqnpsvcon} follows from the Jacobi equation satisfied by this Jacobi field $K_V$. 
We now explain how to adapt the classical proof when $Y$ is only assumed to be a
$\mc C^2$ Hadamard manifold with a $\mc C^1$ metric.

$a)$ Since the path $t\mapsto h_t(x)$ is a unit speed geodesic, one has the equality
$d(h_s(x),h_t(x))=|t-s|$ for all $s$, $t$ in $[0,1]$. 
Differentiating this equality  gives, 
when $J_V(s)$ and $J_V(t)$ are defined,
$$
\langle J_V(s),\tau_x(s)\rangle = \langle J_V(t),\tau_x(t)\rangle\, .
$$
Hence this scalar product is almost surely constant.

$b)$ Let $c$ be a $\mc C^1$ curve 
$c:[-\eps_0,\eps_0]\ra X$ with $c(0)=x$ and $\partial _sc(0)=V$.
Since the space $Y$ is ${\rm CAT}(0)$, when $s>0$, the 
 functions 
$$
t\mapsto \ph_s(t):=\frac{1}{s}d(h_t(c(0)),h_t(c(s))
$$ 
are convex on $[0,1]$.
Let $S_V:=\{ t\in [0,1]\mid J_V(t)\;\mbox{\rm is defined}\}$. This set $S_V$ has full measure
and contains the endpoints $0$ and $1$.
For all $t$ in this set $S_V$, one can compute the limit of these functions 
$\lim_{s\ra 0}\ph_s(t)=\|J_V(t)\|$. 
Since these functions $\ph_s$ are convex,  
the limit $\ph_V(t):=\lim_{s\ra 0}\ph_s(t)$ exists for all $t$ in $[0,1]$ and is a convex function.

$c)$ We slightly change the parametrization of the geodesic interpolation~:
the function 
$k:
(t,s)
\mapsto 
k_t(s):=h_{t-s\al_V}(c(s))
$ 
is well defined when $t-s\al_V$ is in $[0,1]$,
and the paths $t\mapsto k_t(s)$ are also unit speed geodesics.
Hence, for almost all $t$ in $[0,1]$,
the vector field 
\begin{equation}
\label{eqnkvt}
t\mapsto K_V(t):= \partial_sk_t(0)\in T_{k_t(0)}Y
\end{equation}
is well-defined and one has the orthogonal 
decomposition 
\begin{equation*}
J_V(t)=K_V(t)+\al_V \tau_x(t).
\end{equation*}
In particular, one has the equality, 
\begin{eqnarray}
\label{eqnpsivkv}
\psi_V(t)=\| K_V(t)\|\, .
\end{eqnarray}
The same argument as in $b)$ with the Jacobi field $K_V$ proves that the function 
$\psi_V$ is also convex.
\end{proof}

\subsubsection{Geodesic interpolation in negative curvature}
The following Lemma \ref{lemgeointneg}  improves Lemma \ref{lemgeoint} when the curvature of $Y$ is uniformly negative.
Indeed, it tells us that 
the norm $t\mapsto \psi_V(t)$ of the Jacobi field $K_V$
is uniformly convex.

\bl 
\label{lemgeointneg}
We keep the assumptions and notations of Lemma \ref{lemgeoint}. 
Moreover we assume that $Y$ is a ${\rm CAT}(-a^2)$-space with $a>0$.
Then the function $\psi_V$ satisfies the following uniform convexity property, 
\begin{equation}
\label{eqnpsivtl}
\psi_V(t)
\leq \tfrac{\sinh(a(1-t))}{\sinh(a)}\psi_V(0)+
\tfrac{\sinh(at)}{\sinh(a)}\psi_V(1)
\;\;\mbox{ for all $t$ in $[0,1]$.}
\end{equation}
\el

\begin{Rem}
One can reformulate \eqref{eqnpsivtl}
as the following inequality between positive measures
\begin{equation*}
\label{eqnpsvcon}
\tfrac{d^2}{dt^2}\psi_V\geq a^2\psi_V\, .
\end{equation*} 
\end{Rem}

\begin{proof}
This inequality \eqref{eqnpsivtl}  will follow 
from an upper bound for the norm of the Jacobi field $t\mapsto K_V(t)$
by the norm of a well chosen Jacobi field $t\mapsto\ol{K}(t)$
along a geodesic segment 
in the hyperbolic plane of curvature  $-a^2$.
Here are the details of the construction of this 
Jacobi field $t\mapsto \ol{K}(t)$.

Using a slight rescaling, 
we can assume without loss of generality  that the geodesics $t\mapsto k_t(s)$ 
are defined for $t$ in $[0,1]$ and that
the 
Jacobi field $ K_V(t)$ is well defined for $t=0$ and for $t=1$.
We choose $s>0$. Later on we will let $s$ go to $0$.
We set $P_t:=k_t(0)$ and $Q_{s,t}:= k_t(s)$,
and we apply Reshetnyak Lemma \ref{lemreshet} to the four points 
$P_0$, $P_1$,  $Q_{s,1}$, $Q_{s,0}$.
According to this lemma, there exists
a convex quadrilateral $\ol{C}_s$ in the hyperbolic 
plane $\ol{Y}$ of curvature $-a^2$ with vertices 
$\ol{P}_0$, $\ol{P}_1$,
$\ol{Q}_{s,1}$, $\ol{Q}_{s,0}$, 
and 
a $1$-Lipschitz map $j:\ol{C}_s\ra Y$ whose restriction to each of the four geodesic sides 
$\ol{P}_0\ol{P}_1$, $\ol{P}_1\ol{Q}_{s,1}$, 
$\ol{Q}_{s,1}\ol{Q}_{s,0}$, $\ol{Q}_{s,0}\ol{P}_0$
is an isometry onto each of the four geodesic 
segments $P_0P_1$, $P_1Q_{s,1}$, $Q_{s,1}Q_{s,0}$, $Q_{s,0}P_0$. 
Indeed, since $d(\ol{P}_0,\ol{P}_1)=1$, we can assume that 
the two vertices $\ol{P}_0$ and $\ol{P}_1$ do not depend on $s$
and that the quadrilateral $\ol{C}_s$ is positively oriented.

Since the vectors $K_V(0)$ and $K_V(1)$ are orthogonal to the geodesic segment  $t\mapsto k_t(0)$, by Lemma \ref{lemhadlip}, each of  the four successive angles $\th_i$ (for $i=1,\ldots ,4$ )
between the four successive geodesic segments 
$P_0P_1$, $P_1Q_{s,1}$, $Q_{s,1}Q_{s,0}$, $Q_{s,0}P_0$
in $Y$ is equal to $\frac{\pi}{2}+o(1)$, where $o(1)$ denotes a quantity that goes to $0$ when $s$ goes to $0$. 
Since $j$ is 1-Lipschitz,  each of the corresponding four successive angles $\ol{\th}_i$ between the four successive 
geodesic sides $\ol{P}_0\ol{P}_1$, $\ol{P}_1\ol{Q}_{s,1}$, 
$\ol{Q}_{s,1}\ol{Q}_{s,0}$, $\ol{Q}_{s,0}\ol{P}_0$ in the hyperbolic plane $\ol{Y}$
is not smaller than $\th_i$. 
Since the sum of these four angles $\ol{\th}_i$ 
is bounded above by $2\pi$,
each of these four angles $\ol{\th}_i$ 
also satisfies when $s$ goes to $0$~:
\begin{equation}
\label{eqnolthif}
\ol{\th}_i=\frac{\pi}{2}+o(1)\, .
\end{equation}

Denote by $t\mapsto \ol{P}_t$ and $t\mapsto \ol{Q}_{s,t}$ 
the unit speed parametrizations of the sides $\ol{P}_0\ol{P}_1$ and
$\ol{Q}_0\ol{Q}_1$. For  $t$ in $[0,1]$, one has  $j(\ol{P}_t)=P_t$
and $j(\ol{Q}_{s,t})= Q_{s,t}$, and also
\begin{equation}
\label{eqndptqst}
d(P_t,Q_{s,t}) \leq d(\ol{P}_t,\ol{Q}_{s,t})\, 
\end{equation} 
with equality when $t=0$ or $1$~:
\begin{equation}
\label{eqndp0qs0}
d(P_0,Q_{s,0}) = d(\ol{P}_0,\ol{Q}_{s,0})
\;\; \mbox{\rm and}\;\;\;
d(P_1,Q_{s,1}) = d(\ol{P}_1,\ol{Q}_{s,1})
\, .
\end{equation} 
We now focus  on these convex quadrilaterals $\ol{C}_s$ 
in the hyperbolic plane $\ol{Y}$ of curvature $-a^2$. 
We write $\ol{Q}_{s,t}={\rm exp}_{\ol{P}_t}(s\ol{K}_{s,t})$
where $\ol{K}_{s,t}$ belongs to $T_{\ol{P}_t}\ol{Y}$.
Since $K_V(0)$ and $K_V(1)$ are well defined, by   \eqref{eqnkvt}, \eqref{eqnpsivkv}, \eqref{eqnolthif} and \eqref{eqndp0qs0},
the limits 
\begin{equation*}
\ol{K}(0)=\lim_{s\ra 0} \ol{K}_{s,0}
\;\;{\rm and }\;\;
\ol{K}(0)=\lim_{s\ra 0} \ol{K}_{s,1}
\end{equation*}
exist and satisfy
\begin{equation}
\label{eqnk0psiv}
\|\ol{K}(0)\|=\psi_V(0)
\;\;{\rm and }\;\;
\|\ol{K}(1)\|=\psi_V(1)\, .
\end{equation}
Therefore,  the limits
\begin{equation*}
\ol{K}(t)=\lim_{s\ra 0} \ol{K}_{s,t}
\end{equation*}
exist for all $t$ in $[0,1]$. Moreover, by \eqref{eqnkvt}, \eqref{eqnpsivkv} and \eqref{eqndptqst}, they satisfy the inequalities
\begin{equation}
\label{eqnktpsiv}
\psi_V(t)\leq \|\ol{K}(t)\|\, .
\end{equation}
Since the vector field $t\mapsto \ol{K}(t)$ is a Jacobi field
along the geodesic segment $t\mapsto \ol{P}_t$, which is orthogonal
to the tangent vector, 
its norm 
\begin{equation*}
\ol{\psi}(t):=\|\ol{K}(t)\|
\end{equation*}
satisfies the Jacobi differential equation
\begin{equation*}
\frac{d^2}{dt^2}\ol{\psi}= a^2\ol{\psi}\, .
\end{equation*}
Hence, one has the equality
\begin{equation}
\label{eqnpsibar}
\ol{\psi}(t)
= \tfrac{\sinh(a(1-t))}{\sinh(a)}\ol{\psi}(0)+
\tfrac{\sinh(at)}{\sinh(a)}\ol{\psi}(1)
\;\;\mbox{ for all $t$ in $[0,1]$.}
\end{equation}
We now deduce Inequality \eqref{eqnpsivtl}
directly from  \eqref{eqnk0psiv},
\eqref{eqnktpsiv} and \eqref{eqnpsibar}.
\end{proof}

We have used the following existence result for a majorizing
quadrilateral due to Reshetnyak in \cite{Reshetnyak68}. 
More precisely we have used the  boundary of
this majorizing quadrilateral $\ol{C}$. 

\bl
\label{lemreshet}
Let $Y$ be a ${\rm CAT}(-a^2)$ metric space and $\ol{Y}$
be the hyperbolic plane of curvature $-a^2$.
Then, for every four points $P_0$, $P_1$, $Q_1$ $Q_0$ in $Y$ 
there exists a convex quadrilateral $\ol{C}$ in $\ol{Y}$
with vertices $\ol{P}_0$,$\ol{P}_1$, $\ol{Q}_1$,$\ol{Q}_0$
and a $1$-Lipschitz map $j:\ol{C}\ra Y$ which is an isometry 
on each of the four geodesic sides of $\ol{C}$, and which sends 
each of these four vertices $\ol{R}_i$ 
on the corresponding given point $R_i$ in $Y$.
\el

\subsection{Limits of Hadamard manifolds}
\label{seclimhad}

\bq
In this section we describe the Gromov--Hausdorff 
limits of pinched Hadamard manifolds.
\eq

The following proposition is a variation on the Cheeger
compactness theorem.

\bp
\label{prohadhad}
Let $(X_n,p_n)_{n\geq 1}$ be a sequence of $k$-dimensional 
pointed Hadamard
manifolds with pinched curvature $-1\leq K_{X_{_n}}\leq -a^2\leq 0$.\\
$a)$ There exists a subsequence of $(X_n,p_n)$ which converges to a 
pointed proper {\rm CAT}-space $(X_\infty,p_\infty)$
with curvature  between $-1$ and $-a^2$.\\
$b)$ This space $X_\infty$ has a structure of a $\mc C^{2}$
Hadamard manifold such that the distance on $X_\infty$ comes from a $\mc C^1$ Riemannian metric.
\ep

The same proof shows that
$X_\infty$ has a structure of a $\mc C^{2,\al}$
Hadamard manifold with a $\mc C^{1,\al}$ Riemannian metric,
for every $0<\al <1$. We will not use this improvement.

Even though this proposition follows from 
\cite[Theorem 72 p. 311]{Petersen16}, we  give a sketch of proof below. 

\begin{proof}
$a)$ The assumption on the curvature of $X_n$ ensures that for 
each $R>0$, one has  uniform estimates for the volumes of balls with radius $R$ in $X_n$~:
for all $n\geq 1$ and $x$ in $X_n$, one has
$$
{\rm vol} B_{\m R^k}(O,R)
\leq
{\rm vol} B_{X_n}(x,R)
\leq
{\rm vol} B_{\m H^k}(O,R) 
\, .
$$
Therefore, for each $0<\eps<R$, there exists 
an integer $N=N(R,\eps)$ such that every ball
$B_{X_n}(p_n,R)$ of $X_n$ can be covered by $N$ balls of radius $\eps$.
Hence, according to Fact \ref{faccomcri},
there exists a subsequence of  $(X_n,p_n)$ 
which converges to a proper pointed metric space $(X_\infty,p_\infty)$.
According to Fact \ref{facgrocat}, $X_\infty$ is a {\rm CAT}-space with curvature
between $-1$ and $-a^2$. 

$b)$ It remains to check that $X_\infty$ is a $\mc C^{2}$ manifold with
a $\mc C^1$ Riemannian metric.
We isometrically imbed the converging sequence $(X_n,p_n)$
in a proper metric space $Z$ as in Fact \ref{facxpxpxp}.
We fix $r_0>0$ and $c_0>0$ as in  Lemma \ref{lemharcoo} where we introduced the harmonic coordinates, and we choose a maximal $\tfrac{r_0}{2c_0}$-separated
subset $S_\infty$ of $X_\infty$. 
For each $x_\infty$ in $S_\infty$, we choose a sequence $x_n$ of points in $X_n$ that
converges to $x_\infty$. 
By \eqref{eqndpsdps}, the harmonic charts 
\begin{equation}
\label{eqnpsixnb}
\Psi_{x_n}:\mathring{B}(x_n,\tfrac{r_0}{c_0})\ra \m R^k
\end{equation}
are uniformly bi-lipschitz. More precisely, for all $z$, $z'$ in $\mathring{B}(x_n,\tfrac{r_0}{c_0})$,
one has
$$
c_0^{-1}d(z,z')\leq \|\Psi_{x_n}(z)-\Psi_{x_n}(z')\|\leq c_0\, d(z,z').
$$
Hence after extracting a subsequence, this sequence of charts $\Psi_{x_n}$ converges
towards a bi-lipschitz map 
\begin{equation}
\label{eqnpsixib}
\Psi_{x_\infty}:\mathring{B}(x_\infty,\tfrac{r_0}{c_0})\ra \m R^k.
\end{equation}
The extraction can be chosen simultaneously 
for all the points $x_\infty$ in the countable set  $S_\infty$.
This collection of maps $\Psi_{x_\infty}$ endows $X_\infty$ with a structure
of a Lipschitz manifold.

We now want to prove that this manifold $X_\infty$
is a $\mc C^2$ manifold. Indeed we will check that, for any $x_\infty$ and $x'_\infty$
in $S_\infty$, the transition functions
$\Ph_{x'_\infty}\circ\Ph_{x_\infty}^{-1}$ are of class $\mc C^2$.
This just follows from the fact that these transition functions 
are uniform limit on compact sets of 
the transition functions $\Ph_{x'_n}\circ\Ph_{x_n}^{-1}$ which are, by \eqref{eqnpsipsical},
uniformly bounded in the $\mc C^{2,\al}$-norm.

Finally, we check that the distance $d$ on $X_\infty$
comes from a $\mc C^1$ Riemannian metric on $X_\infty$. 
By \eqref{eqngijcal},  the Riemannian metrics 
$(g_n)_{ij}$ on $X_n$, seen as functions in the charts $\Psi_{x_n}$ of $X_n$, 
are uniformly bounded in the $\mc C^{1,\al}$-norm.
Extracting again a subsequence, 
there exists a $\mc C^1$
Riemannian metric $(g_\infty)_{ij}$ 
in the charts $\Psi_{x_\infty}$ of $X_\infty$
such that
the  sequence of metrics
\begin{equation}
\label{eqngijgij}
\mbox{\rm
$(g_n)_{ij}$ converges  to $(g_\infty)_{ij}$ in the $\mc C^1$ topology.}
\end{equation}
Let $d_\infty$ be the distance on $X_\infty$ associated with $g_\infty$. We check that $d_\infty=d$
on $X_\infty$. 
Let $x'_\infty$ and $x''_\infty$ be points in $X_\infty$.
They are limits of points $x'_n$ and $x''_n$ in $X_n$.
Let $c_n$ be the geodesic segment joining $x'_n$ to $x''_n$.
Extracting once more a subsequence, 
the curves $c_n$ converge uniformly to a 
curve joining $x'_\infty$ and $x''_\infty$.
This curve must be a geodesic for $g_\infty$.
This proves that $d_\infty(x'_\infty,x''_\infty)=d(x'_\infty,x''_\infty)$.
\end{proof}

\subsection{Convergence of harmonic maps}
\label{secconhar}

\bq
We now explain how to obtain the limit harmonic maps.
\eq

We first notice that we can extend Definition \ref{defharmonic}~: 
A  $\mc C^2$ map $h:X\ra Y$  between two $\mc C^2$ Riemannian manifolds with $\mc C^1$ metrics $X$ and $Y$
is said to be {\it harmonic} if  its tension field  is zero, namely $\tau(h):= {\rm tr} D^2h=0$.  
Indeed, the tension field
of a $\mc C^2$ map $h$ at a point $x$ depends only on the $2$-jet of $h$ and on the $1$-jet of the metrics of $X$ and $Y$  at the points $x$ and $h(x)$. 
More precisely, writing $h$ in a coordinate system $h:(x_1,\ldots,x_k)\mapsto (h_1,\ldots, h_k')$, the 
equation  ${\rm tr} D^2h=0$ reads as

\begin{equation}
\label{eqnharmonic3}
\Delta h_\la =
- \sum_{ij\mu\nu}g^{ij}\Gamma^\lambda_{\mu\nu}
\frac{\partial h_{\mu}}{\partial x_i}
\frac{\partial h_{\nu}}{\partial x_j} \qquad (\lambda\leq k' )
\end{equation}
where 
$\Gamma^\lambda_{\mu\nu}$ are the Christoffel coefficients on $Y$
and where  $\Delta$ is the Laplace operator on $X$ defined as in
\eqref{eqnlaplacian}~:
\begin{equation}
\label{eqnlaplacian2} 
\Delta:\ph\mapsto \tfrac{1}{v}\textstyle
\tfrac{\partial}{\partial x_i}(v\, g^{ij}
\tfrac{\partial \ph}{\partial x_j})
\end{equation}
where $v=\sqrt{\det(g_{ij})}$ denotes the volume density on $X$. 
See \cite[Section 1.3]{Jost84} for more details.

\bl
\label{lemconhar}
Let $(X_n,p_n)_{n\geq 1}$ and $(Y_n,q_n)_{n\geq 1}$ be two sequences of 
equi\-dimensional pointed Hadamard manifolds with curvature between $-1$ and $0$.
Let $c,C>0$ and let $h_n:X_n\ra Y_n$ be a sequence of $(c,C)$-quasi-isometric harmonic maps
such that $\sup_n d(h_n(p_n),q_n)<\infty$.
After extracting a subsequence, 
the sequences of pointed metric spaces $(X_n,p_n)$ and $(Y_n,q_n)$ converge respectively 
to pointed $\mc C^2$ manifolds with $\mc C^1$ Riemannian metrics $(X_\infty,p_\infty)$ and $(Y_\infty, q_\infty)$, and
the sequence of maps $h_n$ converges to a $c$-quasi-isometric  map
$h_\infty:X_\infty\ra Y_\infty$. 
This map $h_\infty$ is of class $\mc C^2$ and is harmonic.
\el

\begin{proof}
Since they are harmonic, the maps $h_n$ are $\mc C^\infty$.
Since these maps are also 
$(c,C)$-quasi-isometric,
according to Cheng's Lemma \ref{lemcheng}, 
there exists some constant $C'>0$
such that the maps $h_n$ are $C'$-Lipschitz.
The first two statements then follow from Proposition  \ref{prohadhad}
and Lemma \ref{lemcomcri}. 

It remains to show that the limit map $h_\infty$ is of class $\mc C^2$ and harmonic.
The key point will be a uniform bound for the $\mc C^{2,\alpha}$-norm of $h_n$ 
in suitable harmonic coordinates.
Let $k:=\dim X_n$ and $k':=\dim Y_n$.
Let $x_\infty$ be a point in $X_\infty$ and $y_\infty:=h_\infty(x_\infty)$.
Let $x_n$ be a sequence in $X_n$ converging to $x_\infty$ and let 
$y_n:=h_n(x_n)$. 

We look at the maps $h_n$ through the harmonic charts
$\Psi_{x_n}$ of $X_n$ and 
$\Psi_{y_n}$ of $Y_n$
as in \eqref{eqnpsixnb}. By \eqref{eqnpsixib},
these charts converge respectively to charts
$\Psi_{x_\infty}$ of $X_\infty$ and 
$\Psi_{y_\infty}$ of $Y_\infty$.
By \eqref{eqngijgij}, in these charts, 
the Riemannian metrics of  $X_n$ and $Y_n$ 
converge to the Riemannian metrics
of $X_\infty$ and $Y_\infty$ in the $\mc C^{1,\al}$-norm.

Let $0<\alpha <1$.
When one writes Equation \eqref{eqnharmonic3} for $h=h_n$ in these harmonic coordinates 
on a small open ball $\Omega:=\mathring{B}(0,\frac{r_0}{c_0C'})$ of $\m R^k$ that does not depend on $n$, one gets
\begin{equation}
\label{eqnharmonic4}
\sum_{ij}g^{i j} \frac{\partial^2 h_{\lambda}}{\partial z_i \partial z_j}
=- \sum_{ij\mu\nu}g^{ij}\Gamma^\lambda_{\mu\nu}
\frac{\partial h_{\mu}}{\partial z_i}
\frac{\partial h_{\nu}}{\partial z_j}.
\end{equation}
The  coefficients of this equation depend on $n$, but 
Lemma \ref{lemharcoo} ensures that they are uniformly bounded in the $\mc C^{\al}$-norm.
The Schauder estimates  
for  functions $u$ on $\Omega$ and compact sets $K$ of $\Omega$ 
as in \cite[Theorem 70 p. 303]{Petersen16}  thus tell us that 
\begin{equation}
\label{eqnschac1}
\|u\|_{\mc C^{1,\al},K}\leq M\, (\|\Delta u\|_{\mc C^{0},\Omega}+\|u\|_{\mc C^{\al},\Omega})
\end{equation}
\begin{equation}
\label{eqnschac2}
\|u\|_{\mc C^{2,\al},K}\leq M\, (\|\Delta u\|_{\mc C^{\al},\Omega}+\|u\|_{\mc C^{\al},\Omega})
\end{equation}
for some constant $M=M(k,\Omega,K)$.
Therefore, since the maps $h_n$ are $C'$-Lipschitz, combining \eqref{eqnharmonic3}, 
\eqref{eqnschac1}  and \eqref{eqnschac2} yields a uniform bound for the $\mc C^{2,\al}$-norm of the maps $h_n$.
Hence the Ascoli Lemma ensures that, after extracting a subsequence, 
the sequence of maps $h_n$ converges towards a $\mc C^2$ map in the $\mc C^2$ topology.
This proves that the limit map $h_\infty$ is $\mc C^2$ and is harmonic.  
\end{proof}

\subsection{Construction of the limit equidistant harmonic maps}
\label{secconequ}

\bq
We now explain why the limit harmonic maps $h_{0,\infty}$ and $h_{1,\infty}$ constructed in 
the strategy of Proposition \ref{prouni} are equidistant.
\eq

We first sum up 
the construction of these limit maps.

We start with two Hadamard manifolds $X$, $Y$ 
of bounded curvatures,
and  with two distinct quasi-isometric  harmonic maps  $h_0, h_1:X\ra Y$
such that $\de:=d(h_0,h_1)$ is finite
and non-zero.
We choose a sequence of points $p_n$ in $X$ 
such that $
d(h_0(p_n),h_1(p_n))$
converges to $\de$
and we set $q_{0,n}:=h_0(p_n)$ and $q_{1,n}:=h_1(p_n)$. 
We will frequently replace this sequence by subsequences
without mentioning it.
By Proposition \ref{prohadhad}, there exist 
two $\mc C^{2}$ Hadamard manifolds
with  $\mc C^1$ metrics $(X_\infty,p_\infty)$ and $(Y_\infty,q_{0,\infty})$ which are the Gromov--Hausdorff limits
of the pointed metric spaces
$(X,p_n)$ and $(Y,q_{0,n})$. These limit Hadamard manifolds also 
have bounded curvature.
We denote by $q_{1,\infty}$ the limit in $Y_\infty$
of the sequence $q_{1,n}$.
By the Cheng Lemma \ref{lemcheng}, 
the harmonic quasi-isometric  maps $h_0$
and $h_1$ are Lipschitz continuous.
By Lemma \ref{lemcomcri}, there exists a limit map  $h_{0,\infty}: (X_\infty,p_\infty)\ra (Y_\infty,q_{0,\infty})$
of the sequence of Lipschitz continuous maps 
$h_0:(X,p_n)\ra (Y,q_{0,n})$. There exists also  
a limit map  $h_{1,\infty}: (X_\infty,p_\infty)\ra (Y_\infty,q_{1,\infty})$
of the sequence of Lipschitz continuous maps 
$h_1:(X,p_n)\ra (Y,q_{1,n})$.
By Lemma \ref{lemconhar}, 
these limit maps $h_{0,\infty}$ and $h_{1,\infty}$ are still harmonic quasi-isometric maps.

\bl
\label{lemlimequ}
With the above notation, the two limit harmonic quasi-iso\-metric maps 
$h_{0,\infty}$, $h_{1,\infty}$ 
are equidistant. More precisely, 
for all $x$ in $X_\infty$, one has 
$d(h_{0,\infty}(x),h_{1,\infty}(x))=\de>0$
where $\de:=d(h_0,h_1)$.
\el

We will apply this lemma to two pinched Hadamard manifolds $X$, $Y$.
In this case, the limit $\mc C^2$ Hadamard manifolds 
$X_\infty$, $Y_\infty$ will also be pinched.

\begin{proof}
Let $\Delta_\infty$ be the Laplace operator on $X_\infty$
defined as in  \eqref{eqnlaplacian2}.
We first check that the  function
$
\ph_\infty:x\mapsto d(h_{0,\infty}(x),h_{1,\infty}(x))
$
is subharmonic on $X_\infty$. This means that 
$\Delta_\infty\ph_\infty$ is a positive measure on $X_\infty$.
Assume first that the Riemannian metric on $Y_\infty$
is $\mc C^\infty$.
In this case, $\ph_\infty$ is the composition of a harmonic map 
$h=(h_0,h_1): X_\infty\ra Y_\infty\times Y_\infty$ and of
a convex $\mc C^\infty$-function $F=d:Y_\infty\times Y_\infty\ra \m R$,
so that the function $\ph_\infty$ is subharmonic on $X_\infty$
because of
the formula
\begin{equation*}
\label{eqndefh}
\Delta_\infty (F\circ h) =\sum_{1\leq i\leq k} D^2F(D_{e_i}h,D_{e_i}h)+ \langle DF ,\tau(h)\rangle\, ,
\end{equation*}
where 
$(e_i)_{1\leq i\leq k}$ is an orthonormal basis of the tangent space to $X$.

Since the Riemannian metric on $Y$
might not be of class $\mc C^\infty$, we 
will use instead a limit argument.
We fix a point $x_\infty$ in $X_\infty$.
In a chart $(x_1,\ldots,x_k)$, the Laplace operator 
$\Delta_\infty$ of 
the Riemannian metric $(g_\infty)_{ij}$ of $X_\infty$ reads  
as 
\begin{equation}
\label{eqndelpsi}
\psi\mapsto \Delta_\infty\psi=\frac{1}{v_\infty}
\frac{\partial}{\partial x_i}(v_\infty g_\infty^{ij}\frac{\partial\psi}{\partial x_j})\, ,
\end{equation}
where $v_\infty$ is the volume density.
We want to prove that 
for every 
$\mc C^2$ function $\psi\geq 0$ with compact support in 
a small neighborhood of $x_\infty$,
one has 
\begin{equation}
\label{eqnphdeps}
\int_{\m R^k} \ph_\infty \, \Delta_\infty\psi \, v_\infty {\rm d} x
\geq 0\, .
\end{equation}  

This function $\ph_\infty$ on the pointed metric space
$(X_\infty,p_\infty)$
is equal to the limit of the sequence of functions 
$
\ph_n:x\mapsto d(h_{0}(x),h_{1}(x))
$ 
on the pointed metric spaces $(X,p_n)$,
as defined in Lemma \ref{lemcomcri}.
Note that the dependence on $n$ comes from the 
base point $p_n$ which moves with $n$.  
We choose  a sequence $x_n$ in  $X_n$ converging to $x_\infty$. 
As in the proof of Lemma \ref{lemconhar},
we look at the functions $\ph_n$ through the harmonic charts
$\Psi_{x_n}$ of $X_n$. By \eqref{eqnpsixib},
these charts converge to a chart
$\Psi_{x_\infty}$ of $X_\infty$.
By \eqref{eqngijgij}, in these charts $(x_1,\ldots, x_k)$
the Riemannian metric $(g_n)_{ij}$ of  $X_n$ 
converge to the Riemannian metric $(g_\infty)_{ij}$
of $X_\infty$ in $\mc C^1$ topology.

Since, by the above argument, the functions $\ph_n$ are subharmonic for the metric $(g_n)_{ij}$, for every 
$\mc C^2$ function $\psi\geq 0$ with compact support in these charts,
one has at each step $n$
\begin{equation}
\label{eqnphnden}
\int \ph_n \Delta_n\psi\, v_n{\rm d} x\geq 0 
\end{equation}
where $\Delta_n$ and $v_n$ are the Laplace operator 
and the volume density of the metric $(g_n)_{ij}$.
Letting $n$ go to $\infty$ in  \eqref{eqnphnden} gives \eqref{eqnphdeps}.
This proves that the function $\ph_\infty$ is subharmonic.

By construction, this
subharmonic function $\ph_\infty$ on $X_\infty$
achieves its maximum $\delta>0$ at the point $p_\infty$. 
By \eqref{eqndelpsi}, the Laplace operator 
is an elliptic linear differential operator
with continuous coefficients.
Hence, by the strong  maximum principle  in \cite[Theorem 8.19 p.198]{GilbargTrudinger}, 
this function $\ph_\infty$ is constant and equal to $\delta$.
\end{proof}

The aim of Sections \ref{secequhar} 
and \ref{secequneg} is to prove that such 
equidistant harmonic quasi-isometric maps 
$h_{0,\infty}$ and $h_{1,\infty}$ can not exist
(Corollary \ref{cortwohar}) when $Y_\infty$ is pinched.
This will conclude the proof of Proposition \ref{prouni}.

\subsection{Equidistant harmonic maps}
\label{secequhar}

\bq
We first study equidistant harmonic maps without any pinching assumption.
\eq

The following lemma  \ref{lemtwohar} extends \cite[Lemma 2.2]{LiWang98}
to the case where the source space $X$ 
is only assumed to be a $C^2$-Hadamard manifold.

\bl
\label{lemtwohar} 
Let $X$, $Y$ be two  $\mc C^{2}$ Hadamard manifolds
with $\mc C^1$ Riemannian metrics of bounded curvature.
Let $h_0, h_1:X\ra Y$ be two harmonic maps 
such that the distance function $x\mapsto d(h_0(x),h_1(x))$
is constant. For $t$ in $[0,1]$, let $h_t$ be the geodesic interpolation of $h_0$ and $h_1$ as in \eqref{eqnhtxhtx}.
Then for almost all $x$ in $X$, $t$ in $[0,1]$ 
and $V$ in $T_xX$, one has 
\begin{equation}
\label{eqndhv0t1}
\|Dh_0(V)\|=\| Dh_t(V)\|=\| Dh_1(V)\|\, .
\end{equation}
\el

Note that we can not conclude that Equality \eqref{eqndhv0t1}
is valid for all $x$ and $t$ since the interpolation $h_t$ 
might not be of class $\mc C^1$.

We will use the following straightforward inequality
for convex functions.

\bl
\label{lemphtpht}
Let $t\mapsto\Phi_t$  be a non-negative convex function on $[0,1]$. 
Then, for all $t$ in $[0,\frac12]$, one has
\begin{equation}
\label{eqnphtpht}
\Ph_t+\Ph_{1-t}\leq \Ph_0 +\Ph_1-2t\,(\Ph_0+\Ph_1-2\Ph_{1/2})\, .
\end{equation}
\el

\begin{proof}[Proof of Lemma \ref{lemphtpht}]
We just add the following two convexity  inequalities
$\Ph_t\leq (1-2t)\Ph_0 +2t\Ph_{1/2}$ and
$\Ph_{1-t}\leq (1-2t)\Ph_1 +2t\Ph_{1/2}$. 
\end{proof}

\begin{proof}[Proof of Lemma \ref{lemtwohar}]
The idea is to construct two small perturbations $f_\eps$
and $g_\eps$ of the harmonic maps $h_0$ and $h_1$
with support in a compact set $K$ of $X$,
and to compare the sum of the energies of $f_\eps$ and $g_\eps$ 
with the sum of the energies of $h_0$ and $h_1$.

Let $0\leq \eps\leq 1$. 
Here is the definition of the two maps $f_\eps:X\ra Y$ and $g_\eps:X\ra Y$. We fix a $\mc C^1$ cut-off function $\eta:X\mapsto[0,1];x\mapsto \eta_x$ 
with compact support $K$, and we let for all $x$ in $X$~:
\begin{equation} 
\label{eqnfepgep}
f_\eps(x):=h_{\eps\eta_x}(x)
\;\;{\rm and}\;\;
g_\eps(x):=h_{1-\eps\eta_x}(x)\, .
\end{equation}
These functions are Lipschitz continuous, so that they are almost everywhere differentiable.
In order to compute their differentials, we use the notations 
\eqref{eqnjvtdhv} and \eqref{eqntxtdht}~:
for all $x$ in a subset $X'\subset X$ of full measure, 
for all $V$ in  $T_xX$, for almost all $t$ in $[0,1]$, we let
\begin{eqnarray*}
\label{eqnjvttxt}
J_V(t):= D_xh_t(V) 
&{\rm and}&
\tau_x(t):= \partial_th_t(x)\, .
\end{eqnarray*}
For such a tangent vector $V$, it follows from Lemma \ref{lemgeoint}.$b$ that there exists a convex function 
$t\mapsto \ph_V(t)$ such that  $\ph_V(t)= \|J_V(t)\|$
for all $t$ where the derivative $J_V(t)$ exists.
By the chain rule, for almost all $\eps$ in $[0,1]$, the differentials of $f_\eps$ and $g_\eps$
are given, for almost all $x$ in $X$ and all $V$ in $T_xX$, by 
\begin{eqnarray}
\label{eqndxfgtv}
Df_\eps(V)&=& \;\;\;\; J_V(\eps\eta_x)\;\;+\;\; \eps\, V.\eta\, \tau_{x}(\eps\eta_x),\\
Dg_\eps(V)&=& J_V(1-\eps\eta_x)- \eps\, V.\eta\, \tau_{x}(1-\eps\eta_x)
\end{eqnarray}
where $V.\eta=d\eta(V)$ is the derivative of the function $\eta$ in the direction $V$.

According to Lemma \ref{lemgeoint}.$a$, for almost all $x$ in $X$ and all $V$ in $T_xX$, the scalar product
$\langle J_V(t),\tau_x(t)\rangle$ is almost surely constant.
Therefore, for almost all $\eps$ in $[0,1]$, $x$ in $X$ 
and $V$ in the unit tangent bundle $T^1_xX$, one has the equality
\begin{equation}
\label{eqndfvdgv}
\| Df_\eps(V)\|^2+\| Dg_\eps(V)\|^2
= \ph_{V}(\eps\eta_x)^2 +\ph_{V}(1-\eps\eta_x)^2 + 2 \eps^2(V.\eta)^2\, .
\end{equation} 
We introduce the convex function $t\mapsto \Ph^V_t:=\ph_V(t)^2$.
Using Inequality \eqref{eqnphtpht}, one gets
for almost all $\eps$ in $[0,1]$, $x$ in $X$ 
and $V$ in $T^1_xX$ the bound
\begin{equation*}
\label{eqndfvbis}
\| Df_\eps(V)\|^2+\| Dg_\eps(V)\|^2
\leq \Ph^V_0 +\Ph^V_1 -
2\eps\eta_x(\Ph^V_{0}+\Ph^V_{1}-2\Ph^V_{1/2})
+ 2 \eps^2(V.\eta)^2.
\end{equation*} 

We recall that the energy over $K$ of a Lipschitz map $h:X\ra Y$
is 
$$
E_K(h):=\int_K\|D_xh\|^2 \rmd x=\int_{T^1K} \|Dh(V)\|^2\rmd V\, ,
$$
where $\rmd x$ is the Riemannian measure on $X$ and 
$\rmd V$ the Riemannian measure on $T^1X$. Integrating the previous inequality on the unit tangent bundle of $K$,
one gets the following inequality relating the energy over $K$ 
of $f_\eps$, $g_\eps$, $h_0$ and $h_1$~:
\begin{equation}
\label{eqnkfeps}
E_K(f_\eps)+E_K(g_\eps)-E_K(h_0)-E_K(h_1)\leq
-\eps\int_{T^1_K}A(V)\rmd V  +O(\eps^2)\, 
\end{equation} 
where $A$ is the function on $T^1X$
defined for almost all $x$ in $X$ and $V$ in $T_x^1X$, by 
$$
A(V):=2\eta_x\, (\Ph^V_{0}+\Ph^V_{1}-2\Ph^V_{1/2}).
$$ 
Since  the harmonic maps $h_0$ and $h_1$ are critical points for the energy functional,
Inequality \eqref{eqnkfeps} implies that 
\begin{equation}
\label{eqnintkav}
\int_{T^1K}A (V)\rmd V \leq 0\, .
\end{equation} 
Since the function $\Ph^V$ is convex, the function $A$ is non-negative.
Therefore Inequality \eqref{eqnintkav} implies that 
the function $A$ is almost surely zero.
Since the function $\eta$ was arbitrary, this tells us that,
for almost all $V$ in $T^1X$, one has
\begin{equation*}
2\Ph^V_{1/2}=\Ph^V_{0}+\Ph^V_{1}.
\end{equation*}
Since $\Ph^V$ is the square of the convex function $\ph_V$,
this tells us that for almost all $V$ in $TX$,
the function $\ph_V$ is constant.
This proves \eqref{eqndhv0t1}.
\end{proof}

\subsection{Equidistant harmonic maps in negative curvature}
\label{secequneg}

The following Corollary \ref{cortwohar} 
improves  the conclusion of Lemma \ref{lemtwohar}  
when the curvature of $Y$ is uniformly negative.

\bc
\label{cortwohar} Let $a>0$.
Let $X$, $Y$  be two  $\mc C^{2}$ Hadamard manifolds
with $\mc C^1$ Riemannian metrics. 
Assume moreover that $Y$ is ${\rm CAT}(-a^2)$.
Let $h_0, h_1:X\ra Y$ be two harmonic maps 
such that the distance function $x\mapsto d(h_0(x),h_1(x))$
is constant.
Then either $h_0=h_1$ or 
\begin{equation}
\label{eqnh0h1ga}
\mbox{$h_0$ and $h_1$
take their values in the same  geodecic $\Ga$ of $Y$.}
\end{equation}
\ec

This means that, when $h_0\neq h_1$, there exists a geodesic $t\ra \ga(t)$ in $Y$ and two harmonic functions $u_0$, $u_1$ on $X$ such that 
$h_0=\ga\circ u_0$, $h_1=\ga\circ u_1$
and the difference $u_1\!-\! u_0$ is a bounded harmonic function 
on $X$.

Note that this case is ruled out when $h_0$ and $h_1$ are 
within bounded distance from a quasi-isometric
 map $f:X\ra Y$ since $X$ has dimension $k\geq 2$.

\begin{proof}[Proof of Corollary \ref{cortwohar}]
We can assume that the distance 
between $h_0$ and $h_1$ is equal to $1$.
We recall a few notations that we have already used.
For $t$ in $[0,1]$, let $h_t$ be the geodesic interpolation of $h_0$ and $h_1$.
For $x$ in $X$, let $\tau_x(t):=\partial_th_t(x)$.
Since the map $(t,x)\mapsto h_t(x)$ is Lipschitz continuous,
the vector $J_V(t):=Dh_t(V)$ is well-defined 
for almost all $t$ in $[0,1]$, $x$ in $X$ and $V$ in $T_xX$.
For all such $t$, $x$, $V$, we set
\begin{equation*}
\al_V(t):= \langle J_V(t),\tau_x(t)\rangle
\, ,\;\;
\ph_V(t):=\|J_V(t)\|
\, ,\;\;
\psi_V(t):=(\ph_V(t)^2-\al_V(t)^2)^{1/2}\, .
\end{equation*} 
By Lemmas \ref{lemgeoint}.a and \ref{lemtwohar},
one has the equalities
\begin{equation}
\label{eqnalv0t1}
\al_V(0)=\al_V(t) =\al_V(1)
\;\;\;{\rm and}\;\;\;
\ph_V(0)=\ph_V(t)=\ph_V(1)\, 
\end{equation} 
for almost all $t$ in $[0,1]$ and almost all $V$ in $TX$, so that
\begin{equation*}
\psi_V(0)=\psi_V(t)=\psi_V(1)\, .
\end{equation*} 
Comparing these equalities with the uniform convexity 
of the function $\ps_V$ in
\eqref{eqnpsivtl},
one infers that
$\psi_V(t)=0$.
Hence, when $J_V(t)$ is defined, one has
\begin{equation}
\label{eqnjvtavt}
J_V(t)=\al_V(0)\,\tau_x(t)\, .
\end{equation} 

We now explain why \eqref{eqnjvtavt} implies \eqref{eqnh0h1ga}.
It is enough to check that, for every $\mc C^1$ 
curve 
$$
c:[0,1]\ra X;s\mapsto c_s
$$ 
with speed at most $1/3$, 
the images
\begin{equation}
\label{eqnh0c1ga}
\mbox{\rm $h_0( c_{[0,1]})$ 
and $h_1( c_{[0,1]})$ are both included in the geodesic $\Ga$}
\end{equation}
of $Y$
containing both $h_0(c_0)$ and $h_1(c_0)$.

The idea is to construct an auxiliary curve $C$ with zero derivative.
Let $\beta:[0,1]\ra  [-1/3,1/3]$ be the function 
given by $s\mapsto \beta_s:=\int_0^s\al_{c'_r}(0)\rmd r$.
For $t_0$ in $[1/3,2/3]$, let $C$ be the curve
$$
C:[0,1]\ra  Y; s\mapsto C(s):=h_{t_0-\beta_s}(c_s).
$$
Since the speed of $c$ is bounded by $1/3$, the curve $C$ is 
well-defined. 
By construction $C$ is a Lipschitz continuous path,
and by \eqref{eqnalv0t1} and \eqref{eqnjvtavt},  
for almost all $s$, its derivative is,
$$
C'(s)= \left(\al_{c'_s}(t_0\! -\! \beta_s)-\al_{c'_s}(0)\right)
\tau_{c_s}(t_0\! -\! \beta_s) =0
$$
Therefore, 
one has $C(s)=C(0)$ for all $s$ in $[0,1]$, that is
$$
h_{t_0-\beta_s}(c_s)=h_{t_0}(c_0)\, .
$$
Using this equality for two distinct values of  $t_0$,
we deduce that the geodesic segments $h_{[0,1]}(c_0)$
and $h_{[0,1]}(c_s)$ meet in at least two points.
This proves \eqref{eqnh0c1ga}  
and ends the proof of Corollary \ref{cortwohar}.
\end{proof}

This also ends the proof of Proposition \ref{prouni}.

\section{Boundary maps for weakly coarse embeddings}
\label{secbouweacoa}
\bq
This chapter is independent of the previous ones.
We prove that a weakly coarse embedding between pinched 
Hada\-mard manifolds admits a boundary map 
which is well-defined outside a set of zero Hausdorff dimension. 
We prove that the fibers of this boundary map 
also have zero Hausdorff dimension (Theorem \ref{thmboumap}). 
More precisely, we will prove  quantitative versions of these facts
(Propositions \ref{propca} and \ref{propcb}) that we will use in Chapter \ref{secweacoahar}.
\eq

\subsection{Weakly coarse embeddings}
\label{secweacoaemb}
\bq
In this section, we introduce various classes of 
rough Lipschitz maps $f:X\ra Y$ between pinched Hadamard manifolds
generalizing the quasi-isometric maps.
\eq

Let $X$ and $Y$ be Hadamard manifolds with pinched sectional curvatures
$
-b^2\leq K_X\leq -a^2<0
$,
$
-b^2\leq K_Y\leq -a^2<0
$.
Let $k=\dim X$, $k'=\dim Y$.

\bd 
Let $c>0$. A map $f:X\to Y$ is  rough $c$-Lipschitz if, for all $x,x'\in X$,
with $d(x,x')\leq 1$ 
one has 
$d(f(x),f(x'))\leq c \, .$
\ed 

When $f:X\to Y$ 
is a rough $c$-Lipschitz map,
one has  
for all $x$, $x'$ in $X$~: 
$$
d(f(x),f(x'))\leq c\, d(x,x') + c\, .
$$ 

\bd
\label{defcoaemb}
A map $f:X\to Y$  is a {\it coarse embedding} 
if there exist two non-decreasing unbounded functions 
$\ph_1$, $\ph_2$ such that, for all $x,x'\in X$~:
\begin{equation}
\label{eqncoaemb}
\ph_1(d(x,x'))\leq d(f(x),f(x'))\leq \ph_2(d(x,x'))\, .
\end{equation}
\ed

Note that a map which is within bounded distance 
from a coarse embedding is also a coarse embedding.
In Definition \ref{defcoaemb}
one may always assume that the function $\ph_2$ is an affine function, that is, $f$ is rough Lipschitz.
A quasi-isometric map is a special case of a coarse embedding, 
where $\ph_1$ is also an affine function.

\bd 
\label{defweacoa} 
A weakly coarse embedding $f:X\to Y$ is a rough Lipschitz map 
for which there exist $c_0,C_0>0$ such that, for all $x$, $x'$ in $X$~: 
\begin{eqnarray}
\label{eqnweacoa}
d(f(x),f(x'))\leq c_0 &\Rightarrow & d(x,x')\leq C_0 \, .
\end{eqnarray}
\ed

Equivalently, this means that 
there exist two non-decreasing non-nega\-tive and non-zero functions 
$\ph_1$, $\ph_2$ such that
\eqref{eqncoaemb} holds. 
Of course, any coarse embedding $f:X\to Y$ is a weakly coarse embedding.

\bex 
\label{exacoaemb}
There exist many coarse and weakly coarse embeddings 
$f$ from $\m H^2$ to $\m H^3$. More precisely, 
for any non-decreasing $1$-Lipschitz function
$\ph_1:[0,\infty[\ra[0,\infty[$ with $\ph_1(0)=0$
one can choose a $1$-Lipschitz map $f$ for which $\ph_1$ is the best lower bound 
in \eqref{eqncoaemb}.
\eex

\begin{proof}
Indeed, one first constructs a unit-speed $\mc C^1$ curve $f_0:\m R\to\m H^2$ 
such that $\ph_1(t)=\min\limits_{s\in \m R}d(f_0(s+t),f_0(s))$
for every $t\geq 0$.
We set $\m H^1:= \m R$ and, for $k\geq 1$, we embed each  space $\m H^{k}$ as a totally geodesic 
hyperplane in $\m H^{k+1}$
and denote by 
$x\to n_x$ a unit normal vector field to $\m H^{k}$ in $\m H^{k+1}$.
We now define the Lipschitz map $f:\m H^2\to\m H^3$
as
$f(\exp(tn_s)):=\exp(tn_{f_0(s)})$ 
for all $s$ in $\m H^1$ and $t\in \m R$.
\end{proof}

For any point $x_0\in X$ and $r>0$,  
we identify through the exponential map each sphere $S(x_0,r)$ 
with the unit tangent sphere 
$$
S_{x_0}:=\{\xi\in T_{x_0}X\mid \|\xi\|=1\}.
$$
More precisely, when $\xi\in S_{x_0}$ is a unit tangent vector 
at the point $x_0$, we  denote by 
$r\mapsto \xi_r:={\rm exp}_{x_0}(r\xi)$ 
the corresponding unit-speed geodesic ray (so that $\xi_0=x_0$).

We denote by $\ol{X}=X\cup \partial X$ the visual compactification of $X$.
The boundary $\partial X$ is the set of 
(equivalence classes of) rays in $X$. 
The map $\psi_{x_0}:\xi\mapsto \lim\limits_{r\ra\infty}\xi_r$ 
gives an homeomorphism 
of the unit tangent sphere $S_{x_0}$
with the sphere at infinity $\partial X$.
We say that a subset $A$ of $\partial X$ has 
{\it zero Hausdorff dimension} if, seen in $S_{x_0}$, 
it has zero Hausdorff dimension. 
One can check that this property 
does not depend on the choice of $x_0$,
because for any another point $x_1\in X$, the homeomorphism 
${\psi_{x_1}\!\!}^{ -1}\circ \psi_{x_0}$ is bi-H\"{o}lder.

In this Chapter \ref{secbouweacoa}
we will prove the following theorem.

\bt
\label{thmboumap}
Let $f:X\ra Y$ be a weakly coarse embedding between pinched Hadamard manifolds.\\
$a)$ There exists a subset $A\subset \partial X$ of zero Hausdorff dimension 
such that,
for all $\xi\in \partial X\!\smallsetminus\! A$, the limit 
$\partial f(\xi):=\lim\limits_{r\ra\infty}f(\xi_r)$ exists
in $\partial Y$.\\
$b)$ For every $\xi\in \partial X\!\smallsetminus\! A$, the fiber
$\{\eta\in \partial X\!\smallsetminus\! A\mid 
\partial f(\eta) = \partial f(\xi)\}$
has zero Hausdorff dimension.
\et
The map $\partial f:\partial X\!\smallsetminus\! A\ra \partial Y$
is called the boundary map of $f$.

The proof of Theorem \ref{thmboumap} will last up to the end of this Chapter.
The quantitative estimates 
\eqref{eqnhuainfaxy} and \eqref{eqnhubinfaxy} that we will obtain during this proof  
will be reused in Chapter \ref{secweacoahar}.

\subsection{Hausdorff dimension and Frostman measures}
\label{sechaudim}
\begin{quote}
In this section we introduce classical notations and 
definitions from geometric measure theory.
\end{quote}

\bd
\label{deffromea}
Let $M ,\nu >0$. A Borel probability measure $\si$ on a 
compact metric space $S$ is said to be $(M,\nu)$-Frostman
if, for all $\xi\in S$ and all $r>0$, one has the following bound for the measures of the balls~:
\begin{equation}\label{eqnmeafro}
\sigma(B(\xi,r))\leq M\, r^\nu\, .
\end{equation}
\ed
Let $\nu>0$ and $\de>0$.
For a subset $A\subset S$, 
we denote
\begin{equation*}
\label{eqndelhau}
H^\nu_\de(A)=\inf
\{\textstyle\sum\limits_{i\geq 1}{\rm diam}(U_i)^\nu\mid
A\subset \bigcup_i U_i\; ,\;\;
{\rm diam}(U_i)\leq \de
\}\,.
\end{equation*}
When $\de=\infty$, we denote similarly
\begin{equation}
\label{eqndelhauinf}
H^\nu_\infty(A)=\inf
\{\textstyle\sum\limits_{i\geq 1}{\rm diam}(U_i)^\nu\mid
A\subset \bigcup_i U_i \}\,.
\end{equation}

We recall that
the $\nu$-dimensional Hausdorff measure of $A$ is defined as 
\begin{equation*}
\label{eqnhaumea}
H^\nu(A)=\sup\limits_{\de>0}H^\nu_\de(A)\,
\end{equation*}
and the Hausdorff dimension of $A$ is
\begin{equation*}
\label{eqnhaudim}
\dim_H(A)=\inf\{\nu>0\mid H^\nu(A)=0\}\,.
\end{equation*}
Observe that one also has
\begin{equation}
\label{eqnhaudiminf}
\dim_H(A)=\inf\{\nu>0\mid H^\nu_\infty(A)=0\}\,.
\end{equation}

The following easy lemma relates $H^\nu_\infty(A)$
with Frostman measures.
 
\bl
\label{lemfro}
Let $\si$ be a $(M,\nu)$-Frostman measure on a compact metric space $S$ and $A\subset S$. Then one has
$\si(A)\leq M\, H^\nu_\infty(A)$.
\el

\begin{proof}
Observe that 
$\si(A)\leq 
\sum\limits_{i\geq 1}\si(U_i)\leq 
M\,\sum\limits_{i\geq 1}{\rm diam}(U_i)^\nu$
for any covering $(U_i)$ of $A$.
\end{proof}

\subsection{Image of a large sphere}
\begin{quote}
In this section we focus on those points of a sphere $S(x_0,r)$ whose images under a weakly coarse embedding are too close from a given point.
\end{quote}

The following definition plays a key role in the proof of Theorem \ref{thmboumap}.

\bd
\label{defpcc}
Let $c$, $C_1$, $C_2>0$.
A rough $c$-Lipschitz map $f:X\ra Y$ 
satisfies property $\mc C_{C_1,C_2}$ 
if,  for all  $\alpha>0$, $x_0\in X$, $y_0\in Y$ and $r>0$,
the set
\begin{equation}
\label{eqnp0}
A_{x_0,y_0,\al,r}:=\{\xi\in S_{x_0}\, |\, d(y_0,f(\xi_r))\leq \alpha r\}
\end{equation}
can be covered by at most $C_1e^{bk'\al r}$ 
balls of radius $C_2e^{-ar}$, where $k'=\dim Y$.

If such constants $C_1$, $C_2$ exist, we say that $f$ 
satisfies property $\mc C$.
\ed

In this definition the unit sphere $S_{x_0}$ is endowed with the distance induced 
by the Riemannian norm on $T_{x_0}X$.

The bound on the size of a covering of the set \eqref{eqnp0} will be very useful for Hausdorff dimension estimations.
The precise value $bk'\al$ for the exponential growth in Definition \ref{defpcc} is not very
important. It is the one obtained in the next proposition and it just avoids the introduction 
of another parameter.

\bp
\label{procoapcc} 
Every weakly coarse embedding 
$f:X\to Y$  satisfies property $\mc C$.
\ep

In particular, Propositions \ref{propca} and \ref{propcb} below apply  to all weakly coarse embeddings $f$.

We will use the Bishop volume estimates (see for example \cite{GHL04})
which compare the volume of balls in $X$ and in the hyperbolic space $\m H^k$.

\bl
\label{lemvolest}
Let $X$ be a pinched Hadamard manifold with dimension $k$ and sectional curvature $-b^2\leq K_X\leq -a^2<0$. Then, 
for $R>0$, one has
$$
a^{-k}{\rm vol} (B_{\m H^k}(O,aR))\leq {\rm vol} (B_X(x,R))\leq 
b^{-k}{\rm vol} (B_{\m H^k}(O,bR))\, .
$$
\el
\noindent We will also need to bound  angles 
by  Gromov products
as in Lemma \ref{lemcomparison}.

\bl  
\label{lemang}
Let $Y$ be a Hadamard manifold with curvature $K_Y\!\leq\! -a^2\!<\!0$. Then, for all $y_0\in Y$ and $y_1,y_2\in Y\!\smallsetminus\!\{y_0\}$ 
one has the bound
$$
\theta_{y_0}(y_1,y_2)\leq 4\, e^{-a (y_1,y_2)_{y_0}} ,
$$
where $\theta_{y_0}(y_1,y_2)$ is the angle at $y_0$ of the geodesic triangle
$(y_0, y_1, y_2)$ and
$(y_1,y_2)_{y_0}:=\frac 12 (d(y_0,y_1)+d(y_0,y_2)-d(y_1,y_2))\, $ 
is the Gromov product.
\el

\begin{proof}[Proof of Proposition \ref{procoapcc}] We will see that $f$ 
satisfies property $\mc C_{C_1,C_2}$
where the constants $C_1$, $C_2$ depend
only on $a$, $b$, $k'$, and on $c_0$, $C_0$
from \eqref{eqnweacoa}.

It follows from the volume estimates of Lemma \ref{lemvolest} that there exists a constant $C_1>0$ such that for each ball $B(y_0,\alpha r)\subset Y$ ($r>0$) and each covering of minimal cardinality of this ball by balls with radii $c_0/2$
$$
B(y_0,\alpha r) \subset \cup _{i\in I}B(y_i,c_0/2),
$$
this cardinality is at most $C_1 e^{bk'\alpha r}$.
 
Since $f$ is a $(c_0,C_0)$-weakly coarse embedding, for each $i \in I$, the inverse image of this ball $f^{-1}(B(y_i,c_0/2))$  is either empty or lies in a ball $B(x_i,C_0)\subset X$.
By Lemma \ref{lemang}, the intersection $B(x,C_0)\cap S(x_0,r)$  
lies in a cone with vertex $x_0$ and angle $\theta_r= C_2 e^{-ar}$.
\end{proof}

\br
Any map $\widetilde{f}:X\to Y$ within  bounded distance from
a map $f:X\to Y$ satisfying Property $\mc C$ 
also satisfies Property $\mc C$.
\er

\subsection{Construction of the boundary map}
\label{secconbou}
\begin{quote}
We now investigate the long-term behavior 
of the images of geo\-desic rays 
under a rough Lipschitz map satisfying Property $\mc C$. 
\end{quote}

Let $X$, $Y$ be pinched Hadamard manifolds
and $f:X\to Y$ be a rough Lipschitz map satisfying 
Property $\mc C$. 
The following proposition \ref{propca} tells us that,
except for a set of rays of zero Hausdorff dimension, 
the image under  $f$ of a ray goes to infinity in $Y$ 
at positive speed
and this image converges to a point
in  $\partial Y$.

We need some notations.
For $x_0\in X$, let $A_{x_0}$
be the set of rays whose image do not go to infinity at positive speed, namely
\begin{equation*}
\label{eqnaxo}
A_{x_0}:=\{ \xi\in S_{x_0}\mid 
\liminf_{n\ra\infty}\tfrac{1}{n}\,
d(f(x_0),f(\xi_n))=0\}\, .
\end{equation*}
One has $A_{x_0}=\bigcap_{\al>0}A_{x_0,\al}$, where, for $\al>0$, we denote
\begin{equation*}
\label{eqnaxa}
A_{x_0,\al}:=\{ \xi\in S_{x_0}\mid \liminf_{n\ra\infty}\tfrac{1}{n}
d(f(x_0),f(\xi_n))< \al\}\, .
\end{equation*}
One has 
$A_{x_0,\al}\subset\bigcap_{n_0\geq 1}A_{x_0,\al}(n_0)$, where, for $n_0\geq 1$, one defines
\begin{equation*}
\label{eqnaxf}
A_{x_0,\al}(n_0):=\{ \xi\in S_{x_0}\mid 
d(f(x_0),f(\xi_n))\leq n\al\; \;
\mbox{\rm for some $n\geq n_0$}
\}\, .
\end{equation*}

\bp
\label{propca}
Let $X$, $Y$ be pinched Hadamard manifolds 
with sectional curvatures $-b^2\leq K\leq -a^2<0$. 
Let  $c, C_1,C_2>0$
and $f:X\to Y$ be a rough $c$-Lipschitz map satisfying 
Property $\mc C_{C_1,C_2}$. 
Let $\al>0$, $k'=\dim Y$ and $\nu_\al:=\tfrac{bk'\al}{a}$. For  $\nu>\nu_\al$, we 
introduce the constant $C_{3,\al,\nu}:=\frac{C_1 C_2^\nu}{1-e^{-a(\nu-\nu_\al)}}$.\\
$a)$ 
For every $x_0\in X$ and $n_0\geq 1$, one has
\begin{equation}
\label{eqnhuainfaxy} 
H^\nu_\infty(A_{x_0,\al}(n_0))\leq 
C_{3,\al,\nu} \, e^{-a(\nu-\nu_\al)n_0}.
\end{equation}
$b)$ 
For every $(M,\nu)$--Frostman measure $\si$ on $S_{x_0}$, one has
\begin{equation}
\label{eqnsiaxy} 
\si(A_{x_0,\al}(n_0))\leq 
M \, C_{3,\al,\nu} \, e^{-a(\nu-\nu_\al)n_0}.
\end{equation}
$c)$ One has
$\dim_H(A_{x_0,\al})\leq \nu_\al$.\\
$d)$ One has
$\dim_H(A_{x_0})=0$.\\
$e)$ For every $\xi\in S_{x_0}\!\smallsetminus\! A_{x_0}$, 
the limit $\partial f(\xi):=\lim\limits_{r\ra\infty}f(\xi_r)$
exists in $\partial Y$.
\ep

The bound \eqref{eqnsiaxy} can be interpreted as a large deviation inequality for the random path $f(\xi_t)$ when the ray
$\xi$ is chosen randomly with law $\si$. 
We will apply it later on to various harmonic measures $\si=\si_{x_0,r}$.
A key point in this bound is 
that the constants  involved in  \eqref{eqnsiaxy}
do not depend on the $(M,\nu)$-Frostman measure $\si$.

\begin{proof}[Proof of Proposition \ref{propca}] 
$a)$ Since $f$ has Property $\mc C_{C_1,C_2}$, one has 
\begin{eqnarray*}
H^\nu_\infty(A_{x_0,\al}(n_0))
&\leq &\textstyle
\sum\limits_{n\geq n_0}H^\nu_\infty(A_{x_0,f(x_0),\al,n})\\
&\leq &\textstyle
\sum\limits_{n\geq n_0}
C_1e^{a\nu_\al n}\, C_2^\nu e^{-a\nu n}
\; =\; 
C_{3,\al,\nu} \, e^{-a(\nu-\nu_\al)n_0}\, .
\end{eqnarray*}

$b)$ This follows from $a)$ and Lemma \ref{lemfro}.

$c)$ Letting $n_0$ go to infinity in \eqref{eqnhuainfaxy},
one gets
$H^\nu_\infty(A_{x_0,\al})=0$
for all $\nu>\nu_\al$. Therefore, 
 \eqref{eqnhaudiminf} yields that
$\dim_H(A_{x_0,\al})\leq \nu_\al$.

$d)$ One has $\dim_H(A_{x_0})\leq \inf\limits_{\al>0}\dim_H(A_{x_0,\al}) =0$.

$e)$ Since $f$ is rough Lipschitz, one can assume that 
the parameters $r$ are integers and applies 
the next lemma \ref{lemlimang} to
the sequence $y_n=f(\xi_n)$.
\end{proof}

\bl
\label{lemlimang} Let $Y$ be a Hadamard manifold
with sectional curvature $K_Y\leq -a^2<0$. 
Let $(y_n)_{n\in\m N}$ be a sequence in $Y$ 
such that 
$$
\sup\limits_{n\geq 0}d(y_n,y_{n+1})<\infty
\;\;{\rm and}\;\;
\liminf\limits_{n\ra\infty}\frac1n d(y_0,y_n)>0.
$$

Then,  the limit $y_\infty:=\lim\limits_{n\ra\infty}y_n$ exists in the visual boundary $\partial Y$.
\el

\begin{proof} Choose  $c>0$, $\al>0$ and $n_0\geq 1$ such that 
$$
d(y_n,y_{n+1})\leq c
\;\;\;{\rm and}\;\;\;
d(y_0,y_n)\geq  n\alpha\;\;\; 
\mbox{\rm for all}\;
n\geq n_0\, .
$$
By Lemma \ref{lemang}, the inequality 
$\theta_{y_0}(y_n,y_{n+1})\leq 4e^{ac/2} e^{-a\alpha n}$ holds for any $n\geq n_0$. Since this series converges,
there exists a geodesic ray $\gamma_+\subset Y$  with origin $y_0$ such that $\lim\limits_{n\ra\infty}\theta_{y_0}(y_n,\gamma_+)= 0$. 
\end{proof}

Unlike quasi-isometric maps, coarse embedding may not have boundary values in every direction. 
See Example \ref{exacoaemb} where we could begin with a curve $f_0$ that spirals away in $\mathbb H^2$.

\subsection{The fibers of the boundary maps}
\begin{quote}
We now investigate the fibers of the boundary map $\partial f$ 
of a rough Lipschitz map satisfying property $\mc C$.
\end{quote}

The following proposition \ref{propcb} tells us that
the fibers of the boundary map have zero Hausdorff dimension.

We keep the notations of Section
\ref{secconbou} and introduce  more notations.
Let $X$, $Y$ be pinched Hadamard manifolds
and $f:X\to Y$ be a rough $c$-Lipschitz map satisfying 
Property $\mc C$. 
Let $x_0\in X$ and $\xi\in S_{x_0}$.
Let $B_{x_0}^\xi$ be the set of rays $\eta$ that ``do not go
away from $\xi$ at positive speed'', namely

\begin{equation*}
\label{eqnbxo}
B_{x_0}^\xi:=\{ \eta\in S_{x_0}\mid 
\lim_{n_0\ra\infty}\,\inf\limits_{n,p\geq n_0}
\tfrac{1}{n+p}\,
d(f(\xi_n),f(\eta_p))=0\}\, .
\end{equation*}
One has $B^\xi_{x_0}=\bigcap_{\al>0}B^\xi_{x_0,\al}$, where, for $\al>0$, we set $\be_\al:=\frac{\al^2}{2\al+c}$ and
\begin{equation*}
\label{eqnbxa}
B^\xi_{x_0,\al}:=\{ \eta\in S_{x_0}\mid \lim_{n_0\ra\infty}\,\inf\limits_{n,p\geq n_0}
\tfrac{1}{n+p}\,
d(f(\xi_n),f(\eta_p))< \be_\al\}\, .
\end{equation*}
The choice of this value $\be_\al$ will be explained in Lemma \ref{lempcb}. 

One has 
$B^\xi_{x_0,\al}\subset\bigcap_{n_0\geq 1}B^\xi_{x_0,\al}(n_0)$, where we set for any $n_0\geq 1$~:
\begin{equation*}
B^\xi_{x_0,\al}(n_0):=\{ \eta\in S_{x_0}\mid 
d(f(\xi_n),f(\eta_p))\leq (n\!+\! p)\be_\al\; \;
\mbox{\rm for some $n,p\geq n_0$}
\}\, .
\end{equation*}

\bp
\label{propcb}
Let $X$, $Y$ be pinched Hadamard manifolds 
with sectional curvatures $-b^2\leq K\leq -a^2<0$. 
Let  $c, C_1,C_2>0$ and $f:X\to Y$ be a rough $c$-Lipschitz map with 
Property $\mc C_{C_1,C_2}$. 
Let $\al>0$, $k'=\dim Y$,  $\nu_\al:=\tfrac{bk'\al}{a}$
and $\be_\al:=\frac{\al^2}{2\al+c}$. For  $\nu>\nu_\al$, we 
set
$
C_{4,\al,\nu}:=
\frac{C_1 C_2^\nu}
{(1-e^{-\be_\al bk'})(1-e^{-a(\nu-\nu_\al)})}$.\\
$a)$ 
For  $x_0\in X$, $n_0\geq 1$ and 
$\xi\in S_{x_0}\!\smallsetminus\! A_{x_0,\al}(n_0)$, one has
\begin{equation}
\label{eqnhubinfaxy} 
H^\nu_\infty(B^\xi_{x_0,\al}(n_0))\leq 
C_{4,\al,\nu} \, e^{-a(\nu-\nu_\al)n_0}.
\end{equation}
$b)$ 
For $\xi\in S_{x_0}\!\smallsetminus\! A_{x_0,\al}(n_0)$ and any $(M,\nu)$--Frostman measure $\si$ on $S_{x_0}$, one has
\begin{equation}
\label{eqnsibxy} 
\si(B^\xi_{x_0,\al}(n_0))\leq 
M \, C_{4,\al,\nu} \, e^{-a(\nu-\nu_\al)n_0}.
\end{equation}
$c)$ For
$\xi\in S_{x_0}\!\smallsetminus\! A_{x_0,\al}$, one has
$\dim_H(B^\xi_{x_0,\al})\leq \nu_\al$.\\
$d)$ For 
$\xi\in S_{x_0}\!\smallsetminus\! A_{x_0}$, one has
$\dim_H(B^\xi_{x_0})=0$.\\
$e)$ Assume $n_0\geq \tfrac{4e^{2ac}}{1-e^{-a\be_\al }}$. For  $\xi,\eta \in S_{x_0}\!\smallsetminus\! A_{x_0,\al}(n_0)$
with $\eta\not\in  B^\xi_{x_0,\al}(n_0)$ and
for all $n,p\geq \ell_0:=\frac{4 n_0 c}{\al}$, one has the lower bound for the angle
\begin{equation}
\label{eqnanglowfff}
\theta_{f(x_0)}(f(\xi_n),f(\eta_p))\geq \tfrac12 e^{-2n_0bc}\, .
\end{equation}
$f)$ For  $\xi,\eta \in S_{x_0}\!\smallsetminus\! A_{x_0}$,
with $\eta\not\in  B^\xi_{x_0}$, one has $\partial f(\eta)\neq \partial f(\xi)$.
\ep

We begin with a technical covering lemma.

\bl
\label{lempcb} Notations as in Proposition \ref{propcb}.
For 
$\xi\in S_{x_0}$ and $p\geq n_0$,~let 
\begin{equation*}
\label{eqnbxp}
B^\xi_{x_0,\al,p}:=\{\eta\in S_{x_0}\, \mid\, d(f(\xi_n),f(\eta_p))\leq (n\! +\! p)\be_\al\; \;
\mbox{\rm for some $n\geq n_0$}\}\, .
\end{equation*}
If $\xi$ is not in $A_{x_0,\al}(n_0)$, the set
$B^\xi_{x_0,\al,p}$ can be covered by at most 
$\displaystyle\frac{C_1\, e^{\al bk'p}}{1-e^{-\be_{\al} bk'}}$ 
balls of radius $C_2\, e^{-ap}$.
\el

The value $\be_\al$ has been chosen in order to obtain 
the same exponential growth $\al b k'$ for these coverings
as in Definition \ref{defpcc}.

\begin{proof}[Proof of Lemma \ref{lempcb}]
Using the notation \eqref{eqnp0}, we have the equality
$$
\textstyle 
B^\xi_{x_0,\al,p}=\bigcup\limits_{n\geq n_0} A_{x_0,f(\xi_n),\be_\al,n+p}.
$$ 
The key point is that, 
since $f$ is rough $c$-Lipschitz and $\xi\not\in A_{x_0,\al}(n_0)$, this union is a finite union. Indeed, assume that an integer $n\geq n_0$ satisfies
$$
d(f(\xi_n),f(\eta_p))\leq (n\! +\! p)\be_\al
$$
for some $\eta\in S_{x_0}$. Since $\; d(f(x_0),f(\xi_n))\geq n\al\; $ and 
$\; d(f(x_0),f(\eta_p))\leq p\, c\, ,$ 
one must have
$$
n\al - p\, c\leq (n\! +\! p)\be_\al.
$$ 
By our choice of $\be_\al$, this inequality is equivalent to
$$
(n\! +\! p)\be_\al\leq p\,\al .
$$
Therefore, using Definition \ref{defpcc}, one can cover the set 
$B^\xi_{x_0,\al,p}$ by at most 
$C_1\sum_n e^{(n+p)bk'\be_\al}$ balls of radius
$C_2e^{-ap}$, where this sum runs over the integers $n\geq n_0$ such that $(n\! +\! p)\be_\al\leq p\al .$
Computing this sum, one deduces that the set
$B^\xi_{x_0,\al,p}$ can be covered by at most 
$\displaystyle\frac{C_1\, e^{\al pbk'}}{1-e^{-\be_{\al} bk'}}$ 
balls of radius $C_2e^{-ap}$.
\end{proof}

\begin{proof}[Proof of Proposition \ref{propcb}] 
$a)$ Using Lemma \ref{lempcb}, one has
\begin{eqnarray*}
H^\nu_\infty(B^\xi_{x_0,\al}(n_0))
&\leq &\textstyle
\sum\limits_{p\geq n_0}H^\nu_\infty(B^\xi_{x_0,\al,p})\\
&\leq &{\textstyle
\sum\limits_{p\geq n_0}}
\frac{C_1e^{a\nu_\al p}}{1-e^{-\be_\al bk'}}\, C_2^\nu e^{-a\nu p}
\; =\; 
C_{4,\al,\nu} \, e^{-a(\nu-\nu_\al)n_0}\, .
\end{eqnarray*}

$b)$ This follows from $a)$ and Lemma \ref{lemfro}.

$c)$ Letting $n_0$ go to infinity in \eqref{eqnhubinfaxy}
one gets
$H^\nu_\infty(B^\xi_{x_0,\al})=0$,
for all $\nu>\nu_\al$. Therefore, 
using \eqref{eqnhaudiminf}, it follows that
$\dim_H(B^\xi_{x_0,\al})\leq \nu_\al$.

$d)$ One has $\dim_H(B^\xi_{x_0})\leq \inf\limits_{\al>0}\dim_H(B^\xi_{x_0,\al}) =0$.

$e)$ This is a consequence of
the following lemma \ref{lemangwea} applied to
the sequences $y_n=f(\xi_n)$ and $z_p=f(\eta_p)$.

$f)$ This follows from $e)$.
\end{proof}

\subsection{Two sequences going away from one another}
\bq
The aim of this section is to prove the following lemma
which provides, in a pinched Hadamard manifold, 
a lower bound for the angle 
between points in two sequences 
with bounded speed that ``go away from one another 
at positive speed''. 
\eq

\bl
\label{lemangwea}
Let $Y$ be a Hadamard manifold with sectional curvature $-b^2\leq K_Y\leq -a^2<0$. 
Let $c\geq \al\geq \be>0$  and  
$n_0\geq \tfrac{4e^{2ac}}{1-e^{-a\be }}$. 
Let $(y_n)_{n\in\m N}$ and $(z_p)_{p\in\m N}$ be two sequences of points in $Y$ with $y_0=z_0$, 
such that
\begin{eqnarray}
d(y_n,y_{n+1})\leq c
\;\;{\rm and}\;\; d(z_p, z_{p+1})\leq c
\;\;\mbox{\rm  for any integers}\;\; n,p\geq 0, &&\label{eqnlip}\\
d(y_0,y_n)\!\geq\!  n\alpha ,\;
d(y_0,z_p)\!\geq\! p\alpha 
\;{\rm and}\; 
d(y_n,z_p)\!\geq\!  (n\! +\! p)\beta
\; {\rm for}\;
n,p\geq n_0.&&
\label{eqnlin}
\end{eqnarray} 
Then,  for any integer $n,p\geq \ell_0:=\frac{4 n_0 c}{\alpha}$, one has
\begin{equation}
\label{eqnanglow}
\theta_{y_0}(y_n,z_p)\geq \tfrac12 e^{-2n_0bc}\, .
\end{equation}
\el

We will need two geometric lemmas.

We know that the orthogonal projection from a Hadamard manifold on a geodesic is a $1$-Lipschitz map. The following lemma gives us a more precise information when the curvature is bounded from above.

\bl 
\label{lemjacfie} 
Let $Y$ be a Hadamard manifold 
with sectional curvature $K_Y\leq -a^2 < 0$. 
Let $\gamma\subset Y$ be a geodesic. 
Then, the orthogonal projection $\pi : Y\to\gamma $ 
is smooth and, for  $y\in Y$, 
the norm of its differential satisfies
$$
\| D_y\pi\|\leq  \frac{1}{\cosh (a\, d(y,\gamma))}
\leq 2\,e^{-a\, d(y,\gamma)}\, .
$$
\el

\begin{proof}[Proof of Lemma \ref{lemjacfie}] 
The proof relies on a Jacobi field estimate (see \cite{GHL04}). 

Let $y\in Y\setminus\gamma$, let $\bar y=\pi (y)\in\gamma$ and  $\ell =d(y,\gamma )=d(y,\bar y$. Denote by $c:s\in[0,\ell ]\to c(s)\in Y$ the unit-speed parametrization of the geodesic segment $[\bar y,y]$ with $c(0)=\bar y$ and $c(\ell)=y$.

Let $v\in T_y Y$. We want to bound the ratio $\|D_y\pi (v)\| /\| v\|$. We may thus assume that $v$ is orthogonal to $ {\rm Ker}D_y\pi$, i.e. that $v$ is orthogonal to the geodesic $c$ at the point $y$. 

Choose a smooth curve $t\to y(t)\in Y$ with $y(0)=y$ and $y'(0)=v$, and let $\bar y(t)=\pi (y(t))\in\gamma$. 
We can assume that, for all $t$, one has
$d(y(t),\bar y(t))=\ell$. For each parameter $t$, introduce the constant-speed geodesic $c_t: [0,\ell] \to Y$ such that $c_t(0)=\bar y(t)$ and $c_t(\ell)=y(t)$. By construction, each vector $u(t):={\frac d{ds}c_t(s)}_{|s=0}\in T_{\bar y(t)}Y$ is normal to $\gamma$ at the point $\bar y(t)$.

The map $(s,t)\to c_t(s)$ is a variation of geodesics, so that $J:s\in [0,\ell]\to {\frac d{dt}c_t(s)}_{|t=0}\in T_{c(s)}Y$ is a Jacobi field along the geodesic $c$. We have $J(0)=D_y\pi (v)$ and $J(\ell)=v$. Since both $J(0)$ and  $J(\ell )$ are normal to $c$,  it follows that $J$ is a normal Jacobi field.  Since $\gamma$ is a geodesic and each $u(t)$ is normal to $\gamma$, we infer from
the equality $J'(0)=u'(0)$
that $J'(0)$ is normal to $\gamma$, i.e. orthogonal to $J(0)$. 
The Jacobi field equation $J''+R(c',J)c'=0$ and the hypothesis on the curvature now yield
$$
(\| J\|^2)''=2 \|J'\|^2 -2 R(c',J,c',J)\geq  2 (\|J\|')^2+ 2a^2 \| J\|^2
$$
and therefore
$$
\| J\|''\geq  a^2 \| J\|\, . 
$$
Since $\| J\|'(0)=\frac{\langle J(0),J'(0)\rangle}{\|J(0)\|} =0$, one deduces that
$\| J(t)\|\geq \|J(0)\|\cosh ( at)$,
for all  $t\geq 0$. In particular, one has 
$\|D_y\pi (v)\|\leq \| v\|/\cosh ( a\ell)$.
\end{proof}

The second lemma is an easy angle comparison lemma.

\bl
\label{lemangcom} 
Let $Y$ be a Hadamard manifold with sectional curvature $-b^2\leq K_Y\leq 0$. 
Let $\gamma\subset Y$ be a geodesic, $y_0\in \ga$, $y\in Y$ and $\bar y=\pi (y)$ be the projection of $y$ on  $\gamma$. Assume that  $d(y_0,\bar y)\leq R$ and $d(\bar y,y)\geq R$, then one has
 $\theta_{y_0}(y,\bar y)\geq \tfrac12e^{-bR}$.
\el

\begin{proof}[Proof of Lemma \ref{lemangcom}]  The angles of a  triangle in $\m H^2(-b^2)$ with same side-lengths are smaller than the angles of the triangle $(y_0y\bar y)$. It follows that   $\theta_{y_0}(y,\bar y)\geq\varphi$, where   
$\ph$ 
is the angle of an isosceles right triangle in 
$\m H^2(-b^2)$ with
adjacent sides $R$, which is  $\varphi=\arctan(\tfrac{1}{\cosh(bR)})\geq \tfrac12e^{-bR}$.
\end{proof}

\begin{proof}[Proof of Lemma \ref{lemangwea}] 
Let $\ga_+$ be a geodesic ray starting from $y_0=z_0$.
Denote by  $\pi : Y\to\gamma$ the orthogonal projection on the geodesic $\gamma$ that contains $\gamma_+$. 
Identify $\gamma\sim\m R$ so that $\gamma_+\sim [0,\infty[$.
Introduce, for $n,p\in\m N$, the points $\bar y_n=\pi (y_n)$ and $\bar z_p=\pi (z_p)$, and the sub-intervals $I_n=[\bar y_n,\bar y_{n+1}]$ and $J_p=[\bar z_p,\bar z_{p+1}]$ of $\gamma$. 

Let $R:=2n_0c$. We claim that 
\begin{equation}
\label{eqnminbynbzp}
\min(\bar y_N,\bar z_P)\leq R
\;\; \mbox{\rm for all $N, P\geq 0$.}
\end{equation}
According to \eqref{eqnlip}, one has the bound $\max(\bar y_{n_0}, \bar z_{n_0})\leq n_0 c$. Hence it is enough to check that the interval $\mc I:=[\bar y_{n_0},\bar y_N]\cap [\bar z_{n_0},\bar z_P]$ has length at most 
$|\mc I|\leq n_0c$.

Let $q\in \mc I$. This point lies in some non-empty interval
$I_n\cap J_p$ with $n,p\geq n_0$. 
Since the projection $\pi$ is $1$-Lipschitz, using  \eqref{eqnlip} again yields $d(\bar y_n,\bar z_p)\leq 2c$.
According to \eqref{eqnlin} one has $d(y_n,z_p)\geq \beta (n+p)$ so that
$$
\hbox{either\;\; $d(y_n,\bar y_n)\geq n\beta -c$
\; or\;\; 
$d(z_p,\bar z_p)\geq p\beta -c$}\, ,
$$
and Lemma \ref{lemjacfie} now provides a bound for the length of one of the intervals $I_n$ or $J_p$~:
$$
\hbox{either\;\; $|I_n|\leq 2c\,e^{2ac-na\beta }$
\; or\;\; $|J_p|\leq 2c\,e^{2ac-pa\beta }$}\, .
$$

It follows that the interval $\mc I$ is bounded by
\begin{eqnarray*}
|\mc I|&\leq& 
\textstyle
\sum\limits_{n\geq n_0}2c\,e^{2ac-na\beta }
+\sum\limits_{p\geq n_0}2c\,e^{2ac-pa\beta }\\
&\leq& \frac{4c\, e^{2ac}}{1-e^{-a\beta}}
\,e^{-n_0a\beta }\leq n_0 c\, .
\end{eqnarray*}
This proves our claim \eqref{eqnminbynbzp}.

Now, let $n, p\geq \ell_0:= \tfrac{4n_0c}{\al}$ 
so that, by \eqref{eqnlin}, 
one has $d(y_0,y_n)\geq 2R$ and $d(y_0, z_p)\geq 2R$.
The claim \eqref{eqnminbynbzp} tells us that
$$
\hbox{either\;\; $d(y_0,\bar y_n)\leq R$
\; or\;\; $d(y_0,\bar z_p)\leq R$}\, .
$$
Hence by Lemma \ref{lemangcom}, one has
$$
\hbox{either\;\; $\theta_{y_0}(y_n,\ga_+)\geq \tfrac12e^{-bR}$
\; or\;\; $\theta_{y_0}(z_p,\ga_+)\geq \tfrac12e^{-bR}$}\, .
$$ 
Since this is true for any ray $\ga_+$ based at
$y_0$, one gets $\theta_{y_0}(y_n,z_p)\geq \tfrac12e^{-bR}$.
\end{proof}

\begin{proof}[Proof of Theorem \ref{thmboumap}]
Point $a)$ follows from Propositions \ref{propca}.$d$ and 
\ref{propca}.$e$.

Point $b)$ follows from Propositions \ref{propcb}.$d$ and 
\ref{propcb}.$f$.  
\end{proof}

\br
\label{remboumap}
It follows from the proof that Theorem \ref{thmboumap}
also holds true for any rough Lipschitz map $f:X\ra Y$ between pinched Hadamard manifolds that satisfies property $\mc C$.
\er

\section{Beyond quasi-isometric maps}
\label{secweacoahar}

The aim of this Chapter \ref{secweacoahar} 
is the following
extension of Theorem \ref{thdfxhx} 
to all
weakly coarse embeddings $f$, and in particular to all
coarse embeddings $f$ (see Definitions
\ref{defcoaemb} and \ref{defweacoa}).

\subsection{Weakly coarse embeddings and harmonic maps}

\bt
\label{thmharcoa}
Every weakly coarse embedding $f:X\to Y$ between 
two pinched Hadamard manifolds is within  bounded distance 
from a unique harmonic map $h: X\to Y$.
\et

Indeed we will prove a  more general proposition
using Definition \ref{defpcc}.

\bp
\label{proharpcc}
Every rough Lipschitz map $f:X\to Y$ satisfying property $\mc C$ 
between two pinched Hadamard manifolds is within  bounded distance 
from a unique harmonic map $h: X\to Y$.
\ep

The main new ingredients in the proof 
are the construction and the properties 
of a boundary map of $f$.
These new ingredients which do not involve
harmonic maps were explained in Chapter \ref{secbouweacoa}. 
We now explain how to adapt the proof of Theorem  \ref{thdfxhx}
using these new ingredients.

\espace

\subsection{Rough Lipschitz harmonic maps}
\label{secrouliphar}
\bq
We first want to point out that Theorem \ref{thmharcoa} 
can not be extended to all rough Lipschitz maps. 
\eq

\bex
\label{exaharlip}
There exists an injective Lipschitz map $f:\m H^2\to\m H^2$ 
from the hyperbolic plane to itself, that extends continuously 
to the visual boundary as the identity map, 
and which is not within  bounded distance from any harmonic map.
\eex

\begin{proof}
We will consider a map $f:\m H^2\to\m H^2$ that  commutes 
to a parabolic subgroup of ${\rm Isom}(\m H^2)$. 
Let us work in the upper half-plane model. The map $f$ is defined by 
$$
f(u,v)=(u,v+v^2) \quad u\in\m R\, , v>0\,,
$$
so that $f\circ s_t=s_t\circ f$ where $s_t(u,v)=(t-u,v)$ 
for any $t\in\m R$. Observe that $f$ extends continuously 
to the visual compactification of $\m H^2$ by the identity, 
and that $f$ is 2-Lipschitz. 

Assume by contradiction that there exists a harmonic map 
$h:\m H^2\to\m H^2$ within  bounded distance from $f$. 
\vs

{\bf First case:} the map $h$ is unique.
In this case $h$ also commutes to the isometries $s_t$, 
so that there exists a continuous function $g:[0,\infty]\to[0,\infty]$ 
such that
$$
h(u,v)=(u,g(v))\quad u\in\m R\, , v>0\,,
$$
and with $g(0)=0$, $g(\infty)=\infty$. 
Saying that $h$ is harmonic is equivalent to requiring 
the function $g$ to satisfy the differential equation
$$
g\, g''=(g')^2-1\, .
$$
It follows that the harmonic map $h$ coincides 
with one of the maps $h_a:\m H^2\to\m H^2$ defined by
$$
h_a(u,v)=(u,\frac 1a\, \sinh (av))
$$
for some constant $a\geq 0$. Observe that none of the maps $h_a$ 
is within  bounded distance from $f$, hence the contradiction.
\vs

{\bf Second case:} the map $h$ is not unique.
Let $h_0$, $h_1$ be two distinct harmonic maps within bounded distance from $f$. We want again to find a contradiction. 
We will use arguments similar to those in Chapter \ref{secunihar}.  
Let $x_0:=(0,1)\in \m H^2$. 
We choose a sequence of points $x_n$ in $\m H^2$ for which the distances
\begin{eqnarray*}
d(h_0(x_n),h_1(x_n))
\;\;\mbox{\rm converge to}\;\; 
\delta:=\sup_{x\in \m H^2}d(h_0(x),h_1(x))>0\, 
\end{eqnarray*}
and we set $y_n:=f(x_n)$. 
Let $\ph_n$ and $\psi_n$ be the isometries of $\m H^2$ 
fixing the point $\infty\in \partial \m H^2$ 
and such that $\ph_n(x_0)=x_n$ and $\psi_n(x_0)=y_n$.
After extraction, the sequence of maps 
$\psi_n^{-1}\circ f\circ \ph_n$ converges 
to one of the maps $f_\beta:\m H^2\ra \m H^2$ with $\beta\in [0,\infty]$ 
where 
$$
f_{\beta}: (u,v)\mapsto 
(\tfrac{u}{1+\be},
\tfrac{ v+\be v^2}{1+\be})
\;\;\;\;{\rm when}\;\;
0\leq \be <\infty\, 
$$
$$
f_{\infty}: (u,v)\mapsto 
(0,v^2)
\;\;\;\;{\rm when}\;\;
\be =\infty\, .
$$
For $i=0$ and $1$, the sequence of harmonic maps 
$h_{i,n}:= \psi_n^{-1}\circ h_i\circ \ph_n$ converges, 
after extraction,
to a harmonic map $h_{i,\infty}:\m H^2\ra \m H^2$
within bounded distance to $f_\beta$.
The subharmonic function $x\mapsto d(h_{0,\infty}(x),h_{1,\infty}(x))$ 
achieves its maximum value at $x=x_0$, 
hence is a constant function equal to $\de$. 
Therefore, by Corollary \ref{cortwohar}, the harmonic maps $h_{0,\infty}$ and $h_{1,\infty}$ take their values in the same geodesic
$\Gamma$.
This forces the equality $\beta=\infty$ and the geodesic $\Gamma$ is the image of $f_\infty$.
Now we write 
$$
f_\infty(u,v)=(0, e^{2F_\infty(u,v)})
\;\;{\rm and}\;\; 
h_{0,\infty}(u,v)=(0, e^{2H_{0,\infty}(u,v)}),
$$
where $F_\infty(u,v)=\log v$ and where $H_{0,\infty}$ is a harmonic function.

The function $G_\infty:=F_\infty-H_{0,\infty}$ is then a bounded function 
on $\m H^2$ such that $\Delta G_\infty=1$. 
Such a function $G_\infty$ does not exist.
Indeed the function $G:x\mapsto 2\log(\cosh(d(x_0,x)/2))$ also satisfies $\Delta G =1$ 
and the function $G-G_\infty$ would be proper and harmonic, contradicting the maximum principle.
\end{proof}

\subsection{An overview of the proof of Proposition \ref{proharpcc}}
\label{secoverview}

\begin{proof}[Proof of Proposition \ref{proharpcc}]
The strategy is the same as for Theorem \ref{thdfxhx}: 

{\sl Step 1~: smoothing $f$ out. } 
By Proposition \ref{prosmoothquasi}  
there exists a smooth map $\widetilde f:X\to Y$ 
within  bounded distance from $f$ 
and whose first and second covariant derivatives are bounded on $X$. This function  $\widetilde f$ is Lipschitz 
and still satisfies property $\mc C$. Hence we can assume that $f=\widetilde{f}$.

{\sl Step 2~: solving a bounded Dirichlet problem. }  
We fix an origin $O\in X$. For any radius $R$ 
we consider the unique harmonic map $h_R:B(O,R)\to Y$ 
satisfying the Dirichlet condition $h_R=  f$ 
on the sphere $S(O,R)$. 

{\sl Step 3~: estimating the distance $d(h_R,f)$. } 
We will check in Section 
\ref{secinteriorbis}:
\bp
\label{proexistharmbis}
There exists a constant $\rho\geq 1$ such that, for any $R\geq 1$, one has
$d(h_{_R},f)\leq \rho$.
\ep

{\sl Step 4~: letting the sequence $h_R$ converge to $h$. } 
We prove this convergence 
as in Section \ref{secexistharm}.
\end{proof}

The proofs of Steps 1, 2 and 4, as well as the proof of uniqueness,  
require only minor modifications from the ones for quasi-isometric maps. 
Thus, the remaining of this paper will be devoted to the proof of Step 3.

\subsection{Interior estimate for rough Lipschitz}
\label{secinteriorbis}

In this section we complete the proof of Proposition \ref{proexistharmbis} whose structure is exactly the same as the proof of Proposition \ref{proexistharm}.
We will just repeat quickly the arguments
of Section \ref{secinterior}
pointing out the changes in the choice
of the many constants involved in the proof.

\subsubsection{Strategy}  
\label{secstrexibis}
Let $X$ and $Y$ be two Hadamard manifolds whose curvatures are pinched 
$-b^2\leq K\leq -a^2<0$. 
Let $k=\dim X$ and $k' =\dim Y$. 
We fix two constants $M, N>0$ as in Proposition \ref{prosixrcxthe}.
We set $\al=\frac{a}{2bk'N}$ so that, with the notation of 
Propositions \ref{propca} and \ref{propcb}, one has $\nu_\al =\frac{1}{2N}$.
We set $\nu=2\nu_\al=\frac{1}{N}$.

We start with a $\mc C^\infty$ 
 Lipschitz map $f:X\ra Y$ whose first and second  covariant derivatives are bounded.
We fix constants $c, C_1, C_2 \geq 1$ such that, 
$f$ satisfies property $\mc C_{C_1,C_2}$ as in Definition \ref{defpcc} 
and such that, for all $x$  in $X$,
one has 
\begin{equation}
\label{eqnquasiiso3bis}
\| Df(x)\|\leq c
\;\;{\rm , }\;\;\;
\| D^2f(x)\|\leq bc^2
\; .
\end{equation}
We let $C_3=C_{3,\al,\nu}\leq C_4=C_{4,\al,\nu}$ 
be the two constants as in Proposition \ref{propca} and \ref{propcb}. 
$$
C_{3} = \frac{C_1C_2^{1/N}}{1-e^{-a/(2N)}}
\; ,\;\;
C_{4} = \frac{C_1C_2^{1/N}}{(1-e^{-\beta bk'})(1-e^{-a/(2N)})}
\;{\rm where}\;\; 
\be=\frac{\al^2}{2\al +c}.
$$

{\bf Choosing $\ell_0$ very large.}
We fix a point $O$ in $X$. 
We introduce a fixed integer radius $\ell_0$ depending only on 
$a$, $b$, $k$, $k'$, $c$, $C_1$ and $C_2$.
This integer $\ell_0\geq 1$ is only required to satisfy 
the three inequalities  \eqref{eqnell0bis}, \eqref{eqnell1bis} and \eqref{eqnell2bis}:
\begin{equation}
\label{eqnell0bis}
b\ell_0 > 1,
\end{equation}
\begin{equation}
\label{eqnell1bis}
\ell_0 > \, 4n_0c/ \alpha
\;\;\mbox{\rm where $n_0\geq \tfrac{4e^{2ac}}{1-e^{-a\be }}$ is chosen with}\;\; 
MC_4 e^{-an_0\al}\leq \frac{\al}{8c},
\end{equation}
\begin{equation}
\label{eqnell2bis}
16\,e^{-\frac{a\al\ell_0}{4}} <
\theta_0
\;\;{\rm where}\;\; 
\th_0:=e^{-2n_0bc}/2\, .
\end{equation}

{\bf Choosing $\rho$ very large.} 
For $R>0$, let $h_{_R}:B(O,R)\ra Y$ be
the harmonic $\mc C^\infty$ map whose restriction to the sphere
$\partial B(O,R)$ is equal to $f$.
We let 
$\rho:=\sup\limits_{x\in B(O,R)}d(h_{_R}(x),f(x))\; .$
We argue by contradiction. If this supremum $\rho$ is not uniformly bounded,
we can  fix a radius $R$ such that 
$\rho$ satisfies the three inequalities \eqref{eqnrho1}, \eqref{eqnrho2} and \eqref{eqnrho3} that we rewrite below:
\begin{equation}
\label{eqnrho1bis}
a\rho > 8kbc^2\ell_0 \, ,
\end{equation}
\begin{equation}
\label{eqnrho2bis}
\frac{2^7(a\rho)^2}{\sinh(a\rho/2)} <  \theta_0\, .
\end{equation} 
\begin{equation}
\label{eqnrho3bis}
\rho >
4c\ell_0 M\,(2^{10}e^{b\ell_0}k)^{N}.
\end{equation}

We denote by $x$ a point of $B(O,R)$ where the supremum is achieved:
$
d(h_{_R}(x),f(x))=\rho\, .
$
According to the boundary estimate  \eqref{eqnboundary}, 
one has, using  \eqref{eqnrho1bis},
$$
d(x,\partial B(O,R))\geq 
\frac{a\rho}{3kbc^2}\geq 2\ell_0\, .
$$

{\bf Getting a contradiction.}
We will focus on the restrictions of both  maps $f$ and $h_{_R}$ 
to this ball $B(x,\ell_0)$.
We introduce the point $y:=f(x)$. 
For $\xi$ on the unit tangent sphere $S_x$, we will 
analyze the triangle inequality:
\begin{equation}
\label{eqnthththbis}
\th_y(f(\xi_{\ell_0}),h_{_R}(x))\leq \th_y(f(\xi_{\ell_0}),h_{_R}(\xi_{\ell_0}))+\th_y(h_{_R}(\xi_{\ell_0}),h_{_R}(x)),
\end{equation}
and prove that on a subset 
$U_{\ell_0}\!\smallsetminus\! A_{x,\al}(n_0)$ 
of the sphere, 
each term on the right-hand side is small
(Lemmas \ref{lemI1bis} and \ref{lemI2bis}) 
while the left-hand side is not always that small
(Lemma \ref{lemI0bis}),
giving rise to the contradiction. 

\begin{Def}
\label{defurvrwrbis}
Let $U_{\ell_0}=
\{\xi\in S_x\mid d(y,h_{_R}(\xi_{\ell_0}))\geq \rho-\ell_0\al/2\, \}\, .$
\end{Def}

\subsubsection{Measure estimate}  
\label{secmeasurebis}

\bl
\label{lemrhxbis}
For $\xi$ in $S_x$, one has 
$d(y,h_{_R}(\xi_{\ell_0}))\leq \rho+c\,\ell_0.$
\el

\begin{proof} This is Lemma \ref{lemrhx}.
\end{proof}

\bl
\label{lemdhxbis}
For $\xi$ in $S_x$, and $r\leq \ell_0$, one has 
$\|Dh_{_R}(\xi_r)\|\leq 2^8 kb\rho.$
\el

\begin{proof} This is Lemma \ref{lemdhx}. 
It uses \eqref{eqnell0bis}, \eqref{eqnrho1bis} and Lemma \ref{lemrhxbis}.
\end{proof}

\begin{Lem}
\label{lemsiwrbis}
Let $\si=\si_{x,\ell_0}$ be the harmonic measure on the sphere $S_x\simeq S(x,\ell_0)$
for the center point $x$.
Then one has
$\si(U_{\ell_0})\geq \frac{\al}{3\,c}\; .$
\end{Lem}

\begin{proof} Same as Lemma \ref{lemsiwr}.
\end{proof}

\subsubsection{Estimating the angles}
\label{secboundi1bis}

\bl
\label{lemI1bis}
\mbox{}\!\!\! For  $\xi$ in $U_{\ell_0}
\!\!\smallsetminus\! A_{x,\al}(n_0)$, one has
$\theta_y(\! f(\xi_{\ell_0}),h_{_R}\!(\xi_{\ell_0}\!)\!)\!\leq\!  4
e^{\! \frac{-a\al\ell_0}{4}}\!\!<\! \frac{\th_0}{4}.$
\el

\begin{proof} Same as Lemma \ref{lemI1}, using \eqref{eqnell2bis}.
\end{proof}

\bl
\label{lemI2bis}
For $\xi$ in  $S_x$, one has
$\theta_y(h_{_R}(\xi_{\ell_0}),h_{_R}(x))\leq 
\frac{2^5\,(a\rho)^2}{{\sinh}(a\rho/2)}< \frac{\th_0}{4}.$
\el

\begin{proof} 
Same as Lemma \ref{lemI2}, relying on  Lemma \ref{lemI3bis} and using both \eqref{eqnrho1bis} and
\eqref{eqnrho2bis}.
\end{proof}

\bl
\label{lemI3bis}
For all $\xi$ in  $S_x$ and $r\leq \ell_0$, one has
$d(y,h_{_R}(\xi_r))\geq \rho/2.$
\el

\begin{proof} Same as  Lemma \ref{lemI3}, using Lemma \ref{lemdhxbis} and  Condition \eqref{eqnrho3bis}. 
\end{proof}

\bl
\label{lemI0bis} 
There exist $\xi$, $\eta$ in 
$U_{\ell_0}\!\smallsetminus\! A_{x,\al}(n_0)$ with  
$\th_y(f(\xi_{\ell_0}),f(\eta_{\ell_0}))\geq \th_0\, .$
\el

\begin{proof}[Proof of Lemma \ref{lemI0bis}]
Let $\si_0:=\frac{\al}{4c}$. 
According to Lemma \ref{lemsiwrbis}, one has
$$
\si( U_{\ell_0})>\si_0 >0.
$$
By the definition of $M,N$ in Proposition \ref{prosixrcxthe}, one can apply 
Proposition \ref{propca}.$b$  to
the harmonic measure $\si=\si_{x,\ell_0}$
and one gets, using \eqref{eqnell1bis},  that
$$
\si(A_{x,\al}(n_0))\leq MC_3e^{-\frac{an_0}{2N}}
\leq \frac{\al}{8c}=\si_0/2.
$$
Therefore, there exists an element 
$\xi\in U_{\ell_0}\!\smallsetminus\! A_{x,\al}(\ell_0)$.
Applying 
Proposition \ref{propcb}.$b$  to
the harmonic measure $\si=\si_{x,\ell_0}$
 one gets, using \eqref{eqnell1bis} again, that
$$
\si(B^\xi_{x,\al}(n_0))\leq MC_4e^{-\frac{an_0}{2N}}
\leq \frac{\al}{8c}=\si_0/2.
$$
Therefore, there exists an element 
$\eta\in U_{\ell_0}\!\smallsetminus\! 
(A_{x,\al}(n_0)\cup B^\xi_{x,\al}(n_0))$.
This element satisfies
\begin{eqnarray*}
\theta_y(f(\xi_{\ell_0}),f(\eta_{\ell_0})
&\geq & 
e^{-2n_0bc}/2
= \th_0
\end{eqnarray*}
because of \eqref{eqnell1bis}, \eqref{eqnell2bis}
and Proposition \ref{propcb}.$e$. 
\end{proof}

\begin{proof}[End of the proof of Proposition \ref{proexistharmbis}] 
Let $\xi$, $\eta$ be  two vectors of
$U_{\ell_0}
\!\smallsetminus\! A_{x,\al}(n_0)$ given by Lemma \ref{lemI0bis}.
Applying  Lemmas \ref{lemI1bis} and \ref{lemI2bis} to  $\xi$ and $\eta$, one gets 
$$
\th_y(f(\xi_{\ell_0}),f(\eta_{\ell_0}))\leq 
\th_y(f(\xi_{\ell_0}),h_{_R}(x))+
\th_y(h_{_R}(x),f(\eta_{\ell_0})) <
\th_0,
$$
which contradicts Lemma \ref{lemI0bis}. 
\end{proof}

{\footnotesize
The first version of this paper 
containing Chapters 1 to 5 was released in February 2017.
In this second version, Chapters 6 and 7 were added.
In between, two related preprints 
were posted in the ArXiv: \cite{PanSou17} and
\cite{SidWen18}.}



\vspace{2em}

{\small\noindent Y. Benoist  \& D. Hulin,\;
CNRS \& Universit\'e Paris-Sud, 
Orsay 91405 France\\
\noindent yves.benoist@u-psud.fr \;\&\; dominique.hulin@u-psud.fr}

\end{document}